\documentclass{article}

\usepackage[colorlinks, citecolor=blue]{hyperref}
\hypersetup{linkcolor=blue,
anchorcolor=blue,
citecolor=blue }

\usepackage{amscd}
\usepackage{amsmath}
\usepackage{amssymb}
\usepackage{amsthm}
\usepackage{dsfont}
\usepackage{mathrsfs}
\usepackage{epsfig}
\usepackage{graphicx,graphics}
\usepackage{multirow}
\usepackage{pstricks}
\usepackage{url}
\usepackage{float}
\usepackage[body={16cm,24cm}, top=2cm]{geometry}
\usepackage{bm}
\usepackage[linewidth=1pt]{mdframed}

\usepackage{enumerate}
\usepackage{cases}

\usepackage[numbers,sort&compress]{natbib}

\allowdisplaybreaks

\newtheorem{theorem}{Theorem}[section]

\newtheorem{definition}[theorem]{Definition}
\newtheorem{remark}[theorem]{Remark}

\newtheorem{assumption}[theorem]{Assumption}
\newtheorem{lemma}[theorem]{Lemma}

\def \st {\mathrm{ s.t. }}
\def \dom {\mathrm{ dom\ }}
\def \dis {\mathrm{ dist }}

\begin{document}
\title {Effective two-stage image segmentation: a new non-Lipschitz decomposition approach with convergent algorithm}

\author{Xueyan Guo{$^a$}   \qquad Yunhua Xue{$^a$}
	\qquad Chunlin Wu{$^{a}$}\footnote{Corresponding author. Email: wucl@nankai.edu.cn}  \\
	\small  $^a$ School of Mathematical Sciences, Nankai University, Tianjin, China}
\date{}
\maketitle

\textbf{Abstract}
Image segmentation is an important median level vision topic. Accurate and efficient multiphase segmentation for images with intensity inhomogeneity is still a great challenge. We present a new two-stage multiphase segmentation method trying to tackle this, where the key is to compute an inhomogeneity-free approximate image. For this, we propose to use a new non-Lipschitz variational decomposition model in the first stage. The minimization problem is solved by an iterative support shrinking algorithm, with a global convergence guarantee and a lower bound theory of the image gradient of the iterative sequence. The latter indicates that the generated approximate image (inhomogeneity-corrected component) is with very neat edges and suitable for the following thresholding operation. In the second stage, the segmentation is done by applying a widely-used simple thresholding technique to the piecewise constant approximation. Numerical experiments indicate good convergence properties and effectiveness of our method in multiphase segmentation for either clean or noisy homogeneous and inhomogeneous images. Both visual and quantitative comparisons with some state-of-the-art approaches demonstrate the performance advantages of our non-Lipschitz based method.

\vspace{0.5cm}

\section{Introduction}
\label{intro}
As an important median level vision problem, image segmentation plays a central role in bridging image restoration and various high level applications. It aims to partition an image into several regions so that there are uniform characteristics in each region. Pixel intensity is the most basic and widely used feature for segmentation. There are various approaches for intensity based segmentation and in this paper we focus on the energy minimization methods.

The main challenge of segmentation in the pixel intensity feature space is the influence of intensity inhomogeneity and the need for multiphase segmentation. Although conventional approaches work quite well in certain cases, most of them are still difficult to efficiently handle simultaneously these two problems well with theoretical convergence guarantee. The recent two-stage segmentation approaches \cite{cai2013two,duan2015l_,chang2017new,Chan2018Convex,li2016a,cai2018linkage} have some good characteristics and partially overcome those drawbacks. Meanwhile there are still some shortcomings and further researches are needed. See the related work.

In this paper, we propose a novel two-stage image segmentation method using a continuous but non-Lipschitz decomposition model. Experiments and comparisons to some typical state-of-the-art techniques demonstrate the good performance of our method. The good results benefit from the clean piecewise constant approximate images after the inhomogeneity removal in the first stage. Moreover, our iterative algorithm is proved to be convergent. These advantages are due to the non-Lipschitz regularization we adopted, which has better edge preservation property than total variation and allows to be minimized efficiently with global convergence guarantee.

\subsection{Related work}
Energy minimization methods for image segmentation solve some predefined variational models. There are two types of such methods, i.e., edge-based and region-based. Roughly speaking, edge-based minimization approaches are earlier ones, which involve only the segmentation curve as the unknown into the objective functional. The curve is represented explicitly as a parameterized curve \cite{kass1988snakes} or implicitly as the zero level curve of a level set function \cite{caselles1993geometric,malladi1995shape,caselles1997geodesic}. These objective models are solved usually by gradient descent algorithms, where the curves evolve with an edge detector dependent speed and stop at the boundaries of the detected objects. This type of methods work well for images with sharp and clean edges, but fail to give good results for those with weak edges. As the empty curve set is the trivial global minimizer of the objectives, we need good initializations for them to get good segmentation results. They are also complicated to deal with multiphase segmentation problems.

Region-based energy minimization methods usually simultaneously compute an approximate image of the original and the segmentation curves, and thus can be applied to segment images without edges. Among these, the Mumford-Shah model proposed in \cite{mumford1989optimal} is the most fundamental one, and many others were indeed proposed based on it. The model aims to minimize the following energy functional
\begin{equation}\label{model Mumfor-Shah}
\min_{u,C} \ \int_{\Omega}(f-u)^2d\mathbf{x} + \lambda\int_{\Omega\backslash C}|\nabla u|^2d\mathbf{x} + \mu |C|,
\end{equation}
where $\Omega\subset \mathbb{R}^2$ is the image domain, $f:\Omega\rightarrow\mathbb{R}$ is the given grayscale image, $u:\Omega\rightarrow\mathbb{R}$ is continuous in $\Omega\backslash C$ but may be discontinuous across $C$, and $|C|$ denotes the length of curve $C$. For a given $f$, it gives a piecewise smooth approximation and the separating curves between the smooth pieces. Clearly, dropping any of the three terms in \eqref{model Mumfor-Shah} leads to a trivial and meaningless solution. This minimization problem \eqref{model Mumfor-Shah} is an abstract model and difficult to be solved directly. Lots of efforts are contributed to reformulate or modify it to implementable ones.

A basic type of such efforts is to restrict the approximate image $u$ to be a piecewise constant function \cite{chan2001ieee,vese2002multiphase,chan2006algorithms,pock2008convex, lellmann2009convex,lellmann2009convex1,brown2009convex,brown2010convex,lellmann2011continuous,bae2011global,brown2012completely,li2010variational}. The very interesting Chan-Vese model in \cite{chan2001ieee} considers the two phase case, which reads
\begin{equation}\label{model Chan-Vese}
\min_{c_1,c_2,\phi} \ \lambda_1\int_{\Omega}(f-c_1)^2H(\phi)d\mathbf{x} + \lambda_2\int_{\Omega}(f-c_2)^2(1-H(\phi))d\mathbf{x} + \mu\int_{\Omega}\big|\nabla H(\phi)\big|d\mathbf{x} + \nu\int_{\Omega} H(\phi)d\mathbf{x},
\end{equation}
where the level set function $\phi: \Omega \to \mathbb{R}$ represents the curve and two regions by $C=\{\mathbf{x}\in\Omega:\phi(\mathbf{x})=0\}, inside(C)=\{\mathbf{x}\in\Omega:\phi(\mathbf{x})>0\}, outside(C)=\{\mathbf{x}\in\Omega:\phi(\mathbf{x})<0\}$; $H$ is the Heaviside function; the third term is the length of the curve; the last term is exactly the area of the $inside(C)$ region. The minimization problem in \eqref{model Chan-Vese} was then solved by gradient descent and alternating minimization. This approach was extended to multiphase case in \cite{vese2002multiphase} by using more level set functions. To overcome the numerical difficulties raised by the Heaviside function, some other sophisticated approaches \cite{chan2006algorithms,pock2008convex,lellmann2009convex,lellmann2009convex1,brown2009convex,brown2010convex,lellmann2011continuous, bae2011global,brown2012completely} were proposed by using characteristic functions of interested regions, yielding convex objectives in the case of the optimal constants $c_i$ known a priori, i.e., in the labeling case. All these piecewise constant models work well for approximately homogeneous images, but fail for those with stronger inhomogeneity.

The other class of models assume the approximate image $u$ to be piecewise smooth functions \cite{vese2002multiphase,vese2003multiphase,li2008minimization,chen2017general, li2011level}. Most of them use level set functions to represent the curve and minimize the energies in their level set formulations. In particular, Vese and Chan generalized their piecewise constant models \cite{chan2001ieee} to the piecewise smooth cases in \cite{vese2002multiphase}. The approaches in \cite{vese2003multiphase, chen2017general} assume $u$ to be piecewise polynomials, where only the optimal polynomial coefficients and the level set function need to be computed. A region-scalable local fitting model was proposed in \cite{li2008minimization} by introducing a kernel function into the Chan-Vese model in \cite{chan2001ieee}. In \cite{li2011level}, the authors constructed another classical local fitting model by assuming $u$ to be the product of a piecewise constant and a smooth function, and applied successfully to MRI image segmentation. This idea was recently used to construct a variant of Mumford-Shah model \cite{Liyutong2020} with $TV_p$ regularized region characteristic functions. These piecewise smooth models have shown some abilities to segment inhomogeneous images. However, whether using level set formulation or characteristic functions, they are computationally expensive and have rare convergence results, especially for the multiphase case.

The above reviewed methods, either edge-based, or region-based, are all one-stage methods and obtain the segmentation results in the energy minimization procedure. In contrast, some recently proposed very interesting approaches \cite{cai2013two,duan2015l_,li2016a,chang2017new,Chan2018Convex,cai2018linkage} do the segmentation in two stages. For an input image, they find an approximate image $u$ in the first stage, and threshold $u$ into its constituents by some thresholding approaches in the second stage. These two-stage approaches have some good characteristics. Firstly, as emphasized in \cite{cai2013two}, there is no need to give the number of segments first, any segmentation can be obtained after $u$ is computed in the first stage; those one-stage methods however require a predefined segments number, and if it changes, a new minimization problem needs to be solved. Secondly, the computation is much more efficient than one-stage methods, especially for multiphase cases, because
one-stage methods use multiple level set functions or region characteristic functions as unknown variables to represent the segmentation curves, which make the computation expensive and complicated.

In two-stage methods, the key is to find an approximate image $u$ in the first stage, which is suitable for the following thresholding or clustering operation. As the first two-stage approach, \cite{cai2013two} presented the following convex variant of the Mumford-Shah model
\begin{equation}\label{model_TSMS}
\min_{u} \frac{\lambda}{2} \int_{\Omega}(f-Au)^2d\mathbf{x} + \frac{\mu}{2}\int_{\Omega}|\nabla u|^2d\mathbf{x} + \int_{\Omega}|\nabla u|d\mathbf{x},
\end{equation}
which computes a smooth approximation of the input image for segmentation. This model benefits from its convexity and is very stable for homogeneous image segmentation, but it has limited ability to handle image inhomogeneity. In \cite{li2016a}, the authors proposed an interesting variant of \cite{cai2013two} with a more stable hill-climbing procedure for multi-channel image segmentation. After providing a deep understanding that a partial minimizer of the piecewise constant Mumford-Shah model can be obtained by thresholding the minimizer of the Rudin-Osher-Fatemi(ROF) model in \cite{cai2018linkage}, the authors constructed a so-called thresholded-ROF segmentation tool. In \cite{duan2015l_}, the authors proposed the following $L_0$ gradient regularized Mumford-Shah model
\begin{equation}\label{model_L0MS}
\min_{u,v} \frac{1}{2}\mathcal{E}(u,v) + \alpha\| \nabla u\|_0 + \frac{\mu}{2}\| \nabla v\|^2 + \frac{\gamma}{2}{\Vert v\Vert}^{2},
\end{equation}
by using the same assumption in \cite{li2011level}. In \cite{chang2017new}, the authors improved their previous model \eqref{model_L0MS} in 3D by replacing the discrete Tikhonov regularizer with a high order one and removing the local kernel function in the data fitting term. These $L_0$ related methods can generate piecewise constant approximations, which facilitate providing good segmentation results in the second stage. However, the $L_0$ minimization may generate false edges like isolated speckles for noisy images with strong inhomogeneity, due to its flatness over $(0,+\infty)$. Besides, there is no convergence analysis for the minimization algorithms in \cite{duan2015l_,chang2017new}. In \cite{Chan2018Convex}, Chan et al. proposed the following Convex Non-Convex (CNC) variational segmentation model
\begin{equation}\label{model_CNC}
\min_{u} \mathcal{J}(u;\lambda,T,a):= \frac{\lambda}{2}\|u-f\|^2 + \sum\limits_{i=1}^{N}\phi(\|(\nabla u)_i\|;T,a),
\end{equation}
where $\phi(\cdot;T,a):[0,+\infty)\to\mathbb{R}$ is a parameterized, piecewise defined non-convex penalty function, helping the model to recover images with sharp edges. Under some sufficient conditions on the parameters $\lambda,T,a$ such that the objective functional $\mathcal{J}(\cdot;\lambda,T,a)$ is strictly convex, the minimization algorithm in \cite{Chan2018Convex} is shown convergent. This method works quite well for approximately homogeneous images even with weak edges. However, like \eqref{model_TSMS}, this model does not explicitly consider image inhomogeneity and thus cannot segment strongly inhomogeneous images well.

As can been seen from the literature review, efficient and flexible multiphase segmentation for inhomogeneous images is still a great challenge.
As two-stage methods have some advantages, we follow this line and present a new approach trying to overcome the drawbacks of existing ones.
From the later Figure \ref{fig0_seg}, we see that the key step is to find a piecewise constant inhomogeneity-corrected approximation. This can be done by a decomposition model with appropriate regularizers for different image components, like those used in \cite{chambolle1997image,ng2011total,liang2015retinex,chang2017new} for different imaging applications. To regularize the piecewise constant inhomogeneity-free component, we are inspired by recent advances in non-Lipschitz regularized image restoration. It has been shown for various image restoration problems that continuous non-Lipschitz regularization has extremely good ability for recovering piecewise constant image with neat edges, by both the lower bound theory \cite{nikolova2005analysis,chen2012non,zeng2018edge,zeng2019non} and numerical experiments \cite{chen2012non,bian2015linearly,chen2013optimality,Chao2018An,zeng2019non}. Although non-Lipschitz minimization problems are very difficult to solve, there are some interesting advances, like smooth approximate methods \cite{chen2012non,chen2013optimality,bian2015linearly}, iterative reweighted $\ell_1$ (IRL1) for sparse recovery \cite{foucart2009sparsest,chen2014convergence}, iterative reweighted least squares
(IRLS) \cite{chartrand2008iteratively,daubechies2010iteratively,lai2013improved}, and the recent iterative support shrinking algorithms with proximal linearization (ISSAPL) for different signal and image restoration problems \cite{Chao2018An,zeng2019non,liu2019new,zhenzhe2019}. We mention that, there is so far no works studying continuous but non-Lipschitz regularization for image decomposition problems in the literature, except the just accepted \cite{Wang2020} presenting a nonconvex exponential TV-type Retinex model solved by an alternating minimization with no convergence analysis provided.

\subsection{Our contribution and paper organization}
In this paper we present a new two-stage method by using a continuous but non-Lipschitz decomposition model, trying to integrate the advantages of all existing two-stage methods together with better performance. The contributions can be summarized as follows:
\begin{enumerate}[1.]
	\item In the first stage, we propose a non-Lipschitz decomposition model to compute a piecewise constant inhomogeneity-corrected approximation. The objective function uses a continuous but non-Lipschitz regularizer and a second order discrete Tikhonov regularizer to model the piecewise constant approximation and the intensity inhomogeneity, respectively.
	\item By a motivating analysis, we naturally extend the recent iterative support shrinking algorithms with proximal linearization to solve our non-Lipschitz decomposition model. The global convergence of the algorithm is also established.
	\item To show the effectiveness and advantage of our method for image segmentation, we conduct a series of numerical experiments and compare the results with one typical level set method and some other two-stage segmentation techniques including the TV regularized model, the $L_0$ regularized model and the Convex Non-Convex model.
\end{enumerate}

The rest of the paper is organized as follows. The overview of our two-stage image segmentation method is given in Section \ref{sec_twostage}, including the proposed non-Lipschitz decomposition model in the first stage and the segmentation in the second stage. In Section \ref{sec_algorithm-convergence}, we give the algorithm for solving the decomposition model and establish its global convergence. The numerical experiments and comparisons with the state-of-art methods are given in Section \ref{sec_experiment}. We conclude the paper in Section \ref{sec_conclusions}.

\subsection*{Notations}\label{sec_notations}
Without loss of generality, we represent a grayscale image as an $n\times n$ matrix $\mathbf{u}$. Denote by $\mathcal{V}$ the linear space $\mathbb{R}^{n\times n}$. $\mathbf{J}:=\{(i,j):1\leq i,j\leq n\}$ is the set of the indices of all pixels. The horizontal and vertical discrete forward and backward difference operators are defined as follows
\begin{align*}
\begin{cases}
({\mathcal{D}_{+}^{x}} \mathbf{u})_{i,j}=\mathbf{u}_{i,j+1}-\mathbf{u}_{i,j}, & 1\leq i\leq n, 1\leq j\leq n, \\
({\mathcal{D}_{-}^{x}} \mathbf{u})_{i,j}=\mathbf{u}_{i,j}-\mathbf{u}_{i,j-1}, & 1\leq i\leq n, 1\leq j\leq n, \\
({\mathcal{D}_{+}^{y}} \mathbf{u})_{i,j}=\mathbf{u}_{i+1,j}-\mathbf{u}_{i,j}, & 1\leq i\leq n, 1\leq j\leq n, \\
({\mathcal{D}_{-}^{y}} \mathbf{u})_{i,j}=\mathbf{u}_{i,j}-\mathbf{u}_{i-1,j}, & 1\leq i\leq n, 1\leq j\leq n,
\end{cases}
\end{align*}
with periodic boundary conditions.
The discrete gradient operator is a mapping $\mathcal{D}$: $\mathcal{V} \longrightarrow  \mathcal{V} \times \mathcal{V}$ by $\forall \mathbf{u}, \mathcal{D}\mathbf{u} = (\mathcal{D}_{+}^x \mathbf{u},\mathcal{D}_{+}^y \mathbf{u})$. Clearly, the adjoint operator of $\mathcal{D}$ is $\mathcal{D}^T$: $\mathcal{V} \times \mathcal{V} \longrightarrow  \mathcal{V}$, written as $\mathcal{D}^T(\mathbf{p}^1,\mathbf{p}^2)=-{\mathcal{D}_{-}^{x}}\mathbf{p}^1-{\mathcal{D}_{-}^{y}}\mathbf{p}^2$. The discrete Hessian operator is a mapping $\mathcal{H}$: $V \longrightarrow  \mathcal{V} \times \mathcal{V} \times \mathcal{V} \times \mathcal{V}$, defined as
$$\mathcal{H}\mathbf{u} = \begin{pmatrix} {\mathcal{D}_{-}^{x}}\mathcal{D}_{+}^x \mathbf{u} & \mathcal{D}_{+}^x\mathcal{D}_{+}^y \mathbf{u} \\ \mathcal{D}_{+}^y\mathcal{D}_{+}^x \mathbf{u} & {\mathcal{D}_{-}^{y}}\mathcal{D}_{+}^y \mathbf{u} \end{pmatrix}.$$
Similarly, under the periodic boundary condition, the adjoint operator of $\mathcal{H}$ is $\mathcal{H}^T$: $\mathcal{V} \times \mathcal{V} \times \mathcal{V} \times \mathcal{V} \longrightarrow  \mathcal{V}$, which reads
$$\mathcal{H}^T\begin{pmatrix} \mathbf{p}^{11} & \mathbf{p}^{12} \\ \mathbf{p}^{21} & \mathbf{p}^{22} \end{pmatrix}={\mathcal{D}_{+}^{x}}\mathcal{D}_{-}^x \mathbf{p}^{11} + {\mathcal{D}_{-}^{y}}{\mathcal{D}_{-}^{x}} \mathbf{p}^{12} + {\mathcal{D}_{-}^{x}}{\mathcal{D}_{-}^{y}} \mathbf{p}^{21} + \mathcal{D}_{+}^{y}\mathcal{D}_{-}^{y} \mathbf{p}^{22}.$$
For more details, see, e.g., \cite{wang2008new,wu2010augmented} and the references therein.

For convenience of description in theoretical analysis, we also use another representation, which rearranges column by column an image $\mathbf{u} \in \mathcal{V}$ into a 1D vector $u\in\mathbb{R}^N, N=n^2$, like \cite{chen2012non,zeng2018edge,Chao2018An,zeng2019non}. Thus the corresponding index set of all pixels is $J:=\{k:1\leq k\leq N\}$. There is a one-to-one correspondence between $J$ and $\mathbf{J}$: $k \leftrightarrows (i,j)$, where $k \in J$ and $(i,j) \in \mathbf{J}$. Using this correspondence, one can reformulate those above discrete difference operators for $u \in\mathbb{R}^N$, which are denoted as $D$, $H$, etc. For instance,
$$
(D_+^x u)_k=(\mathcal{D}_{+}^x \mathbf{u})_{i,j}=
\begin{cases}
u_{k+n}-u_{k}, & 1\leq \lceil k/n \rceil \leq n-1, \\
u_{k-n(n-1)}-u_{k}, & \lceil k/n \rceil = n,
\end{cases}
$$
$$
(D_+^y u)_k=(\mathcal{D}_{+}^y \mathbf{u})_{i,j}=
\begin{cases}
u_{k+1}-u_{k}, & 1\leq k\bmod n \leq n-1, \\
u_{k-n+1}-u_{k}, & k\bmod n = 0,
\end{cases}
$$
where $D_+^x$, $D_+^y  \in \mathbb{R}^{N\times N}$. The discrete gradient operator for  $u \in\mathbb{R}^N$ is $D=\begin{pmatrix} D_+^x \\ D_+^y \end{pmatrix} \in \mathbb{R}^{2N\times N}$. $D_-^x$, $D_-^y$ can be similarly defined. Thus the discrete Hessian operator for $u \in\mathbb{R}^N$ is
$H = \begin{pmatrix} D_{-}^{x}D_{+}^x \\ D_{+}^xD_{+}^y \\ D_{+}^yD_{+}^x \\ D_{-}^{y}D_{+}^y \end{pmatrix} \in \mathbb{R}^{4N\times N}$. We also denote $D_k=\begin{pmatrix} D_+^x(k,:) \\ D_+^y(k,:) \end{pmatrix} \in \mathbb{R}^{2 \times N}$.

To simplify the notation, the $\ell_2$-norm $\|\cdot\|_2$ of a vector is abbreviated as $\|\cdot\|$ in the main body of this paper.
For $V_1 \subseteq \mathbb{R}^N, V_2 \subseteq \mathbb{R}^N$, the tensor product of $V_1$ and $V_2$  reads as $V_1 \otimes V_2 :=\left\{ (v_1,v_2):v_1\in V_1, v_2\in V_2 \right\}$. We also use $\# \Lambda$ to denote the cardinality of a set $\Lambda$. The kernel of $D_k$ is written as $\ker D_k=\{u\in \mathbb{R}^{N}: \Vert D_k u\Vert=0\}$. Given $\hat{u}\in \mathbb{R}^N$, we denote the support set of the gradients of $\hat{u}$ as
$$\Omega_1(\hat{u})=\{k\in J:\Vert D_k \hat{u}\Vert \neq 0\},$$
and denote $\Omega_0(\hat{u}):=J\backslash \Omega_1(\hat{u})$.

\section{An overview of our two-stage image segmentation method}
\label{sec_twostage}
Image segmentation can be regarded as a pixel classification problem in a certain feature space, which can be done by a clustering or thresholding procedure. In median level vision, the most widely used feature is the pixel intensity. Figure \ref{fig0_seg} shows some results by simply applying the thresholding method in \cite{cai2013two}, which is an improved version of a conventional clustering algorithm implemented by the MATLAB command KMEANS. This thresholding method works very well for piecewise constant images, but fails for those with intensity inhomogeneity, yielding uncorrect object recognition. From this observation, we see that, if we can remove the intensity inhomogeneity and compute a piecewise constant approximate image first, then we can segment it more precisely.
This is exactly a two-stage strategy, like in \cite{cai2013two,duan2015l_,chang2017new,Chan2018Convex}, where the first stage is the key one. Our approach is described as follows.

\begin{figure}[H]
	\centerline{
		\begin{tabular}{c@{}c@{}c@{}c@{}c@{}c}
			\includegraphics[width=0.8in]{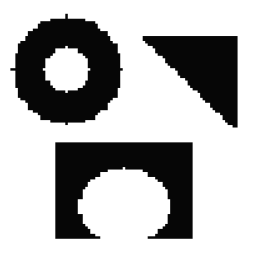} \ &
			\includegraphics[width=0.8in]{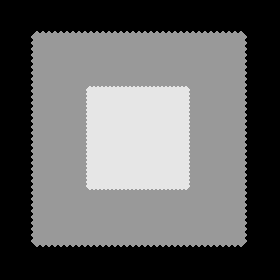} \ &
			\includegraphics[width=0.8in]{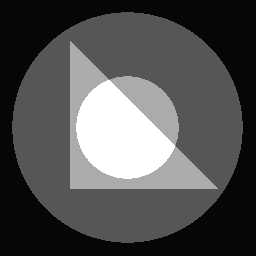} \ \ \ &
			\includegraphics[width=0.8in]{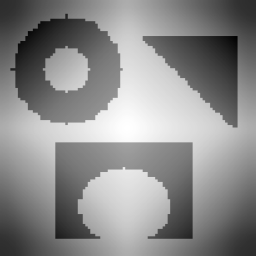} \ &
			\includegraphics[width=0.8in]{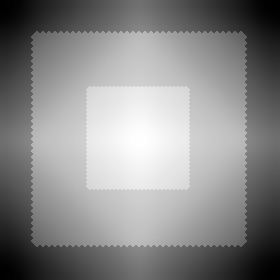} \ &
			\includegraphics[width=0.8in]{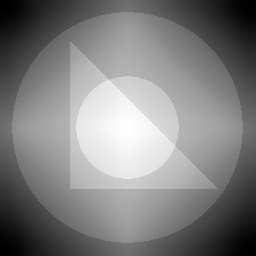} \\
			\includegraphics[width=0.8in]{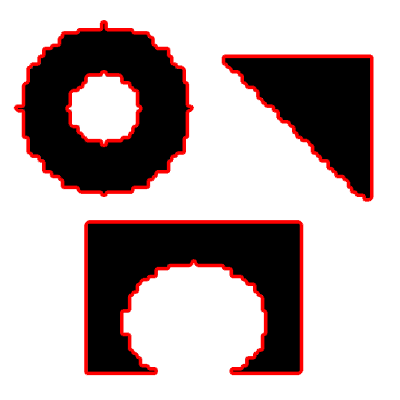} \ &
			\includegraphics[width=0.8in]{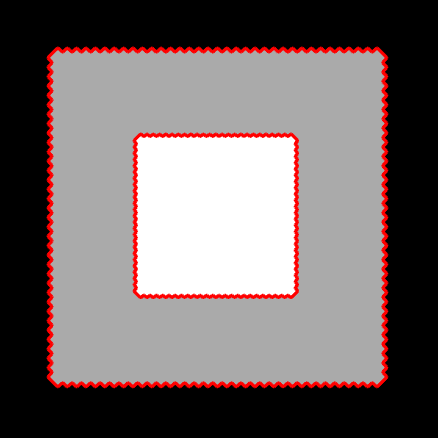} \ &
			\includegraphics[width=0.8in]{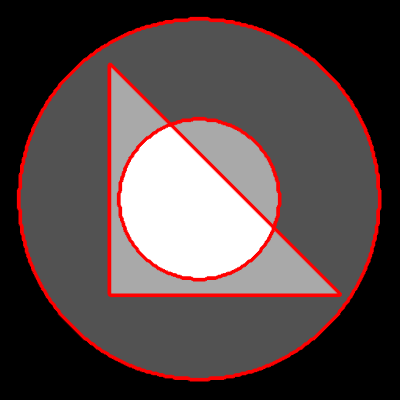} \ \ \ &
			\includegraphics[width=0.8in]{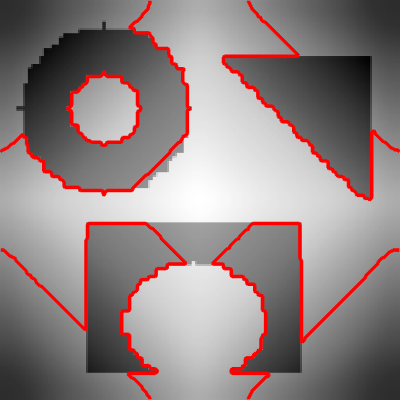} \ &
			\includegraphics[width=0.8in]{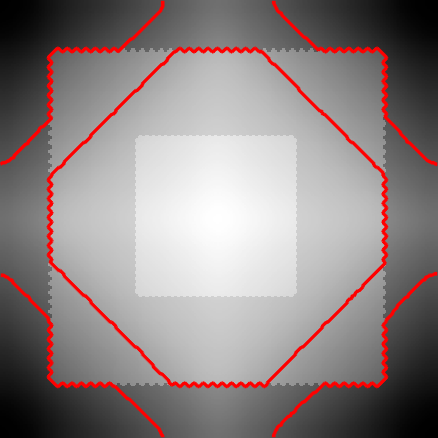} \ &
			\includegraphics[width=0.8in]{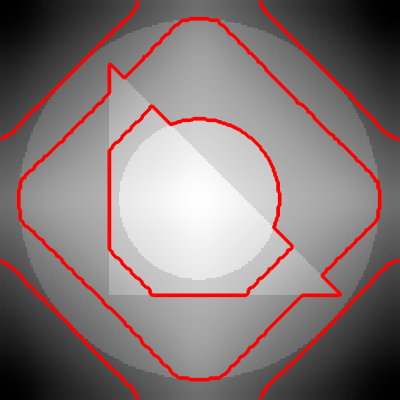} \\
		\end{tabular}
	}
	\caption{\small\sl The necessity of inhomogeneity correction for inhomogeneous image segmentation using the thresholding method in \cite{cai2013two}. Row 1: Some piecewise constant (the first three columns) and intensity inhomogeneous (the last three columns) images; Row 2: the corresponding segmentation results by using the thresholding method directly.}
	\label{fig0_seg}
\end{figure}

\subsection{The first stage: inhomogeneity removal by non-Lipschitz variational decomposition}\label{sec_stageone}

As continuous non-Lipschitz regularization has been shown to generate piecewise constant solutions with neat edges in image restoration \cite{nikolova2005analysis,chen2012non,chen2013optimality,bian2015linearly,zeng2018edge,Chao2018An,zeng2019non}, we believe that it will benefit the solution of image decomposition problems. We here propose to combine it with the Tikhonov regularizer with second-order differences, to build the following non-Lipschitz additive decomposition model
\begin{equation}\label{equ_model1}
\min \limits_{u,v} \ F(u,v):=\frac{1}{2}{\Vert f-u-v\Vert}^{2} + \alpha \sum\limits_{i\in J}\phi(\Vert D_i u\Vert) + \frac{\beta}{2}{\Vert Hv\Vert}^{2} + \frac{\gamma}{2}{\Vert v\Vert}^{2},
\end{equation}
where $\phi:\ [0,\infty)\to[0,\infty)$ is a potential function which is usually assumed to satisfy Assumption \ref{assum-phi}. We will present an algorithm and convergence analysis for this model in the next section.

\begin{assumption}\label{assum-phi}
	\begin{enumerate}[(i)]
		\item $\phi:[0,+\infty)\rightarrow [0,+\infty)$ is continuous, concave and coercive with $\phi(0)=0$;
		\item $\phi$ is $C^1$ on $(0,+\infty)$ with $\phi'(t)|_{(0,+\infty)}>0$ and $\phi'(0+)=+\infty$;
		\item For any $c>0$, $\phi'$ is $L_{c}$-Lipschitz continuous on $[c,+\infty)$, i.e., there exists a constant $L_c$ determined by c, such that for all $x,y \in [c,+\infty),\ |\phi'(x)-\phi'(y)| \leq L_c|x-y|$.
	\end{enumerate}
\end{assumption}
\begin{remark}\label{remark_1}
	\begin{enumerate}[(a)]
		\item Assumption \ref{assum-phi}(\romannumeral1) implies that $F(u,v)$ is continuous, bounded below, and coercive. Therefore, the minimizer of $F(u,v)$ always exists.
		\item Assumption \ref{assum-phi}(\romannumeral2) implies that $\phi$ is non-Lipschitz.
	\end{enumerate}
\end{remark}

\begin{remark}\label{remark_gic}
	We mention that, in \cite{Wuchunlin2020}, we consider a general non-Lipschitz regularized infimal convolution model with two low level vision applications, i.e., Retinex and cartoon-texture decomposition.
\end{remark}

\subsection{The second stage: segmentation by thresholding}\label{sec_stagetwo}
After the inhomogeneity removal from the decomposition model in \eqref{equ_model1}, we adopt the widely-used thresholding approach in \cite{cai2013two} to get the segmentation of $u$ in the second stage, whose thresholds are determined by some clustering algorithm. For the sake of completeness, we give its procedure here. The approach uses the $\mathcal{K}$-means clustering algorithm by the MATLAB $\mathcal{K}$-means command KMEANS, to classify the pixel values of $u$ into $\mathcal{K}$ clusters. Specifically, let the mean value of each cluster be $\hat{\rho}_1, \hat{\rho}_2, \cdots, \hat{\rho}_\mathcal{K}$ and, without loss of generality, $\hat{\rho}_1 \leq \hat{\rho}_2 \leq \cdots \leq \hat{\rho}_\mathcal{K}$. Based on this, the approach defines the $(\mathcal{K}-1)$ thresholds as
$$\rho_i = \frac{\hat{\rho}_i+\hat{\rho}_{i+1}}{2}, \quad i=1,2,\ldots,\mathcal{K}-1.$$
Then the $i$th phase of $u$, $1 \leq i \leq \mathcal{K}$, is simply given by $\{ j\in J:\rho_{i-1} < u_j \leq \rho_i\}$.

\section{The algorithm and convergence analysis} \label{sec_algorithm-convergence}

\subsection{Algorithm}\label{sec_algorithm}
We first compute the subdifferential of $F(u,v)$ and present a motivating theorem of our iterative algorithm.
\begin{lemma}\label{lemma-subdifferential}
	(Subdifferential) Let $\partial F(\hat{u},\hat{v})$ be the subdifferential of $F(u,v)$ at $(\hat{u},\hat{v})$. Then $\partial F(\hat{u},\hat{v})=W_1 \otimes W_2$, where
	\begin{equation} \label{equ_subdifferential}
	\left\{
	\begin{aligned}
	W_1 & = \alpha\sum_{ i \in  \Omega_0(\hat{u})}(\ker D_i)^\bot + \alpha\sum_{ i \in  \Omega_1(\hat{u})} \phi'(\|D_i \hat{u}\|)\frac{{D_i}^TD_i \hat{u}}{\|D_i \hat{u}\|} + (\hat{u}+\hat{v}-f), \\
	W_2 & = \left\{ (\hat{u}+\hat{v}-f) + \beta H^TH\hat{v} + \gamma \hat{v} \right\}.
	\end{aligned} \right.
	\end{equation}	
\end{lemma}
\begin{proof}
	Since $J=\Omega_0(\hat{u}) \cup \Omega_1(\hat{u})$, $F(u,v)$ in \eqref{equ_model1} reads
	$$F(u,v)=\frac{1}{2}{\Vert f-u-v\Vert}^{2} + \frac{\beta}{2}{\Vert Hv\Vert}^{2} + \frac{\gamma}{2}{\Vert v\Vert}^{2} + \alpha \sum\limits_{i\in \Omega_1(\hat{u})}\phi(\Vert D_i u\Vert) + \alpha \sum\limits_{i\in \Omega_0(\hat{u})}\phi(\Vert D_i u\Vert).$$
	Because $\|D_i \hat{u}\| \not= 0,\ \forall i \in \Omega_1(\hat{u})$ and $\|D_i \hat{u}\| = 0,\ \forall i \in \Omega_0(\hat{u})$, we have, by (\cite[Exercise 8.8]{rockafellar2009variational}),
	$$\partial F(\hat{u},\hat{v})=\nabla \left( \frac{1}{2}{\Vert f-u-v\Vert}^{2} + \frac{\beta}{2}{\Vert Hv\Vert}^{2} + \frac{\gamma}{2}{\Vert v\Vert}^{2} + \alpha \sum\limits_{i\in \Omega_1(\hat{u})}\phi(\Vert D_i u\Vert) \right)(\hat{u},\hat{v}) + \ \partial \left( \alpha\sum\limits_{i\in \Omega_0(\hat{u})}\phi(\Vert D_i u\Vert) \right)(\hat{u},\hat{v}).$$
	The first term on the right hand side is clearly
	$$
	\begin{aligned}
	& \nabla \left( \frac{1}{2}{\Vert f-u-v\Vert}^{2} + \frac{\beta}{2}{\Vert Hv\Vert}^{2} + \frac{\gamma}{2}{\Vert v\Vert}^{2} + \alpha \sum\limits_{i\in \Omega_1(\hat{u})}\phi(\Vert D_i u\Vert) \right)(\hat{u},\hat{v}) \\
	= & \begin{pmatrix} \hat{u}+\hat{v}-f + \alpha\sum_{ i \in  \Omega_1(\hat{u})} \phi'(\|D_i \hat{u}\|)\frac{{D_i}^TD_i \hat{u}}{\|D_i \hat{u}\|}\\ (\hat{u}+\hat{v}-f) + \beta H^TH\hat{v} + \gamma \hat{v} \end{pmatrix}.
	\end{aligned}
	$$
	By (\cite[Proposition 10.5]{rockafellar2009variational}) and referring to \cite[Theorem 2.3]{zhenzhe2019}, we obtain
	$$\partial \left( \alpha\sum\limits_{i\in \Omega_0(\hat{u})}\phi(\Vert D_i u\Vert) \right)(\hat{u},\hat{v})=\alpha\sum_{ i \in  \Omega_0(\hat{u})}(\ker D_i)^\bot \otimes \left\{0 : 0\in \mathbb{R}^N \right\}.$$
	This completes the proof.
\end{proof}

\begin{theorem}\label{theorem_motivation}(Motivation theorem)
	For a given point $(\hat{u},\hat{v})$, there exists a constant $\delta_0>0$, such that for any local minimizer $(u^*,v^*) \in B((\hat{u},\hat{v}),\delta_0)$ of $F(u,v)$, if exists, we have
	$$\Omega_0(\hat{u}) \subseteq \Omega_0(u^*).$$
\end{theorem}
\begin{proof}
	Let $g(u,v):=\frac{1}{\alpha}\|f-u-v\|^2$. Since $g(u,v)$ is continuous, for $\epsilon=1$, there exists a $\delta_1$, such that for any $(u^*,v^*) \in B((\hat{u},\hat{v}), \delta_1)$, we have
	\begin{equation} \label{equ_continuity}
	|g(u^*,v^*)-g(\hat{u},\hat{v})| < 1.
	\end{equation}
	
	\noindent Now we establish a lower bound $\theta(\hat{u},\hat{v})$ of $\|D_i u^*\|$ for any $(u^*,v^*) \in B((\hat{u},\hat{v}), \delta_1)$.
	
	Since $(u^*,v^*)$ is a local minimizer, by the first-order optimality condition and Lemma \ref{lemma-subdifferential}, we have
	$$
	0 \in \alpha\sum_{ i \in  \Omega_0(u^*)}(\ker D_i)^\bot + \alpha\sum_{ i \in  \Omega_1(u^*)} \phi'(\|D_i u^*\|)\frac{{D_i}^TD_i u^*}{\|D_i u^*\|} + (u^*+v^*-f).
	$$
	For $\sum_{ i \in  \Omega_0(u^*)}(\ker D_i)^\bot = \left(\cap_{i \in  \Omega_0(u^*)}\ker D_i\right)^\bot$, we have
	$$
	\left\langle \alpha\sum_{ i \in  \Omega_1(u^*)} \phi'(\|D_i u^*\|)\frac{{D_i}^TD_i u^*}{\|D_i u^*\|} + (u^*+v^*-f), w \right\rangle = 0, \hspace{0.5cm} \forall \ w \in \cap_{i \in  \Omega_0(u^*)}\ker D_i.
	$$
	Hence,
	\begin{equation}\label{equ_inequality}
	\begin{aligned}
	\sum_{ i \in \Omega_1(u^*)}\phi'(\|D_i u^*\|)\left\langle \frac{D_i u^*}{\|D_i u^*\|}, D_i w \right\rangle
	= & \frac{1}{\alpha} \left\langle f-u^*-v^*, w \right\rangle   \\
	\leq & \frac{1}{\alpha}  \|f-u^*-v^*\|\|w\|  \\
	= & g(u^*,v^*) \|w\| \\
	\mbox{ [ by \eqref{equ_continuity} ] } < & (g(\hat{u},\hat{v})+1)\|w\|.
	\end{aligned}
	\end{equation}	
	Obviously, $(g(\hat{u},\hat{v})+1)$ is a positive constant relying only on $(\hat{u},\hat{v})$. From \eqref{equ_inequality}, we can obtain a positive lower bound of nonzero $\|D_i u^*\|$ by a similar argument as in \cite[Theorem 1]{Chao2018An}, or \cite{zhenzhe2019}, by using the tool of ``Conservativeness of image gradient fields" developed in \cite{zeng2018edge}. That is to say, there exists $\theta(\hat{u},\hat{v})>0$ such that
	\begin{equation} \label{equ_locallowerbound1}
	either\ \ \Vert D_i u^* \Vert=0\ \ or \ \ \Vert D_i u^* \Vert>\theta(\hat{u},\hat{v}),\quad \forall (u^*,v^*) \in B((\hat{u},\hat{v}), \delta_1),\ \forall i \in J.
	\end{equation}
	We mention that, $\theta(\hat{u},\hat{v})$ is a constant dependent on the given $(\hat{u},\hat{v})$.
	
	Since $D_i$ is continuous, for $\epsilon=\theta(\hat{u},\hat{v})$, there exists a $\delta_0 \leq \delta_1$, such that	
	\begin{equation} \label{equ_localcontinuous}
	\|D_i u^* - D_i \hat{u}\| < \theta(\hat{u},\hat{v}), \quad \forall (u^*,v^*) \in B((\hat{u},\hat{v}), \delta_0),\ \forall i \in J.
	\end{equation}
	
	For any $(u^*,v^*) \in B((\hat{u},\hat{v}), \delta_0)$, we prove $\Omega_0(\hat{u}) \subseteq \Omega_0(u^*)$ by contradiction. If $\exists \ i_0\in J$, $\|D_{i_0} \hat{u}\|=0, \|D_{i_0} u^*\|\not=0$. On one hand, by \eqref{equ_locallowerbound1} and $\delta_0 \leq \delta_1$, we have
	\begin{equation} \label{equ_contradiction1}
	\|D_{i_0} u^* - D_{i_0} \hat{u}\|=\|D_{i_0} u^*\|>\theta(\hat{u},\hat{v}).
	\end{equation}
	On the other hand, by \eqref{equ_localcontinuous}, we have
	
	\begin{equation} \label{equ_contradiction2}
	\|D_{i_0} u^* - D_{i_0} \hat{u}\| < \theta(\hat{u},\hat{v}).
	\end{equation}
	
	\noindent \eqref{equ_contradiction1} and \eqref{equ_contradiction2} form a contradiction, which proves $\Omega_0(\hat{u}) \subseteq \Omega_0(u^*)$.
\end{proof}

According to Theorem \ref{theorem_motivation}, if a local minimizer $(u^*,v^*)$ is near to a given point $(\hat{u},\hat{v})$, $\Vert D_i u^* \Vert$ should be zero when $\Vert D_i \hat u \Vert=0$. This phenomenon naturally implies an iterative support shrinking procedure for the problem in \eqref{equ_model1}. We mention that this kind of strategy was also derived for different signal and image restoration problems with different objective functions \cite{Chao2018An,zeng2019non,liu2019new,zhenzhe2019}. Given $(u^{k},v^{k})$, we thus compute $(u^{k+1},v^{k+1})$ by solving
\begin{equation}\label{model1-shrink}
(\mathcal{E}_k) \hspace{0.2cm} \left\{
\begin{aligned}
\min \limits_{u,v} \quad &E_k(u,v) := \frac{1}{2}{\Vert f-u-v\Vert}^{2}+\alpha\sum\limits_{i\in\Omega_{1}^{k}}\phi(\Vert D_i u\Vert)
+\frac{\beta}{2}{\Vert Hv\Vert}^{2}+\frac{\gamma}{2}{\Vert v\Vert}^{2}, \\
\st \quad &D_i u=0,\ \ \ \forall i \in \Omega_{0}^{k},
\end{aligned} \right.
\end{equation}
where $\Omega_{1}^{k}=\Omega_{1}(u^k)$ and $\Omega_{0}^{k}=\Omega_{0}(u^k)$.
The following relation can be verified easily:
\begin{equation}\label{equ_FkF}
E_k(u^{k+j},v^{k+j}) = F(u^{k+j},v^{k+j}), \quad \forall k,j \geq 0.
\end{equation}

Due to the non-convexity of $\phi$, $(\mathcal{E}_k)$ in \eqref{model1-shrink} is obviously still difficult to solve. Using a linear approximation of $\phi$ and a proximal technique, like in \cite{Chao2018An,zeng2019non}, we give the following iterative support shrinking algorithm with proximal linearization (ISSAPL) for the problem in \eqref{equ_model1}. Note that, we need only a proximal term for the $u$ variable, because the objective is already strongly convex with respect to $v$.

\begin{mdframed}
	\textbf{ISSAPL-ID: iterative support shrinking algorithm with proximal linearization for image decomposition model \eqref{equ_model1}.}
	\begin{enumerate}
		\item Input $f,\alpha,\beta,\gamma,\rho>0$. Initialize $(u^0,v^0)=(f,f)$.
		\item \label{step-ISSAPL} For $k=0,1,\dots$, compute $(u^{k+1},v^{k+1})$ by solving
		\begin{equation}\label{model1-shrink-approx}
		(\mathcal{G}_k) \hspace{0.2cm} \left\{
		\begin{aligned}
		\min \limits_{u,v} \quad & 	G_k(u,v) :=\frac{1}{2}{\Vert f-u-v\Vert}^{2}+\alpha \sum\limits_{i\in\Omega_{1}^{k}}\phi ^{'}(\Vert D_i u^k\Vert)\Vert D_i u\Vert \\
		&\qquad\qquad +\frac{\rho}{2}{\Vert u-u^{k}\Vert}^{2}+\frac{\beta}{2}{\Vert Hv\Vert}^{2}+\frac{\gamma}{2}{\Vert v\Vert}^{2}, \\
		\st \quad & D_i u=0,\ \ \ \forall i \in \Omega_{0}^{k}.
		\end{aligned} \right.
		\end{equation}
		
		Until a termination criterion is met.
	\end{enumerate}
\end{mdframed}

\subsection{Convergence analysis of ISSAPL-ID} \label{sec_convergence}
In this subsection, we establish the global convergence of the sequence $\{(u^k,v^k)\}$ generated by ISSAPL-ID, by using the Kurdyka-{\L}ojasiewicz (KL) property and an abstract framework shown in \cite[Theorem 2.9]{attouch2013convergence} for descent algorithms. We assume that each $(\mathcal{G}_k)$ in \eqref{equ_model1} is exactly solved.

The KL property of real functions studied in \cite{lojasiewicz1963propriete}\cite{kurdyka1998gradients} has recently become a key concept and tool for the convergence analysis in non-convex optimization. One can refer to \cite{bolte2007lojasiewicz,bolte2007clarke,attouch2009convergence,attouch2010proximal,attouch2013convergence,bolte2014proximal} for its applications in optimization and \cite{attouch2010proximal,attouch2013convergence,bolte2014proximal,ochs2015iteratively} for examples of KL functions. Some related preliminaries of KL property have been provided in the Appendix, where we illustrate that $F(u,v)$ in this paper is indeed a KL function.

A very useful abstract framework for analyzing descent algorithms is shown in \cite[Theorem 2.9]{attouch2013convergence}. For a proper lower semicontinuous function $h(x):\mathbb{R}^n \to \mathbb{R}\cup\{\infty\}$ satisfying the KL property, the authors proved that the sequence $\{x^k\}$ converges to a critical point of $h$, if $\{x^k\}$ satisfies three conditions: sufficient decrease condition, relative error condition and continuity condition. For our problem, we will demonstrate the sufficient decrease condition in Lemma \ref{lemma_decrease} and relative error condition in Lemma \ref{lemma_subgradient_bound}, respectively. The continuity condition is obvious. The global convergence is then concluded in Theorem \ref{theorem_convergence}.

For the convenience of later description, we define $\mathcal{C}_k: = \cap_{i \in \Omega_0^k}\ker D_i = \{u\in \mathbb{R}^{N}: \Vert D_i u\Vert=0,\ \forall i \in \Omega_0^k\}$ and denote its indicator function as $\delta_{\mathcal{C}_k}:\mathbb{R}^N\rightarrow (-\infty,+\infty]$, which reads
$$
\delta_{\mathcal{C}_k}(u)=
\begin{cases}
0  & \text{if}\ u\in \mathcal{C}_k,\\
+\infty  & \mbox{otherwise}.
\end{cases}
$$
Clearly the subdifferential of $\delta_{\mathcal{C}_k}$ at $u$ is
$$
\partial\big( \delta_{\mathcal{C}_k}(u) \big)= N_{\mathcal{C}_k}(u)=
\begin{cases}
\left(\cap_{i \in  \Omega_0^k}\ker D_i\right)^\bot & \text{if}\ u\in \mathcal{C}_k,\\
\emptyset & \mbox{otherwise},
\end{cases}
$$
where $N_{\mathcal{C}_k}(u)$ is the normal cone of $\mathcal{C}_k$ at $u$. Using $\delta_{\mathcal{C}_k}$, we reformulate the minimization problem 	$(\mathcal{G}_k)$ to the following unconstrained one
\begin{equation}\label{equ-unchange-approx-delta}
(\mathcal{G}_k^{\delta}) \hspace{0.5cm}
\min \limits_{u,v} \quad G_k^{\delta}(u,v) := G_k(u,v) + \delta_{\mathcal{C}_k}(u),
\end{equation}
which will be used soon.

We start from the following inequality.
\begin{lemma}\label{lemma-convex property}
	Let $(u^{k+1},v^{k+1})$ be the solution of 	$(\mathcal{G}_k^{\delta})$. Then it holds that
	\begin{equation}\label{equ-convex property}
		G_k(u^k,v^k) \geq G_k(u^{k+1},v^{k+1}) + \frac{1}{2}\min\{\rho,\gamma\}  {\left\| (u^k,v^k)-(u^{k+1},v^{k+1}) \right\|}^{2}.
	\end{equation}
\end{lemma}
\begin{proof}
	It is not hard to verify that $G_k^{\delta}(u,v)-\frac{1}{2}\min\{\rho,\gamma\}  {\left\| (u,v) \right\|}^{2}$ is a convex function. By (\cite[Exercise 12.59]{rockafellar2009variational}), $G_k^{\delta}(u,v)$ is thus strongly convex with constant $\frac{1}{2}\text{min}\{\rho,\gamma\}$. We therefore have
	\begin{equation} \label{equ-convex property-temp}
	\begin{aligned}
	G_k^{\delta}(u^k,v^k) \geq G_k^{\delta}(u^{k+1},v^{k+1}) & + \left\langle \partial G_k^{\delta}(u^{k+1},v^{k+1}), (u^k,v^k)-(u^{k+1},v^{k+1}) \right\rangle \\
	& + \frac{1}{2}\text{min}\{\rho,\gamma\} {\left\| (u^k,v^k)-(u^{k+1},v^{k+1}) \right\|}^{2}.
	\end{aligned}
	\end{equation}
	Since $(u^{k+1},v^{k+1})$ is the solution of $(\mathcal{G}_k^{\delta})$, it follows from the first-order optimality condition that
	$$0 \in \partial G_k^{\delta}(u^{k+1},v^{k+1}).$$
	Combining it with \eqref{equ-convex property-temp} and $\delta_{\mathcal{C}_k}(u^k)=\delta_{\mathcal{C}_k}(u^{k+1})=0$, we obtain \eqref{equ-convex property}.
\end{proof}

\begin{lemma}\label{lemma_decrease}(Sufficient decrease condition)
	The sequence $\{F(u^{k},v^{k})\}$ is nonincreasing and in particular
	\begin{equation}\label{equ_Fnonincrease}
	\frac{1}{2}\min\{\rho,\gamma\}  {\left\| (u^{k+1},v^{k+1})-(u^k,v^k) \right\|}^{2} \leq F(u^k,v^k) - F(u^{k+1},v^{k+1}), \quad \forall k \geq 0.
	\end{equation}
\end{lemma}
\begin{proof}
	Since $\phi$ is concave, we have
	$$
	\phi(\|D_i u^{k+1}\|)\leq \phi(\|D_i u^k\|)+ \phi'(\|D_i u^k\|)(\|D_i u^{k+1}\|-\|D_i u^k\|), \quad \forall i \in \Omega_1^k.
	$$
	This indicates
	\begin{equation}\label{equ_FkH}
	E_k(u^{k+1},v^{k+1}) + \frac{\rho}{2}\|u^{k+1}-u^k\|^2 \leq 	G_k(u^{k+1},v^{k+1}) + \alpha\sum_{ i \in \Omega_1^k}\Big(\phi(\|D_i u^k\|) - \phi'(\|D_i u^k\|)\|D_i u^k\|\Big).
	\end{equation}
	It then follows that
	\begin{align*}
	&   F(u^{k+1},v^{k+1}) + \frac{1}{2}\text{min}\{\rho,\gamma\}{\left\| (u^{k+1},v^{k+1})-(u^k,v^k) \right\|}^{2} \\
	\leq & F(u^{k+1},v^{k+1}) + \frac{\rho}{2}\|u^{k+1}-u^k\|^2 + \frac{1}{2}\text{min}\{\rho,\gamma\}{\left\| (u^{k+1},v^{k+1})-(u^k,v^k) \right\|}^{2} \\
	\mbox{ [ by \eqref{equ_FkF} ] } =    & E_k(u^{k+1},v^{k+1}) + \frac{\rho}{2}\|u^{k+1}-u^k\|^2 + \frac{1}{2}\text{min}\{\rho,\gamma\} {\left\| (u^{k+1},v^{k+1})-(u^k,v^k) \right\|}^{2} \\
	\mbox{ [ by \eqref{equ_FkH} ] } \leq & 	G_k(u^{k+1},v^{k+1}) + \alpha\sum_{ i \in \Omega_1^k}\Big(\phi(\|D_i u^k\|) - \phi'(\|D_i u^k\|)\|D_i u^k\|\Big)  \\
	+ & \frac{1}{2}\text{min}\{\rho,\gamma\} {\left\| (u^{k+1},v^{k+1})-(u^k,v^k) \right\|}^{2} \\
	\mbox{ [ by \eqref{equ-convex property} ] }  \leq & 	G_k(u^{k},v^k) + \alpha\sum_{ i \in \Omega_1^k}\Big(\phi(\|D_i u^k\|) - \phi'(\|D_i u^k\|)\|D_i u^k\|\Big)  \\
	= & E_k(u^{k},v^k) \\	
	\mbox{ [ by \eqref{equ_FkF} ] } = & F(u^k,v^k),
	\end{align*}
	which proves \eqref{equ_Fnonincrease}.
\end{proof}

Lemma \ref{lemma_decrease} implies the following properties of $\{(u^k,v^k)\}$, which will be used in the proof of Theorem \ref{theorem_lowerbound} and Theorem \ref{theorem_convergence}.
\begin{lemma}\label{lemma_bounded}
	The sequence $\{(u^k,v^k)\}$ is bounded and satisfies
	\begin{equation}\label{equ_uk}
	\lim_{k\rightarrow \infty}{\left\| (u^{k+1},v^{k+1})-(u^k,v^k) \right\|}=0.
	\end{equation}
\end{lemma}
\begin{proof}
	Combining Lemma \ref{lemma_decrease} and $F(u,v)\geq 0$, we know that $\{F(u^k,v^k)\}$ is bounded and convergent.
	By Assumption \ref{assum-phi}(\romannumeral1), $F(u,v)$ is coercive (\cite[Definition 3.25]{rockafellar2009variational}). Thus the sequence $\{(u^k,v^k)\}$ is bounded.
	
	Let $M$ be a positive integer. Summing \eqref{equ_Fnonincrease} from $k=0$ to $M$, we obtain
	$$
	\sum_{k=0}^{M} {\left\| (u^{k+1},v^{k+1})-(u^k,v^k) \right\|}^{2} \leq \frac{2}{\text{min}\{\rho,\gamma\}} \left(F(u^0,v^0) - F(u^{M+1},v^{M+1})\right) \leq \frac{2}{\text{min}\{\rho,\gamma\}} F(u^0,v^0).
	$$
	Taking the limit as $M\rightarrow \infty$ yields
	$$
	\sum_{k=0}^{\infty} {\left\| (u^{k+1},v^{k+1})-(u^k,v^k) \right\|}^{2} \leq \frac{2}{\text{min}\{\rho,\gamma\}} F(u^0,v^0),
	$$
	which leads to \eqref{equ_uk}.
\end{proof}

The subgradient lower bound for the iterates gap (Lemma \ref{lemma_subgradient_bound}) is not difficult for those objective functions with Lipschitz gradients \cite{bolte2014proximal,zhang2017nonconvex}. This is however not trivial in our problem due to the non-Lipschitz objective gradient. To overcome this difficulty, we construct a lower bound for the nonzero gradients of the iterative sequence (Theorem \ref{theorem_lowerbound}). For this, let us analyze our algorithm in more details.

A basic yet crucial property of our algorithm is the finite convergence of the support set sequence. Like \cite[Lemma 2]{Chao2018An}, since the sequence \{$\Omega_0^k$\} satisfies $\Omega_0^k \subseteq \Omega_0^{k+1}$ and $0 \leq \#\Omega_0^k \leq N$, it converges within a finite number of iterations, i.e., there exists $K$, such that
\begin{equation}\label{equ_unchangeconstaint}
\Omega_0^k =\overline{\Omega}_0:=\Omega_0^K \  and \  \Omega_1^k =\overline{\Omega}_1:=\Omega_1^K, \quad \forall k \geq K.
\end{equation}

We now write the first order optimality condition of 	$(\mathcal{G}_k^{\delta})$ for $k \geq K$, which will be used in the proofs of Theorem \ref{theorem_lowerbound} and Lemma \ref{lemma_subgradient_bound}. When $k \geq K$, we have $\Omega_1^k=\overline{\Omega}_1$, $\Omega_0^k=\overline{\Omega}_0$, $\mathcal{C}_k=\mathcal{C}_K$ and $\|D_iu^{k}\| \neq 0, \forall i \in \overline{\Omega}_1$. By a similar reasoning as the proof of Lemma \ref{lemma-subdifferential}, we obtain
$\partial G_k^{\delta}(u,v)=W_1 \otimes W_2$ with	
\begin{numcases}{}
W_1 = \alpha\sum_{ i \in  \overline{\Omega}_1} \phi'(\|D_i u^k\|)\frac{{D_i}^TD_i u}{\|D_i u\|} + (u+v-f) + \rho(u-u^k) + \left(\cap_{i \in  \overline{\Omega}_0}\ker D_i\right)^\bot, \nonumber \\
W_2 = \left\{ (u+v-f) + \beta H^THv + \gamma v \right\}. \nonumber
\end{numcases}
As $(u^{k+1},v^{k+1})$ solves $(\mathcal{G}_k^{\delta})$, we clearly have
\begin{equation} \label{equ_optimalityK}
\left\{
\begin{aligned}
&\alpha\sum_{ i \in  \overline{\Omega}_1} \phi'(\|D_i u^k\|)\frac{{D_i}^TD_i u^{k+1}}{\|D_i u^{k+1}\|}
+ (u^{k+1}+v^{k+1}-f) + \rho(u^{k+1}-u^k)+ \left(\cap_{i \in  \overline{\Omega}_0} \ker D_i\right)^\bot \ni 0, \\
&(u^{k+1}+v^{k+1}-f) + \beta H^THv^{k+1} + \gamma v^{k+1} =0.
\end{aligned} \right.
\end{equation}

The following theorem shows a lower bound theory for the iteration sequence. It not only helps to overcome the non-Lipshitz difficulty in the convergence analysis, but also indicates in some sense that the algorithm generates good approximate image components suitable for thresholding.
\begin{theorem}\label{theorem_lowerbound}(Lower bound of $\|D_i u^{k}\|$)
	There exists a constant $\theta>0$ such that
	$$
	\mbox{either }\quad \|D_i u^{k}\|=0 \quad \mbox{ or } \quad \|D_i u^{k}\|>\theta,  \quad \forall k \geq K, \forall i \in J,
	$$
	with $K$ defined in \eqref{equ_unchangeconstaint}. Moreover, for each $i \in \overline{\Omega}_1$ and $k \geq K$, we have
	\begin{equation}\label{equ_lip}
	\Big|\phi'(\|D_i u^{k+1}\|)-\phi'(\|D_i u^k\|)\Big| \leq L_{\theta}\|D_i\| \|u^{k+1}-u^k\|,
	\end{equation}
	where $L_{\theta}$ is as defined in Assumption \ref{assum-phi}(\romannumeral3).
\end{theorem}
\begin{proof}
	Obviously, it suffices to prove that there exists a constant $\theta>0$ such that
	$$
	\|D_i u^{k}\|>\theta,  \quad \forall k \geq K, \forall i \in \overline{\Omega}_1.
	$$
	
	Since $(u^{k+1},v^{k+1})$ solves $(\mathcal{G}_k^{\delta})$ in \eqref{equ-unchange-approx-delta}, by the first formula in \eqref{equ_optimalityK}, we have
	$$
	0 \in \alpha\sum_{ i \in  \overline{\Omega}_1} \phi'(\|D_i u^k\|)\frac{{D_i}^TD_i u^{k+1}}{\|D_i u^{k+1}\|} + (u^{k+1}+v^{k+1}-f) + \rho(u^{k+1}-u^k) + \left(\cap_{i \in  \overline{\Omega}_0}\ker D_i\right)^\bot.	
	$$
	Therefore, for any $w \in \cap_{i \in \overline{\Omega}_0}\ker D_i$, one deduces that
	$$
	\left\langle \alpha\sum_{ i \in  \overline{\Omega}_1} \phi'(\|D_i u^k\|)\frac{{D_i}^TD_i u^{k+1}}{\|D_i u^{k+1}\|} + (u^{k+1}+v^{k+1}-f) + \rho(u^{k+1}-u^k), w \right\rangle = 0.
	$$
	This indicates
	\begin{align*}
	\alpha\sum_{ i \in \overline{\Omega}_1}\phi'(\|D_i u^k\|)\left\langle \frac{D_i u^{k+1}}{\|D_i u^{k+1}\|}, D_i w \right\rangle
	= & \left\langle f-u^{k+1}-v^{k+1}, w \right\rangle + \rho \left\langle u^k-u^{k+1}, w \right\rangle   \\
	\leq & \left( \|f\|+\|u^{k+1}\|+\|v^{k+1}\| \right) \|w\| +\rho \|u^k-u^{k+1}\|\|w\|.
	\end{align*}
	As $\|u^k-u^{k+1}\|\rightarrow 0$ and $\{(u^k,v^k)\}$ is bounded (Lemma \ref{lemma_bounded}), there exists $\delta>0$, which is independent of $k$, such that
	\begin{equation}\label{proof_upper}
	\sum_{ i \in \overline{\Omega}_1}\phi'(\|D_i u^k\|)\left\langle \frac{D_i u^{k+1}}{\|D_i u^{k+1}\|}, D_i w \right\rangle
	\leq \delta \|w\|.
	\end{equation}
	From \eqref{proof_upper}, the existence of a positive lower bound of nonzero $\|D_i u^k\|$ can be established by a similar argument as in
	\cite[Theorem 1]{Chao2018An}, or \cite{zhenzhe2019}, by using the tool of ``Conservativeness of image gradient fields" developed in \cite{zeng2018edge}. That is to say, there exists $\theta>0$ such that
	\begin{equation}\label{equ_lowerbounds}
	\|D_i u^k\| > \theta,\ \forall k \geq K, \forall i \in \overline{\Omega}_1.
	\end{equation}
	
	Combining \eqref{equ_lowerbounds} with Assumption \ref{assum-phi}(\romannumeral3), we obtain that, for each $i \in \overline{\Omega}_1$ and $k \geq K$,
	$$
	\Big|\phi'(\|D_i u^{k+1}\|)-\phi'(\|D_i u^k\|)\Big| \leq L_{\theta} \Big| \|D_i u^{k+1}\|-\|D_i u^k\| \Big|
	\leq L_{\theta}\|D_i\| \|u^{k+1}-u^k\|.
	$$
\end{proof}

\begin{lemma}(A subgradient lower bound for the iterates gap)\label{lemma_subgradient_bound}
	There exists a constant $\Gamma >0$, and for each $k \geq K$, there exists $\tau^{k+1} \in \partial	F(u^{k+1},v^{k+1})$, such that
	\begin{equation}\label{equ_gradientbound}
	\|\tau^{k+1}\| \leq \Gamma \|(u^{k+1},v^{k+1})-(u^{k},v^{k})\|.
	\end{equation}
\end{lemma}
\begin{proof}
	By Lemma \ref{lemma-subdifferential}, for any $g \in \partial F(u^{k+1},v^{k+1})$, we have $g=(g_u,g_v)$, where
	\begin{equation} \label{equ_gradient}
	\left\{
	\begin{aligned}
	g_u & \in \alpha\sum_{ i \in  \overline{\Omega}_0}(\ker D_i)^\bot + \alpha\sum_{ i \in  \overline{\Omega}_1} \phi'(\|D_i u^{k+1}\|)\frac{{D_i}^TD_i u^{k+1}}{\|D_i u^{k+1}\|} + (u^{k+1}+v^{k+1}-f), \\
	g_v & = (u^{k+1}+v^{k+1}-f) + \beta H^THv^{k+1} + \gamma v^{k+1}.
	\end{aligned} \right.
	\end{equation}	
	When $k \geq K$, it follows from \eqref{equ_optimalityK} and $\left(\cap_{i \in  \overline{\Omega}_0} \ker D_i\right)^\bot=\sum_{ i \in  \overline{\Omega}_0}(\ker D_i)^\bot$ that,
	\begin{equation} \label{equ_gradient0}
	\left\{
	\begin{aligned}
	& -\alpha\sum_{ i \in  \overline{\Omega}_1} \phi'(\|D_i u^{k}\|)\frac{{D_i}^TD_i u^{k+1}}{\|D_i u^{k+1}\|}
	- (u^{k+1}+v^{k+1}-f) - \rho(u^{k+1}-u^{k}) \in \alpha\sum_{ i \in  \overline{\Omega}_0}(\ker D_i)^\bot, \\
	& (u^{k+1}+v^{k+1}-f) + \beta H^THv^{k+1} + \gamma v^{k+1} =0.
	\end{aligned} \right.
	\end{equation}
	
	Denote $\tau^{k+1}=\begin{pmatrix} \tau_1^{k+1} \\ 0 \end{pmatrix}$, where
	$$\tau_1^{k+1}=\alpha\sum_{ i \in  \overline{\Omega}_1} \left( \phi'(\|D_i u^{k+1}\|)-\phi'(\|D_i u^{k}\|) \right) \frac{{D_i}^TD_i u^{k+1}}{\|D_i u^{k+1}\|}
	- \rho(u^{k+1}-u^{k}).$$
	Combining \eqref{equ_gradient} and \eqref{equ_gradient0}, we obtain $\tau^{k+1} \in \partial	F(u^{k+1},v^{k+1})$. Moreover,
	\begin{align*}
	\|\tau^{k+1}\| = & \|\tau_1^{k+1}\| \\
	\leq & \alpha\sum_{ i \in  \overline{\Omega}_1} \left| \phi'(\|D_i u^{k+1}\|)-\phi'(\|D_i u^{k}\|) \right| \frac{\|{D_i}^T\|\|D_i u^{k+1}\|}{\|D_i u^{k+1}\|}
	+ \rho\|u^{k+1}-u^{k}\|  \\
	\mbox{ [ by \eqref{equ_lip} ] } \leq & \alpha\sum_{i \in \overline{\Omega}_1} L_{\theta}\|D_i\|^2  \|u^{k+1}-u^{k}\| + \rho \|u^{k+1}-u^{k}\| \\
	\leq  & \left( \rho+\alpha L_{\theta} \sum_{i \in J} \|D_i\|^2\right)  \|u^{k+1}-u^{k}\|  \\
	=   & \Gamma  \|u^{k+1}-u^{k}\|  \\
	\leq   & \Gamma  \|(u^{k+1},v^{k+1})-(u^{k},v^{k})\|,  \\
	\end{align*}
	where $\Gamma:= \rho+\alpha L_{\theta} \sum_{i \in J} \|D_i\|^2$.
\end{proof}

Finally, we are able to show the global convergence result.

\begin{theorem}\label{theorem_convergence}(Global convergence)
	Suppose that Assumption \ref{assum-phi} holds and $F(u,v)$ is a KL function. Let $\{(u^k,v^k)\}$ be a sequence generated by ISSAPL-ID. Then $\{(u^k,v^k)\}$ converges
	to a point $(u^*,v^*)$ which is a critical point of $F(u,v)$.
\end{theorem}
\begin{proof}
	We need only prove the convergence of $\{(u^k,v^k)\}$ for $k \geq K$.
	
	Since $\{(u^k,v^k)\}$ is bounded (Lemma \ref{lemma_bounded}) and $F(u,v)$ is continuous, there exists a subsequence $\{(u^{k_j},v^{k_j})\}$ and $(\widetilde{u},\widetilde{v})$ such that
	\begin{equation}\label{equ_convergence}
	\{(u^{k_j},v^{k_j})\} \rightarrow (\widetilde{u},\widetilde{v}) \ \mbox{ and } \ F(u^{k_j},v^{k_j}) \to F(\widetilde{u},\widetilde{v}), \quad \mbox{ as } j \to \infty.
	\end{equation}
	
	Combining \eqref{equ_Fnonincrease}, \eqref{equ_gradientbound} and \eqref{equ_convergence}, by \cite[Theorem 2.9]{attouch2013convergence},
	we conclude that the sequence $\{(u^k,v^k)\}$ converges to $(u^*,v^*)=(\widetilde{u},\widetilde{v})$, and $(u^*,v^*)$ is a critical point of $F(u,v)$.
\end{proof}

\section{Algorithm implementation}
We present implementation details of ISSAPL-ID. We need to solve the problem $(\mathcal{G}_k)$ in \eqref{model1-shrink-approx}. It is strongly convex with linear constraints. It can be solved to any accuracy by numerous efficient and convergent algorithms like those in \cite{wang2008new,wu2010augmented,liang2015retinex}. We here elaborate on the ADMM. We set $\omega_{i}^{k}=\phi ^{'}(\Vert D_i u^k\Vert)$, $\forall i \in \Omega_{1}^{k}$ and introduce a new variable $q=(q_i)$, where $q_i \in \mathbb{R}^2$ and $i \in \Omega_{1}^{k}$. Then $(\mathcal{G}_k)$ in \eqref{model1-shrink-approx} is reformulated as
\begin{equation}\label{model1-split}
\left\{
\begin{aligned}
\min \limits_{u,v,q} \quad & \frac{1}{2}{\Vert f-u-v\Vert}^{2} + \alpha \sum\limits_{i\in\Omega_{1}^{k}}\omega_{i}^{k}\Vert q_i\Vert +\frac{\rho}{2}{\Vert u-u^{k}\Vert}^{2} + \frac{\beta}{2}{\Vert Hv\Vert}^{2} + \frac{\gamma}{2}{\Vert v\Vert}^{2}\\
\st \quad & q_i=D_i u,\ \ \ \forall i \in \Omega_{1}^{k},\\
& D_i u=0,\ \ \ \forall i \in \Omega_{0}^{k}.
\end{aligned} \right.
\end{equation}
The augmented Lagrangian function for the above constrained problem is defined as
\begin{align*}
L(u,v,q;\mu)= & \frac{1}{2}{\Vert f-u-v\Vert}^{2} + \alpha\sum\limits_{i\in\Omega_{1}^{k}}\omega_{i}^{k}\Vert q_i\Vert + \frac{\rho}{2}{\Vert u-u^k\Vert}^{2} + \frac{\beta}{2}{\Vert Hv\Vert}^{2} + \frac{\gamma}{2}{\Vert v\Vert}^{2} \\
& + \sum_{i\in \Omega_1^k} \langle \mu_i, D_i u-q_i \rangle + \sum_{i\in \Omega_0^k} \langle \mu_i, D_i u \rangle \\
& + \frac{r_1}{2}\sum_{i\in \Omega_1^k} \|D_i u-q_i\|^2 + \frac{r_2}{2}\sum_{i\in \Omega_0^k} \|D_i u\|^2,
\end{align*}
where $r_1,r_2>0$ are penalty parameters and $\mu_i \in \mathbb{R}^2$, $i \in J$ are lagrange multipliers. For convenience, we let $r_1=r_2=r$. Applying the ADMM yields the following algorithm.
\begin{mdframed}
	\textbf{ADMM:} the alternating direction method of multipliers for solving $(\mathcal{G}_k)$ in \eqref{model1-shrink-approx}
	\begin{enumerate}
		\item Input $u^k,\omega_i^k,\Omega_0^k,\Omega_1^k,r$. Initialize $(u^{k,0},v^{k,0})=(u^{k},v^{k}),\mu^{k,0}=0$.
		\item \label{step-ADMM}For $t=0,1,\dots$, compute
		\begin{numcases}{}
		q^{k,t+1}=\arg \min_{q}L(u^{k,t},v^{k,t},q;\mu^{k,t});  \label{sub_y}\\
		(u^{k,t+1},v^{k,t+1})=\arg \min_{u,v}L(u,v,q^{k,t+1};\mu^{k,t}); \label{sub_x}\\
		\mu^{k,t+1}_i = \begin{cases}
		\mu^{k,t}_i+ r D_iu^{k,t+1}, & \mbox{if } i \in \Omega_0^k;  \\
		\mu^{k,t}_i+ r (D_iu^{k,t+1}-q^{k,t+1}_i), & \mbox{if } i \in \Omega_1^k.
		\end{cases} \nonumber
		\end{numcases}
		
		Until a termination criterion is met.
		\item Output $(u^{k+1},v^{k+1})=(u^{k,t+1},v^{k,t+1})$.
	\end{enumerate}
\end{mdframed}

The two subproblems in the above algorithm are calculated as follows.
\begin{enumerate}
	\item The $q$-subproblem in \eqref{sub_y}: we can simplify \eqref{sub_y} as
	\begin{equation*}
	q^{k,t+1} = \arg \min_{q} \sum_{ i \in \Omega_1^k} \left( \alpha w^k_i \|q_i\| + \frac{r}{2} \left\| q_i -\left( D_iu^{k,t} + \frac{1}{r} \mu^{k,t}_i \right)\right \|^2 \right).
	\end{equation*}
	According to \cite{wu2010augmented}, the solution of the above problem is
	$$q_i^{k,t+1}=\max \left\{ 1 - \frac{\alpha w_{i}^{k}}{r\| D_iu^{k,t}+\frac{1}{r}\mu_i^{k,t} \|},\ 0 \right\}
	\left(D_iu^{k,t} + \frac{1}{r}\mu_i^{k,t}\right), \quad \forall i \in \Omega_{1}^{k}.$$
	
	\item The $(u,v)$-subproblem in \eqref{sub_x}: we introduce $\widetilde{q}^{k,t+1}=(\widetilde{q}^{k,t+1}_i)$ with $\widetilde{q}^{k,t+1}_i=0 \in \mathbb{R}^2, \forall i \in \Omega_{0}^{k}$ and define $\overline{q}^{k,t+1}=(q^{k,t+1},\widetilde{q}^{k,t+1})$. Then we can simplify \eqref{sub_x} as
	\begin{equation*}
	\begin{aligned}
	(u^{k,t+1},v^{k,t+1}) = \arg \min_{u,v} \bigg\{ & \frac{1}{2}{\Vert f-u-v\Vert}^{2} + \frac{\rho}{2}{\Vert u-u^k\Vert}^{2} + \frac{\beta}{2}{\Vert Hv\Vert}^{2} + \frac{\gamma}{2}{\Vert v\Vert}^{2} \\
	& + \frac{r}{2} \sum_{i\in J} {\Vert D_i u\Vert}^{2} -r \sum_{i\in J} \langle \overline{q}_i^{k,t+1}, D_i u \rangle +\sum_{i\in J} \langle \mu^{k,t}_i, D_i u \rangle \bigg\}.
	\end{aligned}
	\end{equation*}
	
	This is a quadratic optimization problem and its optimality condition gives a linear system 	
	\begin{equation} \label{equ_optimalx}
	\left\{
	\begin{aligned}
	& (1+\rho+r D^TD)u + v - f - \rho u^{k} + D^T(\mu^{k,t}-r \overline{q}^{k,t+1})=0, \\
	& u + (1+\gamma+\beta H^TH)v-f=0.
	\end{aligned}
	\right.
	\end{equation}
	Here we use the periodic boundary condition for the discrete difference, then \eqref{equ_optimalx} can be solved by the fast Fourier transforms
	(FFTs). One can refer to \cite{wang2008new,wu2010augmented} for calculation details.
	If the Neumann boundary condition is used, \eqref{equ_optimalx} can be solved by discrete cosine transforms (DCTs) (refer to \cite{ng1999fast}) or conjugate gradient (CG) method.
\end{enumerate}

\section{Experimental results} \label{sec_experiment}
In this section, we present our experimental results on comparing our method with several state-of-the-art approaches, i.e., four two-stage methods \cite{cai2013two,duan2015l_,chang2017new,Chan2018Convex} and one typical level set method \cite{li2011level}.
We implemented our algorithm in MATLAB R2016a, and the codes of the compared methods were provided by their authors. All the experiments are performed under Windows 10 and MATLAB R2016a running on a desktop (Intel(R) Core(TM) i7-6700 CPU @ 3.40GHz 3.40GHz, 8.00G RAM).

\subsection{The test images, compared methods and result assessment metrics}\label{sec_experiment_compared}
The test images include two synthetic ones in Section \ref{experiment1}, four simulated and real medical images in Section \ref{experiment2}, and 12 real brain MRI images involved in a 3D brain MRI dataset in Section \ref{experiment3}. The simulated and real images in Section \ref{experiment2} and Section \ref{experiment3}, as pointed out in \cite{duan2015l_,chang2017new,li2011level}, are assumed to be product approximations by some piecewise constant functions and smooth functions. That is, an observation $\overline{f}$ is written as
\begin{equation}\label{equ_product}
\overline{f}\approx\overline{u}\times \overline{v},
\end{equation}
with a piecewise constant $\overline{u}$ and a smooth $\overline{v}$.
In order to fit the addition model \eqref{equ_model1}, we therefore convert the image into the logarithmic domain, i.e.,
$$f=\text{log}(\overline{f}),u=\text{log}(\overline{u}),v=\text{log}(\overline{v}),$$
and \eqref{equ_product} becomes
$$f\approx u+v.$$
After solving the model \eqref{equ_model1}, we then reconstruct the piecewise constant part and the smooth part by an exponential transformation.

As in the literature, the five compared methods are abbreviated as CCZ \cite{cai2013two}, L0MS \cite{duan2015l_}, HoL0MS \cite{chang2017new} ,CNCS \cite{Chan2018Convex} and LIC \cite{li2011level} in the following. We mention that different methods have different applicabilities. The code of LIC provided by its authors does not apply to five-phase segmentation, thus LIC is not compared in the test in Figure \ref{fig1_geometry9fn}. L0MS and HoL0MS are two similar approaches and the latter is constructed in a 3D formulation, so HoL0MS is not compared in Section \ref{experiment1} and Section \ref{experiment2}, while L0MS is not compared in Section \ref{experiment3}. As for the brain segmentation test in Section \ref{experiment3}, we do not include LIC for comparison.

The results are compared in two aspects: the inhomogeneity correction results and the segmentation results. In addition to assessing the results visually, we quantitatively evaluate the results in Section \ref{experiment1} and Section  \ref{experiment3}, where ground truth for the test images are available. (Please note that the following widely used quantitative indices can be computed, only for tests with given ground truth.) Like \cite{duan2015l_,chang2017new}, we adopt the coefficient of variations (CV) \cite{likar2001retrospective} to measure the degree of intensity inhomogeneity in a region $\text{T}$, which is defined as
$$\text{CV(T)}=\frac{\sigma(\text{T})}{\mu(\text{T})},$$
where $\mu(\text{T})$ and $\sigma(\text{T})$ are the mean and the standard deviation of the intensities in $\text{T}$. A result with smaller CV value means a better inhomogeneity correction. The segmentation results are evaluated by the following Jaccard similarity (JS) metric \cite{shattuck2001magnetic}
$$\text{JS}(S_1,S_2)=\frac{|S_1 \cap S_2|}{|S_1 \cup S_2|} \times 100\%,$$
where $|\cdot|$ denotes the area of a region, $S_1,S_2$ are the region in the ground truth and the corresponding segmentation result by an algorithm, respectively. Clearly, a result with JS value closer to 1 means a better segmentation.

\subsection{The stopping conditions, parameter settings, and convergence behavior of our algorithm}

We first give stopping conditions. The compared five algorithms are terminated using their default stopping criterions or tuned for them to get good results. For our algorithm, the stopping condition for the inner loop is whether $\frac{\left\| (u^{k,t+1},v^{k,t+1})-(u^{k,t},v^{k,t}) \right\|}{\left\| (u^{k,t},v^{k,t}) \right\|}\leq \epsilon_{\text in}=10^{-4}$ or the ADMM iteration number reaching MAXIT\textunderscore in=100. The outer loop is stopped, if $\frac{\left\| (u^{k+1},v^{k+1})-(u^k,v^k) \right\|}{\left\| (u^k,v^k) \right\|}\leq \epsilon_{\text out}=10^{-4}$ or the iteration number reaches MAXIT\textunderscore out=10.

The parameter settings are as follows. In each method, there are several parameters to be tuned. More specifically, as pointed out in the literature, the scale parameter $\sigma$ and the time step $\Delta t$ of LIC \cite{li2011level}, the fidelity parameter $\lambda$ and the smooth parameter $\mu$ of CCZ \cite{cai2013two}, the regularization parameter $\alpha$ and an algorithm parameter $k$ of L0MS \cite{duan2015l_}, two model parameters $\alpha$ and $\mu$ of HoL0MS \cite{chang2017new}, the regularization parameters $\lambda$ and $T$ of CNCS \cite{Chan2018Convex}, need to be tuned. In all the following experiments, we adopt their default parameters (e.g., HoL0MS in Section \ref{experiment3}) or adjusted carefully these parameters to achieve best results, by CV and JS values (if computable, i.e., the ground truth is available) or visual effect. Note that, the above parameter symbols are directly taken from their papers. Thus one same symbol may correspond to different model or algorithm parameters in different methods.

We now discuss the parameter sensitivity and settings of our algorithm, where we use the widespread potential function: $\phi(t)=t^p, 0<p<1$. According to our experiments, the model parameter $\gamma$ and the algorithm parameter $\rho$ can be simply fixed as $\gamma=10^{-8},\ \rho=10^{-8}$. Figure \ref{fig0_sensitive} gives the sensitivity test of parameters $p, r, \alpha, \beta$ on the second synthetic test image in Figure \ref{fig1_geometry1fn} and the real MRI dataset in Figure \ref{fig3_testimages}. As shown in Figure \ref{fig0_sensitive}(a), our method is robust to $p$, and we therefore set $p=0.5$ in the following experiments. For different test images, our algorithm performs well with $r$ values in different but reasonably large intervals; see Figure \ref{fig0_sensitive}(b). As there is also a large common interval shown in Figure \ref{fig0_sensitive}(b), we set $r=10$ for all the following experiments. As shown in Figure \ref{fig0_sensitive}(c-d), the two model parameters $\alpha, \beta$ have large ranges (especially $\beta$) to give good results, but the range is dependent on specific images. In the following experiments, we tuned $\alpha, \beta$ for different images.

 \begin{figure}[H]
	\centerline{
		\begin{tabular}{c@{}c@{}c@{}c}
			\includegraphics[width=1.3in]{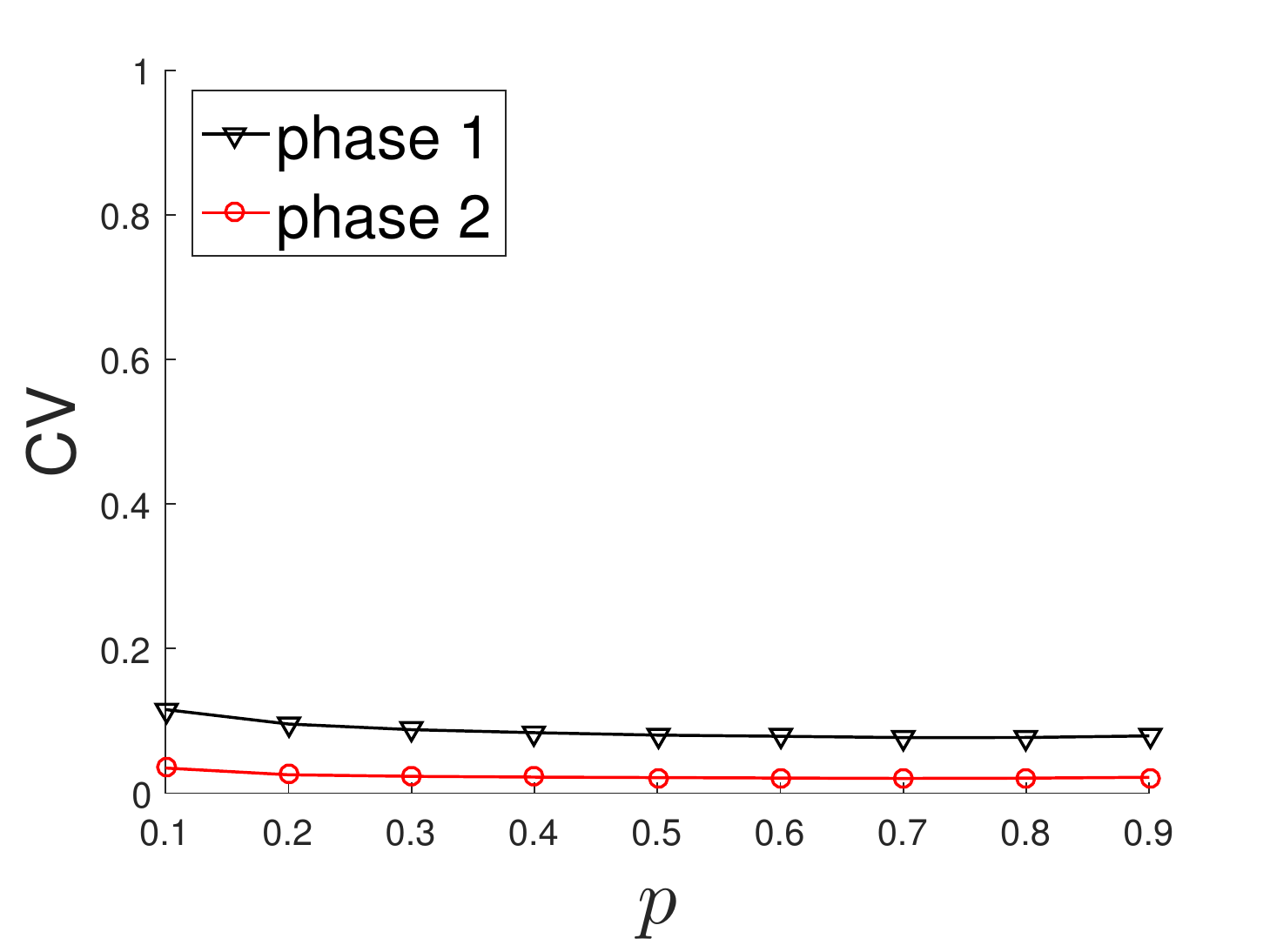} \ &				
			\includegraphics[width=1.3in]{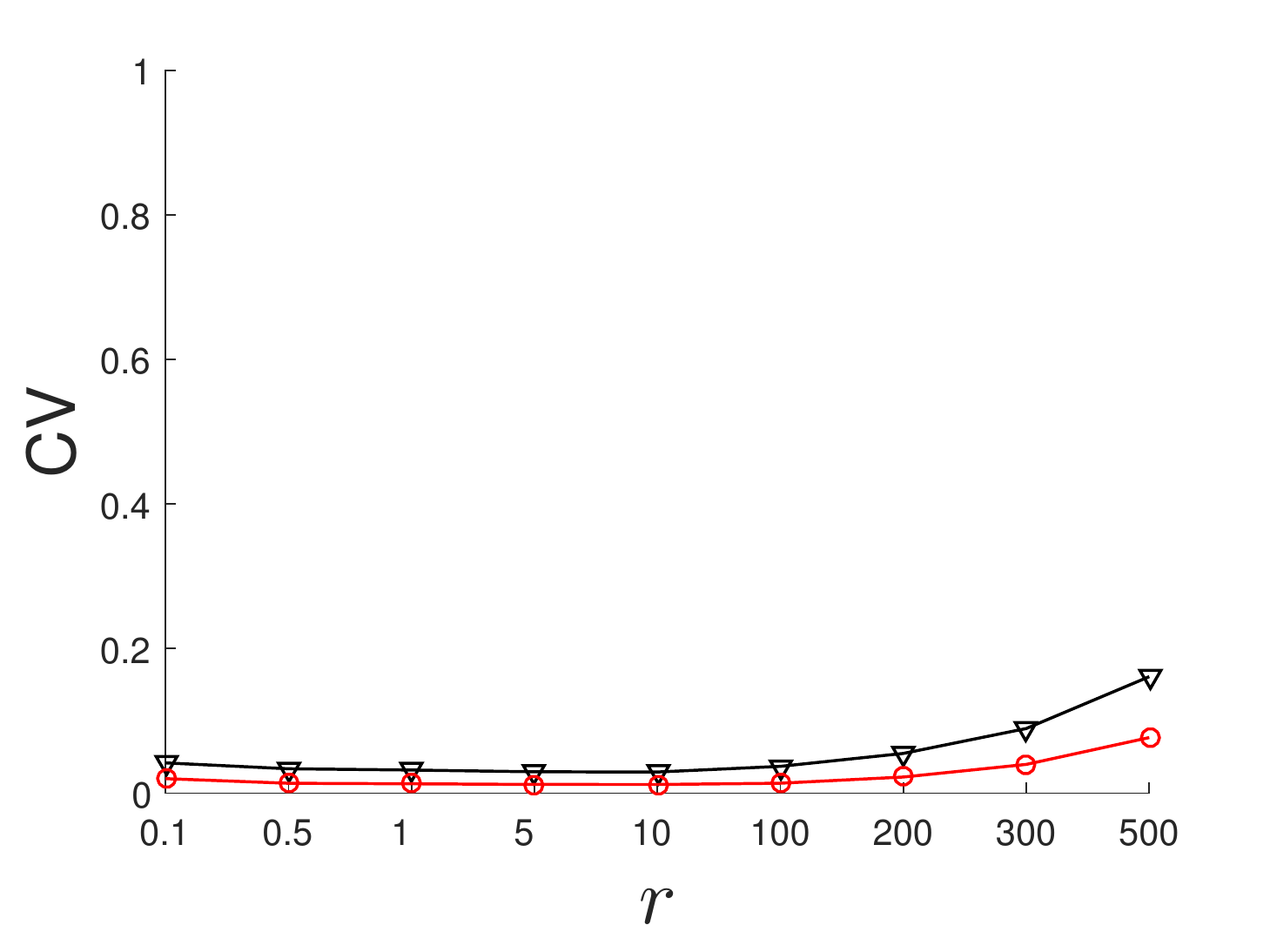} \ &				
			\includegraphics[width=1.3in]{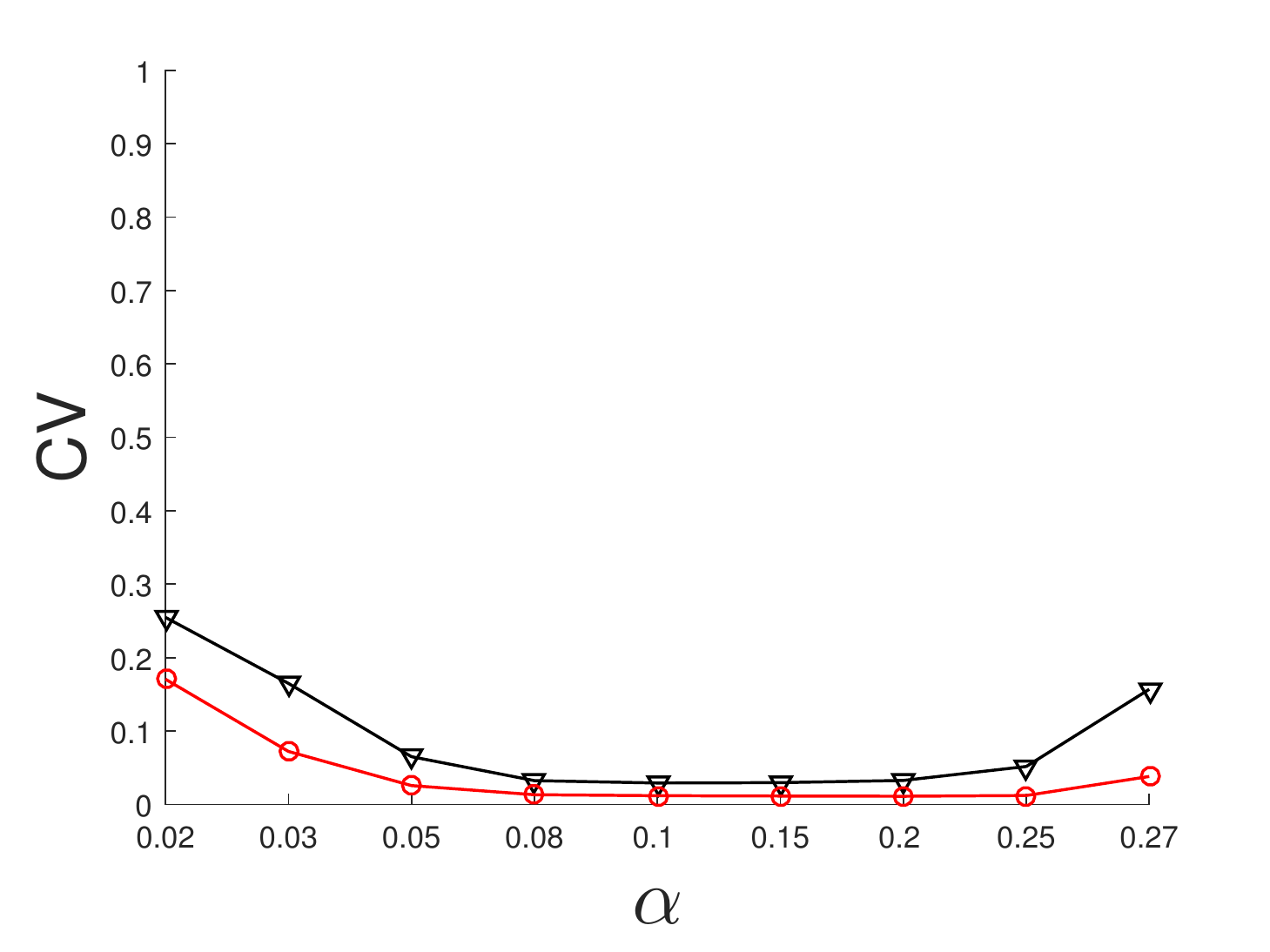} \ &				
			\includegraphics[width=1.3in]{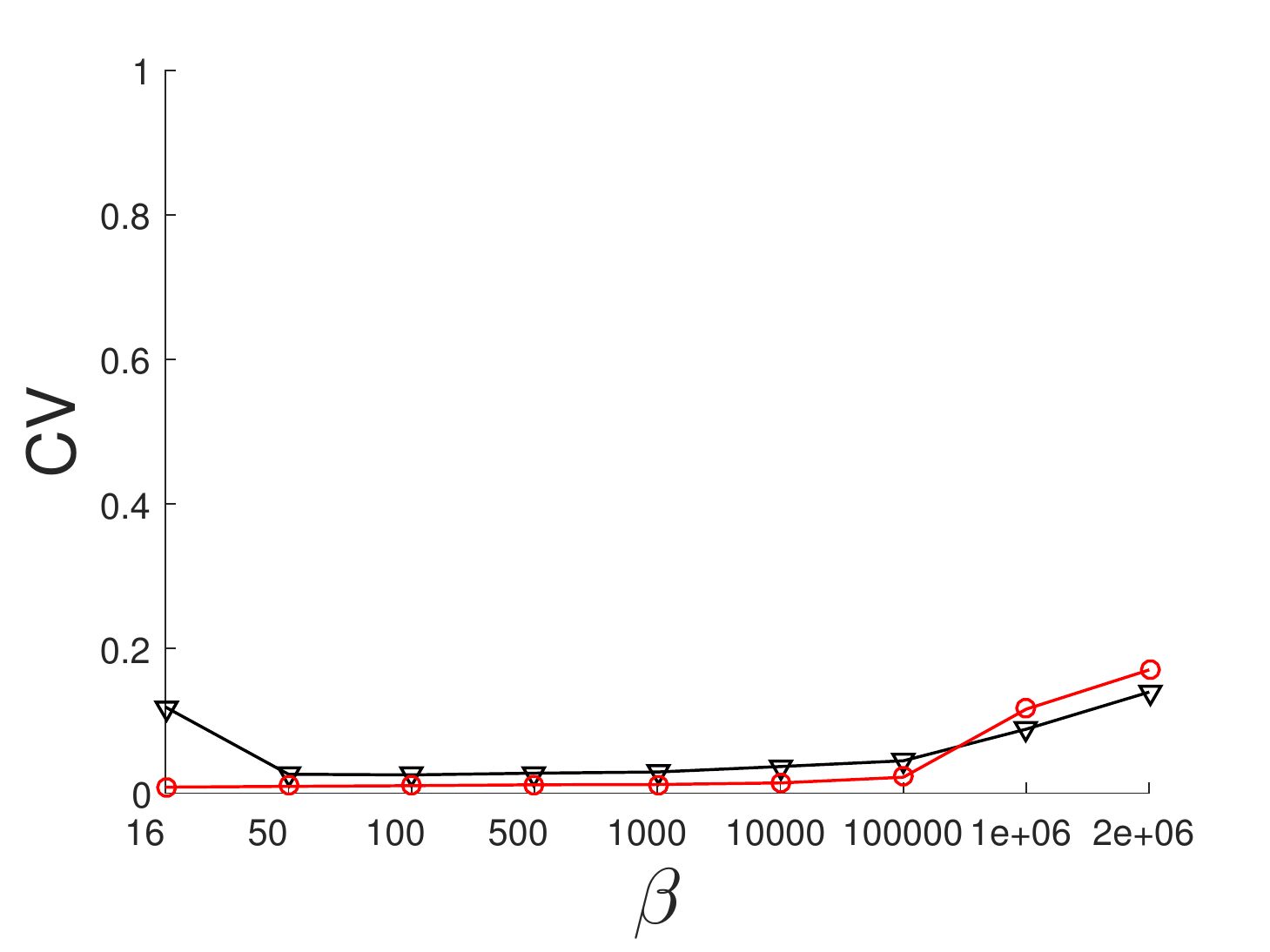} \\			
			\includegraphics[width=1.3in]{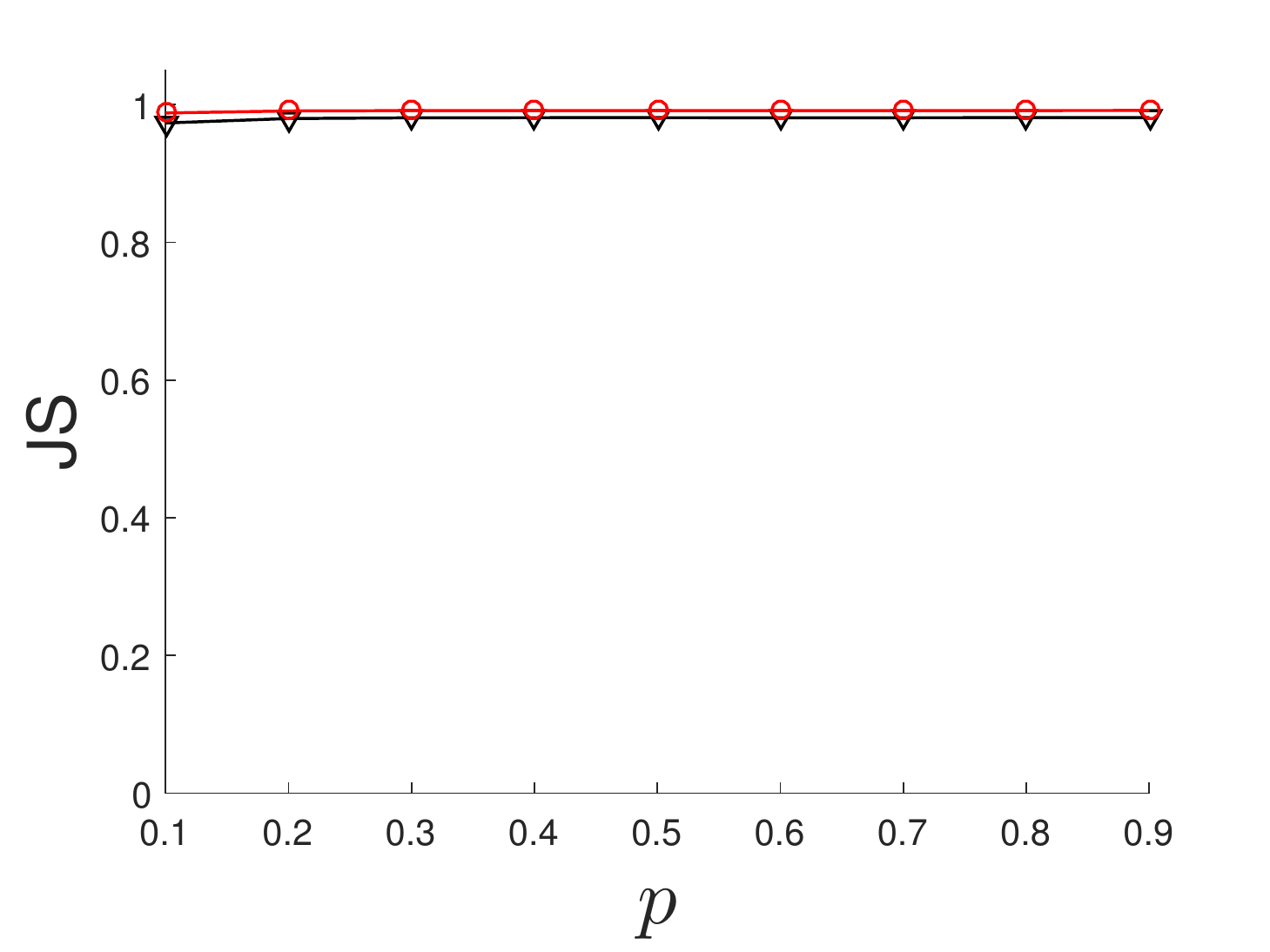} \ &
			\includegraphics[width=1.3in]{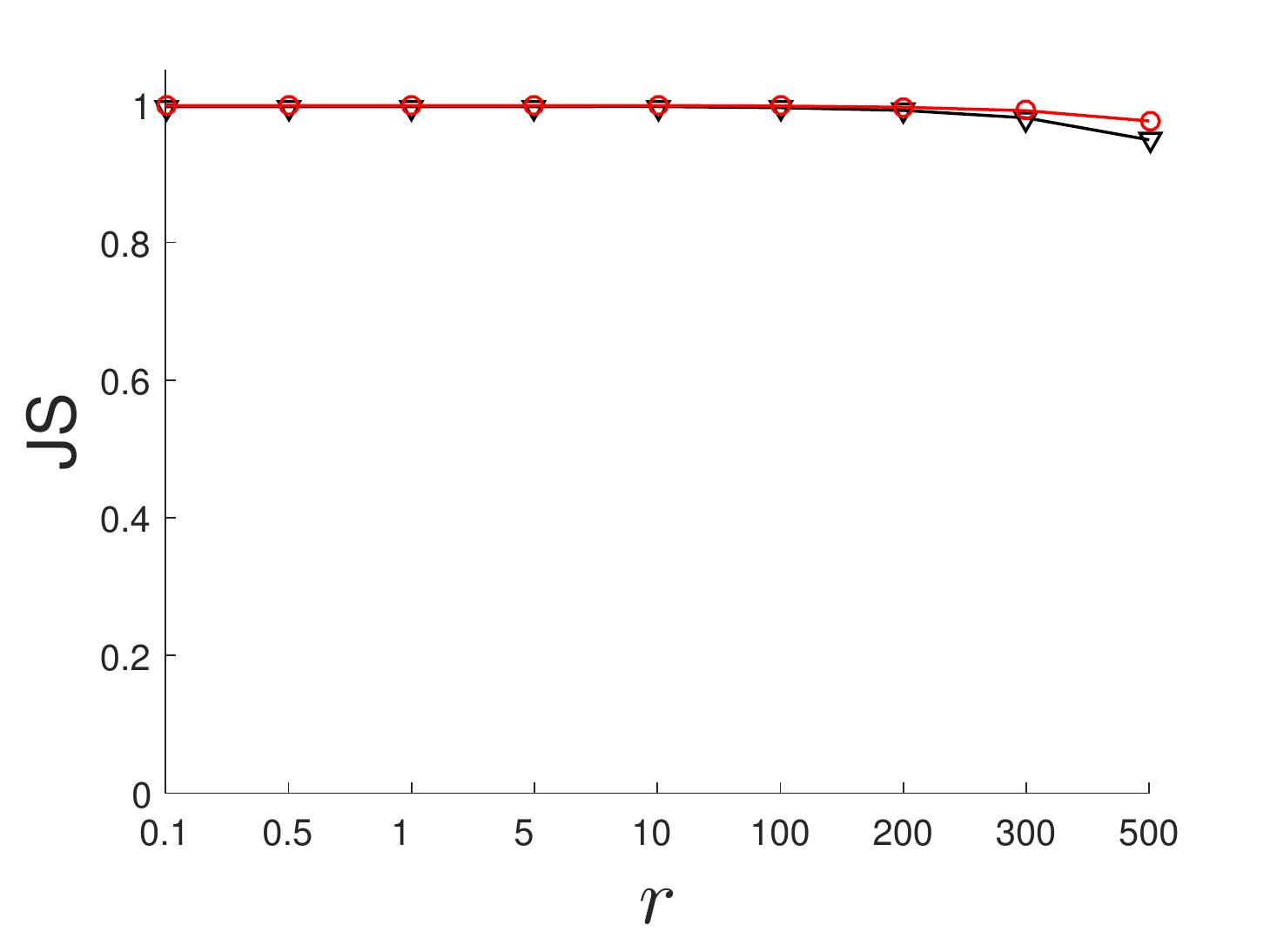} \ &
			\includegraphics[width=1.3in]{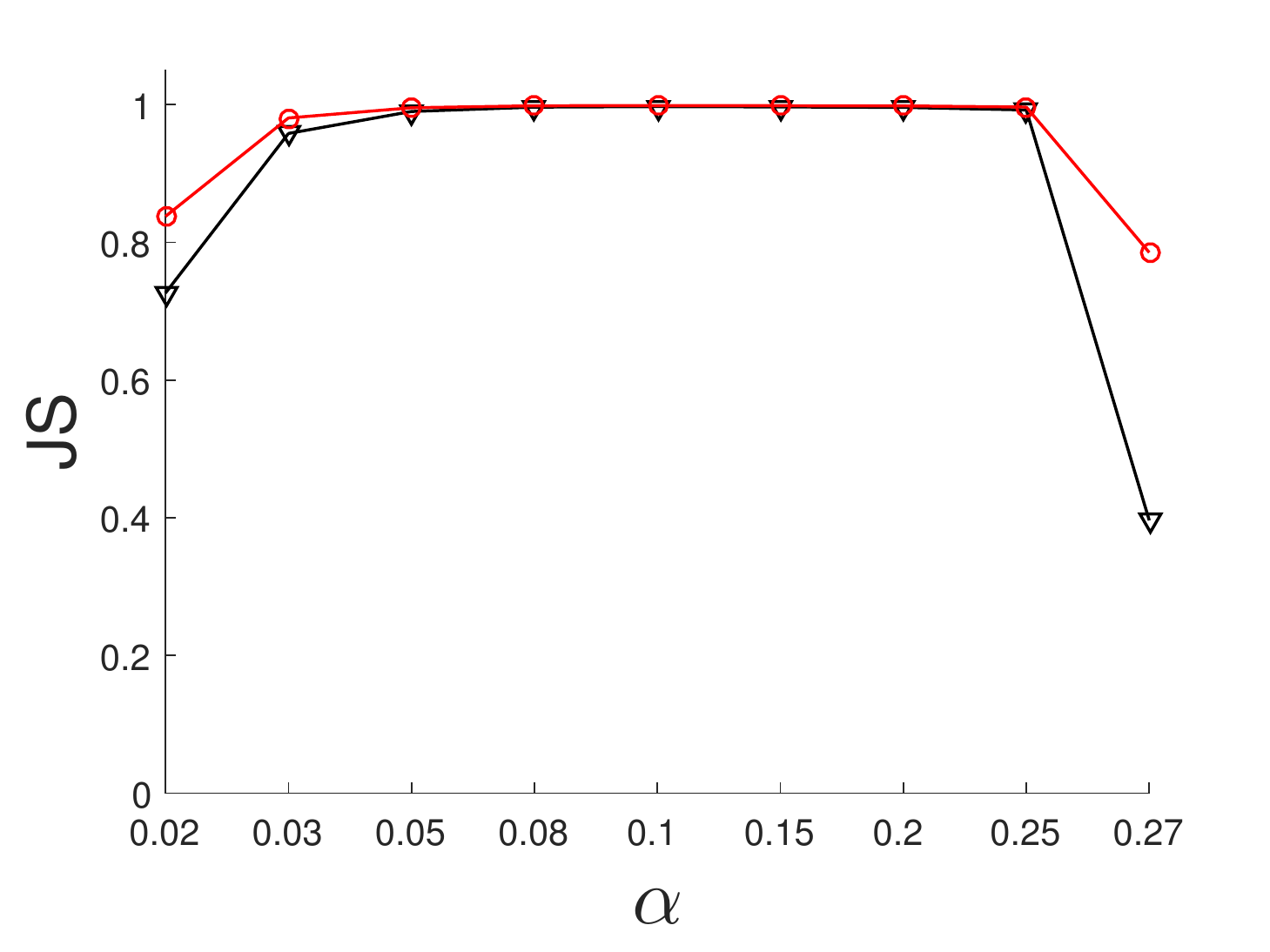} \ &
			\includegraphics[width=1.3in]{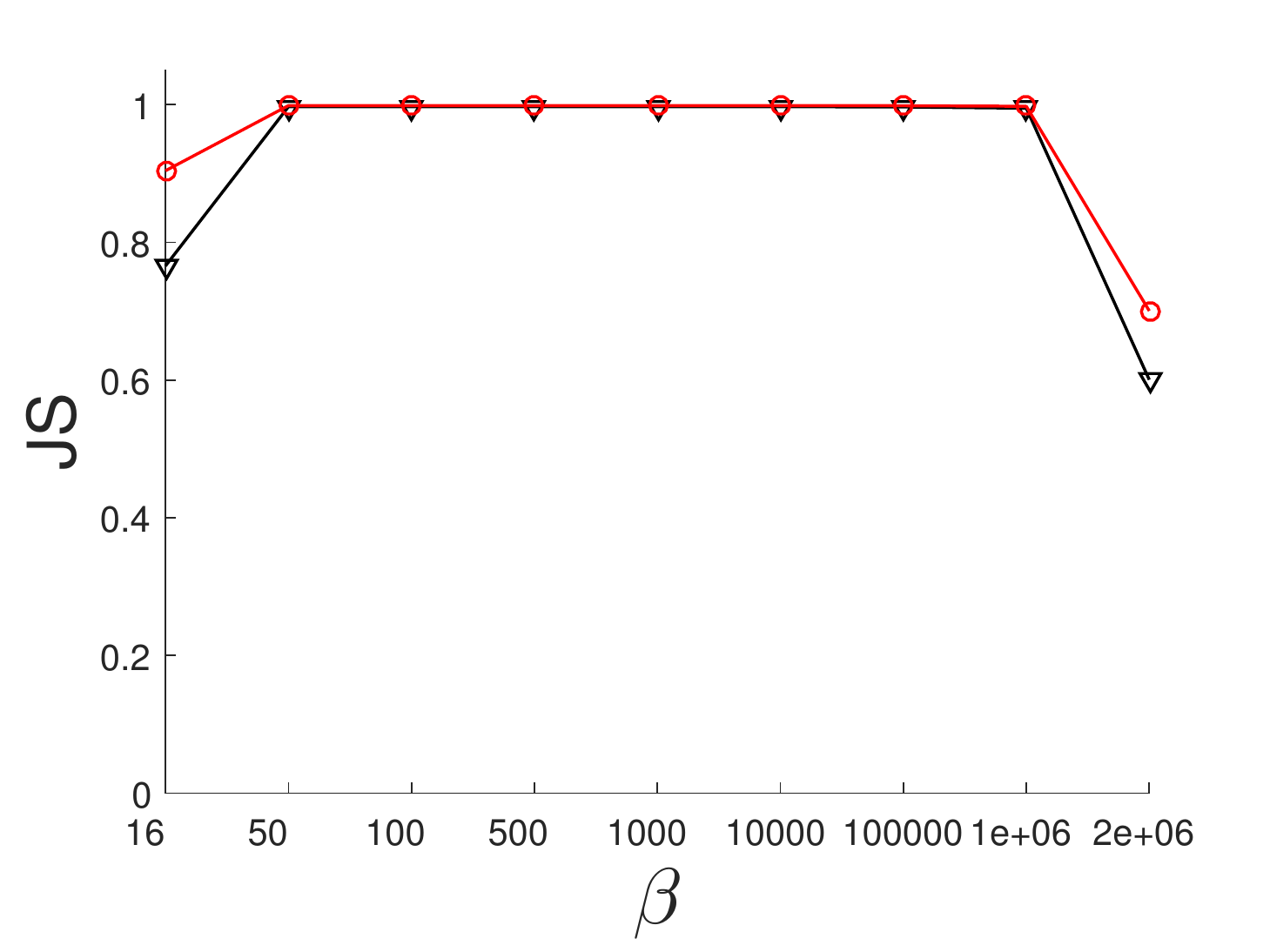} \\
			\includegraphics[width=1.3in]{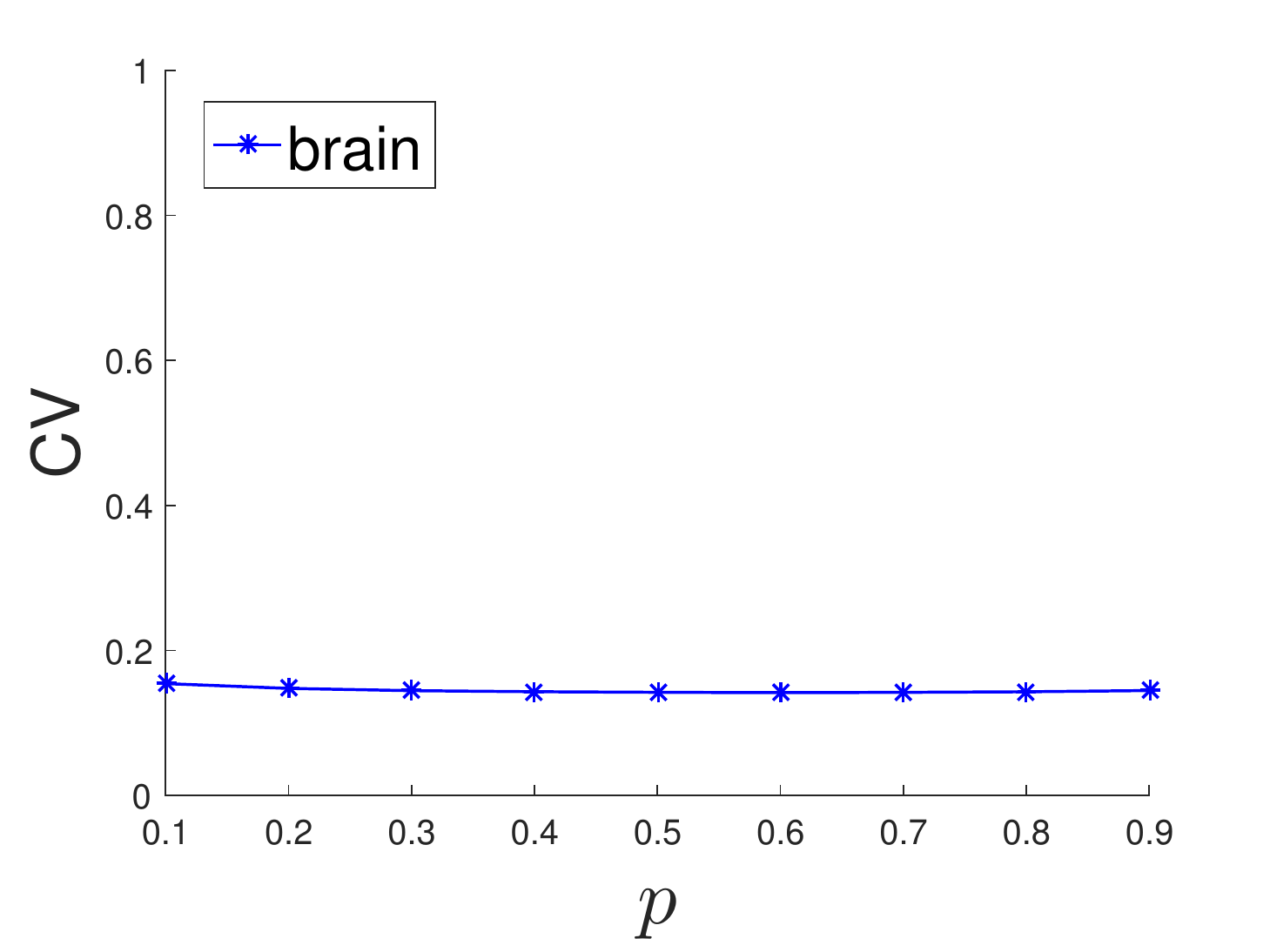} \ &				
			\includegraphics[width=1.3in]{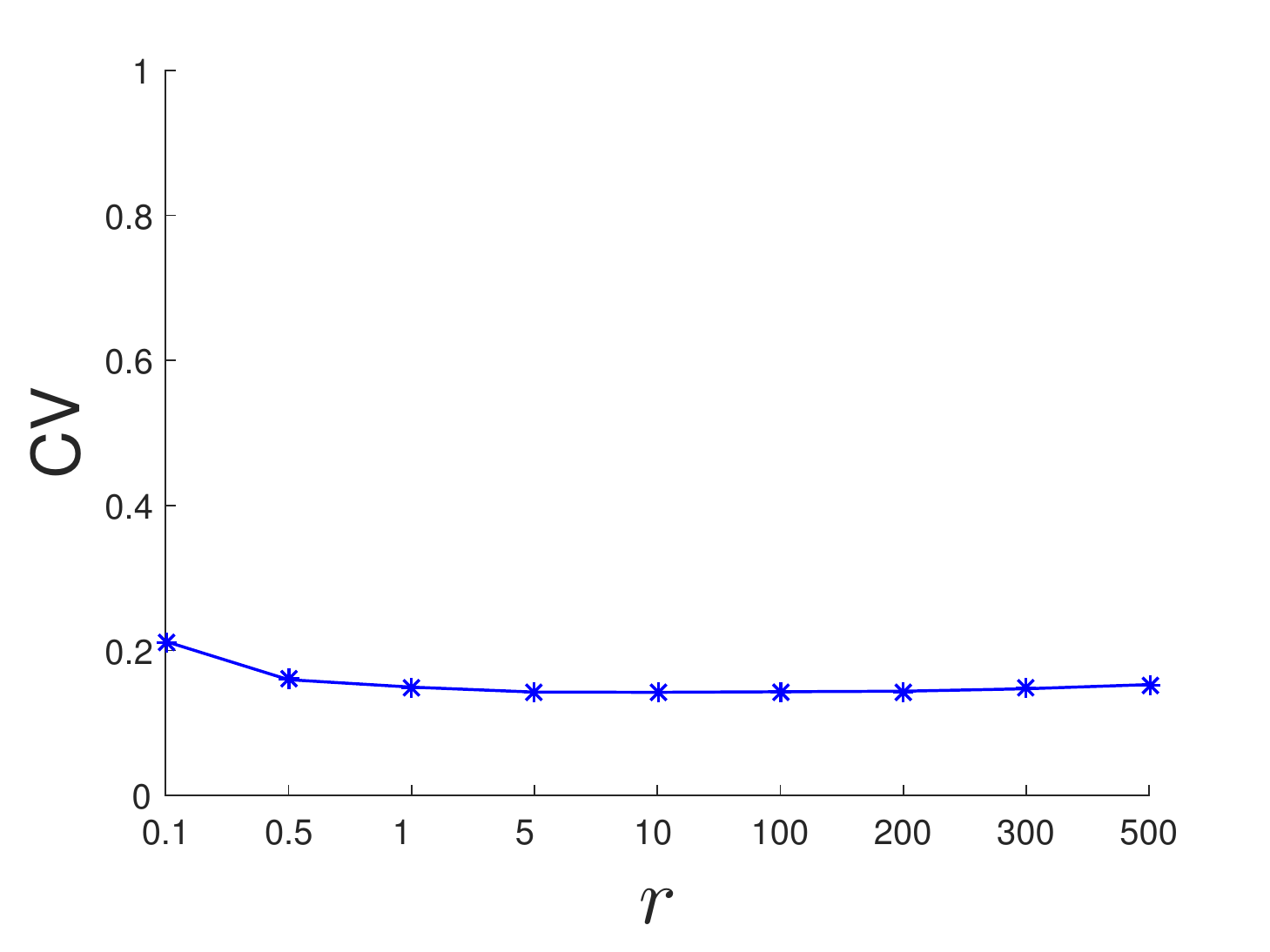} \ &				
			\includegraphics[width=1.3in]{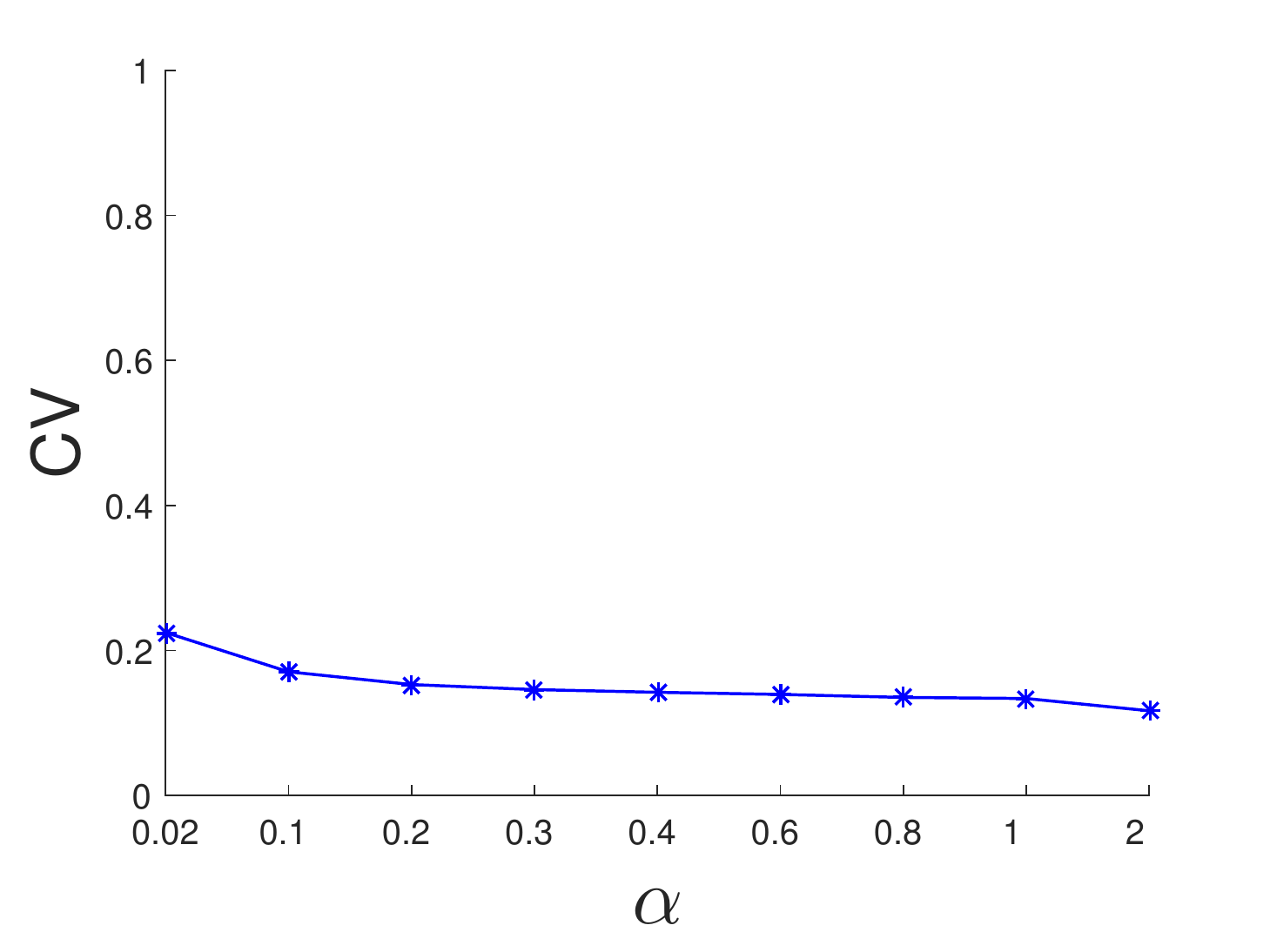} \ &	
			\includegraphics[width=1.3in]{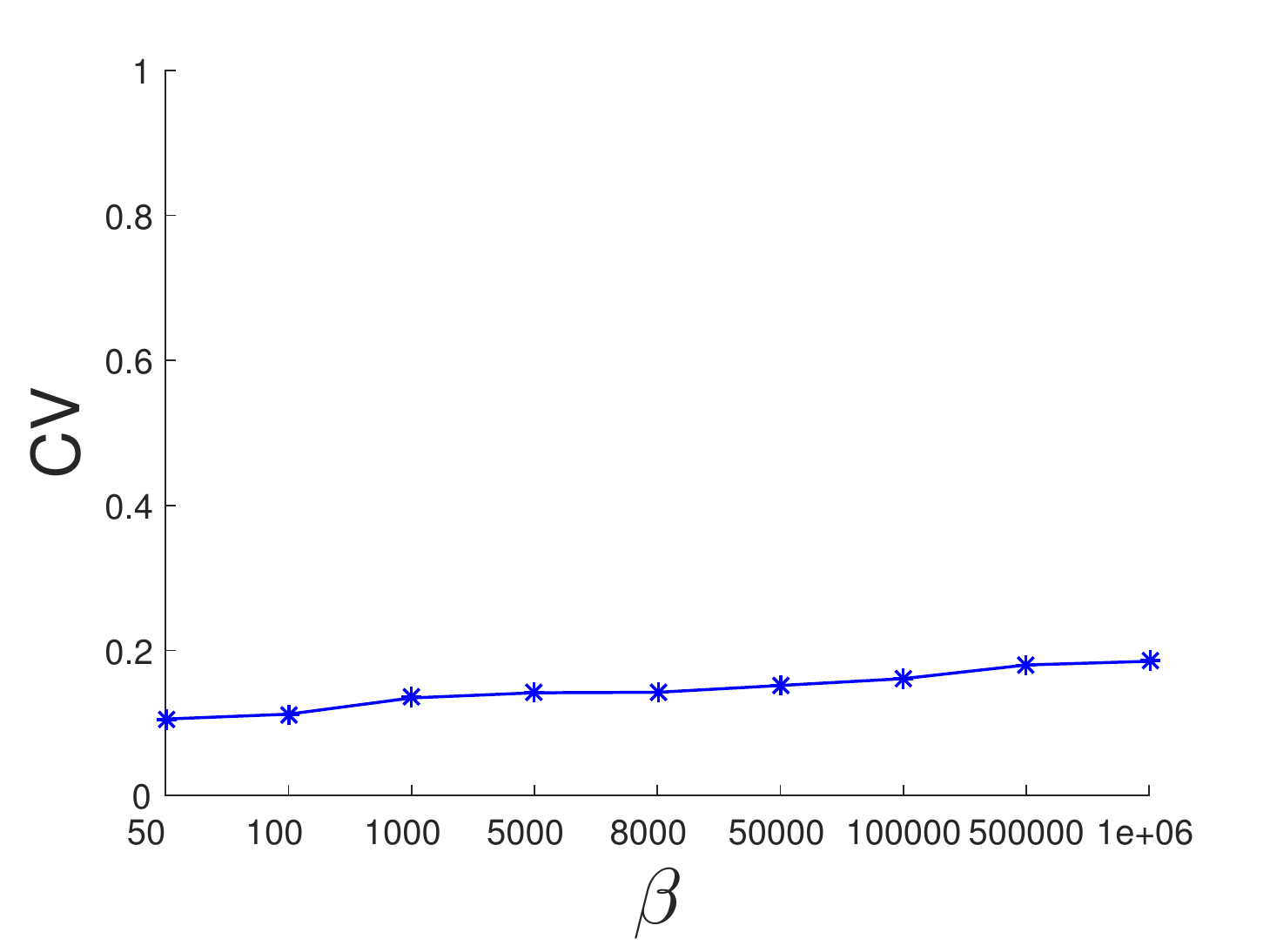} \\	 			
			\includegraphics[width=1.3in]{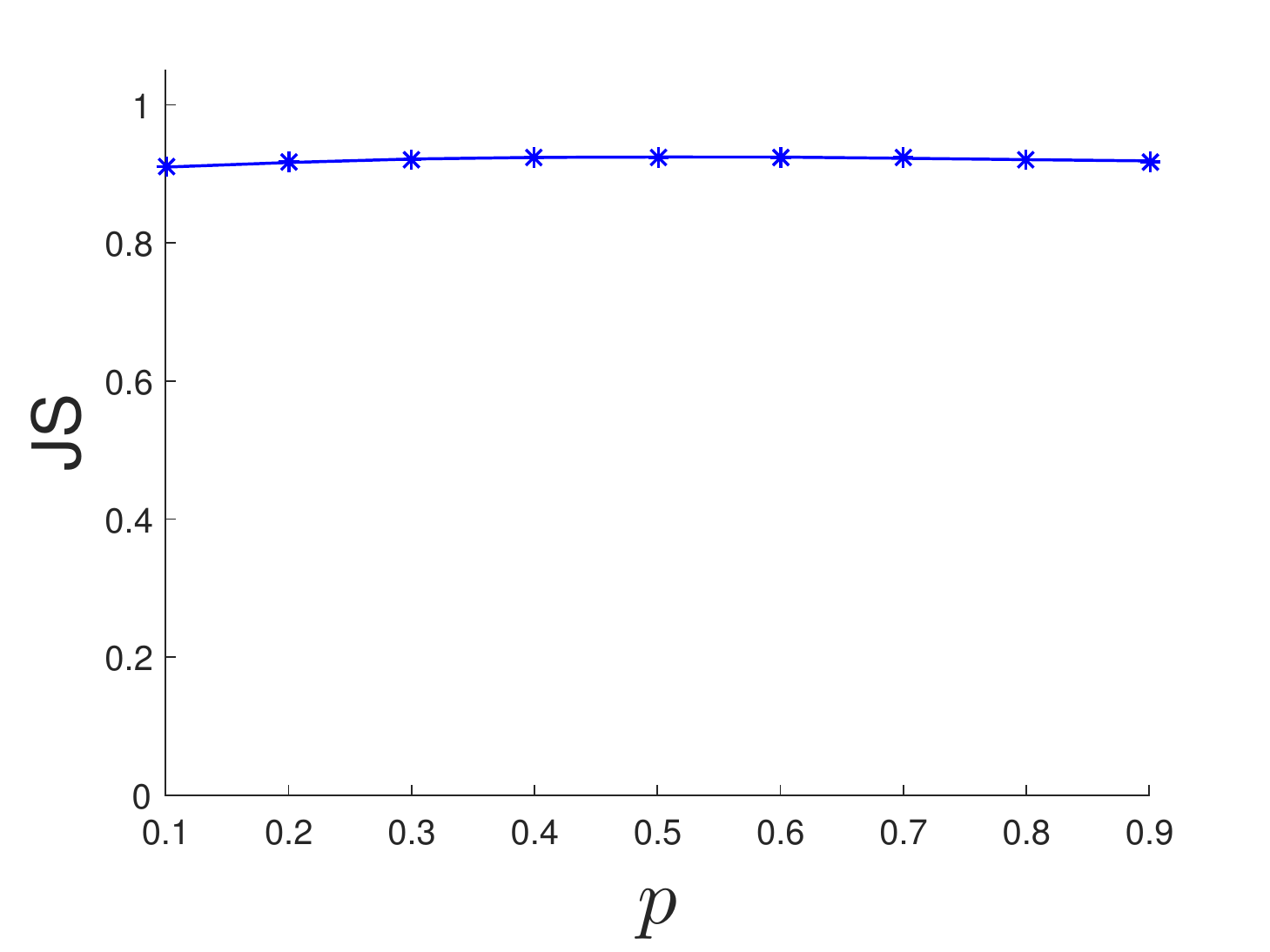} \ &
			\includegraphics[width=1.3in]{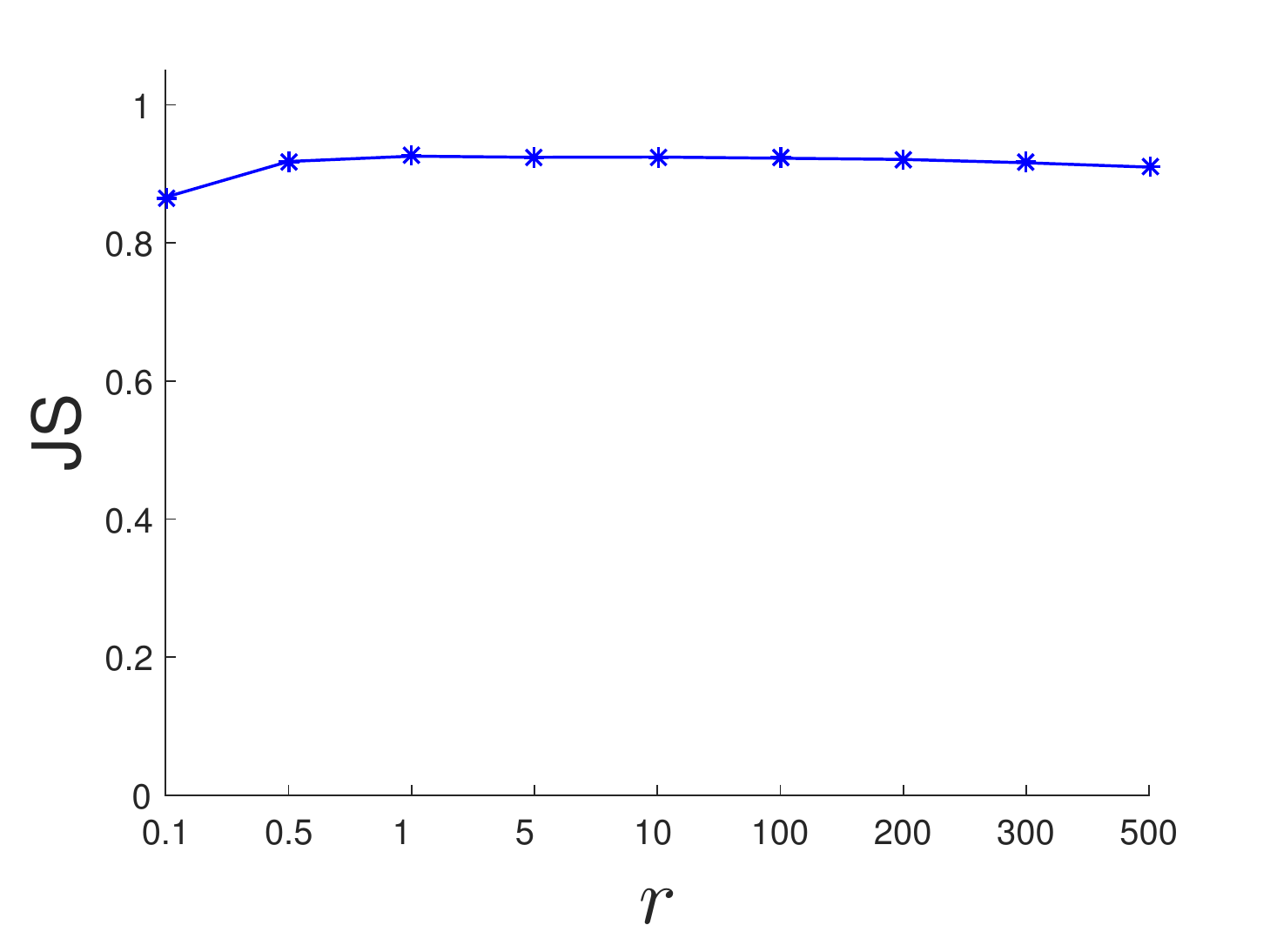} \ &
			\includegraphics[width=1.3in]{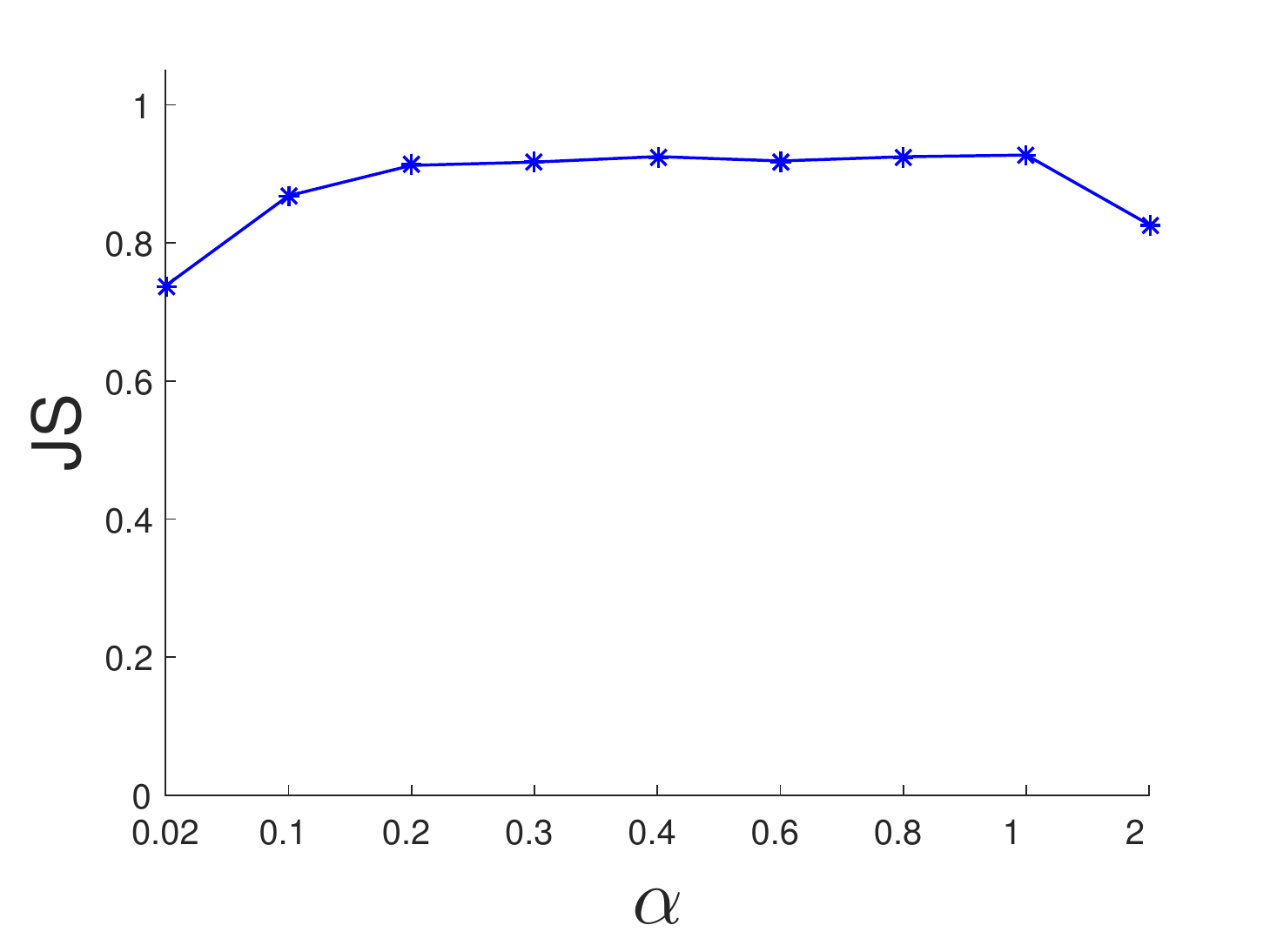} \ &
			\includegraphics[width=1.3in]{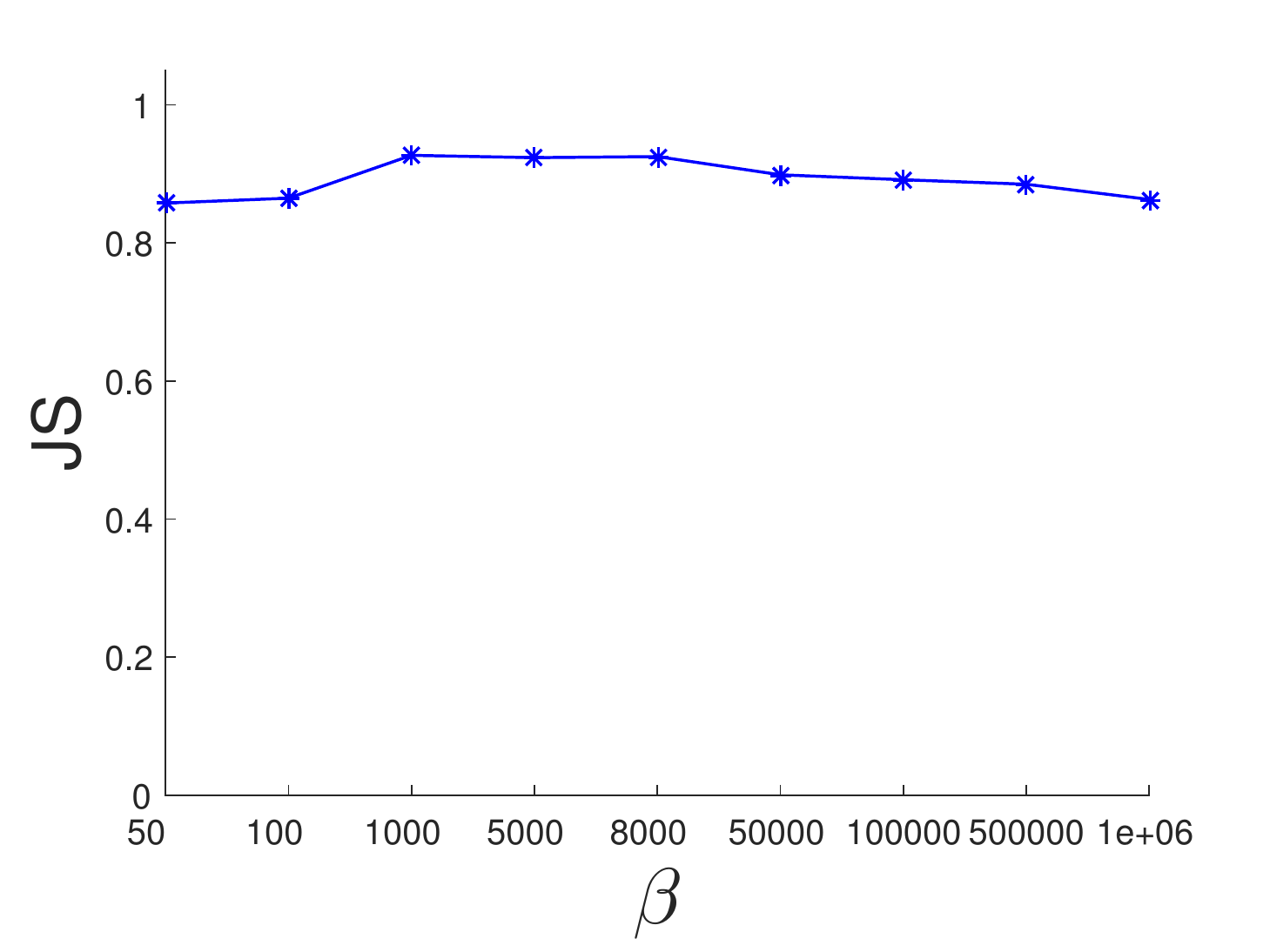} \\
			(a) & (b) & (c) & (d) \\
		\end{tabular}
	}
	\caption{\small\sl Sensitivity test of parameters $p,r,\alpha,\beta$ of our algorithm: CV values on the corrected images and JS values of the segmentation results with different values of the parameters. The first, second, third and fourth columns are the results for parameter $p,r,\alpha,\beta$, respectively. The first two rows: results on the second test image in Figure \ref{fig1_geometry1fn}. The last two rows: results (mean values over the 12 test images) on the dataset in Figure \ref{fig3_testimages}.}
	\label{fig0_sensitive}
\end{figure}

We check the convergence behavior of our algorithm by several tests, i.e., the second test image in Figure \ref{fig1_geometry1fn}, the second test image in Figure \ref{fig1_geometry9fn}, the second test image in Figure \ref{fig2_twophase} and the second test image in Figure \ref{fig2_mri14}. As shown in Figure \ref{fig0_convergence}, the evolution curves for the 4 tests are similar and demonstrated our theoretical analysis. The objective value $F(u^k,v^k)$ is decreasing and converges, which verifies the theoretical results in Lemma \ref{lemma_decrease}. The increments $\Vert (u^{k+1},v^{k+1})-(u^{k},v^{k}) \Vert$ converge to zero, which is consistent with Lemma \ref{lemma_bounded}. The support set sequence $\#\Omega_1(u^k)$ is monotonically decreasing and converges too.

\begin{figure}[H]
	\centerline{
		\begin{tabular}{c@{}c@{}c@{}c}
			\includegraphics[width=1.3in]{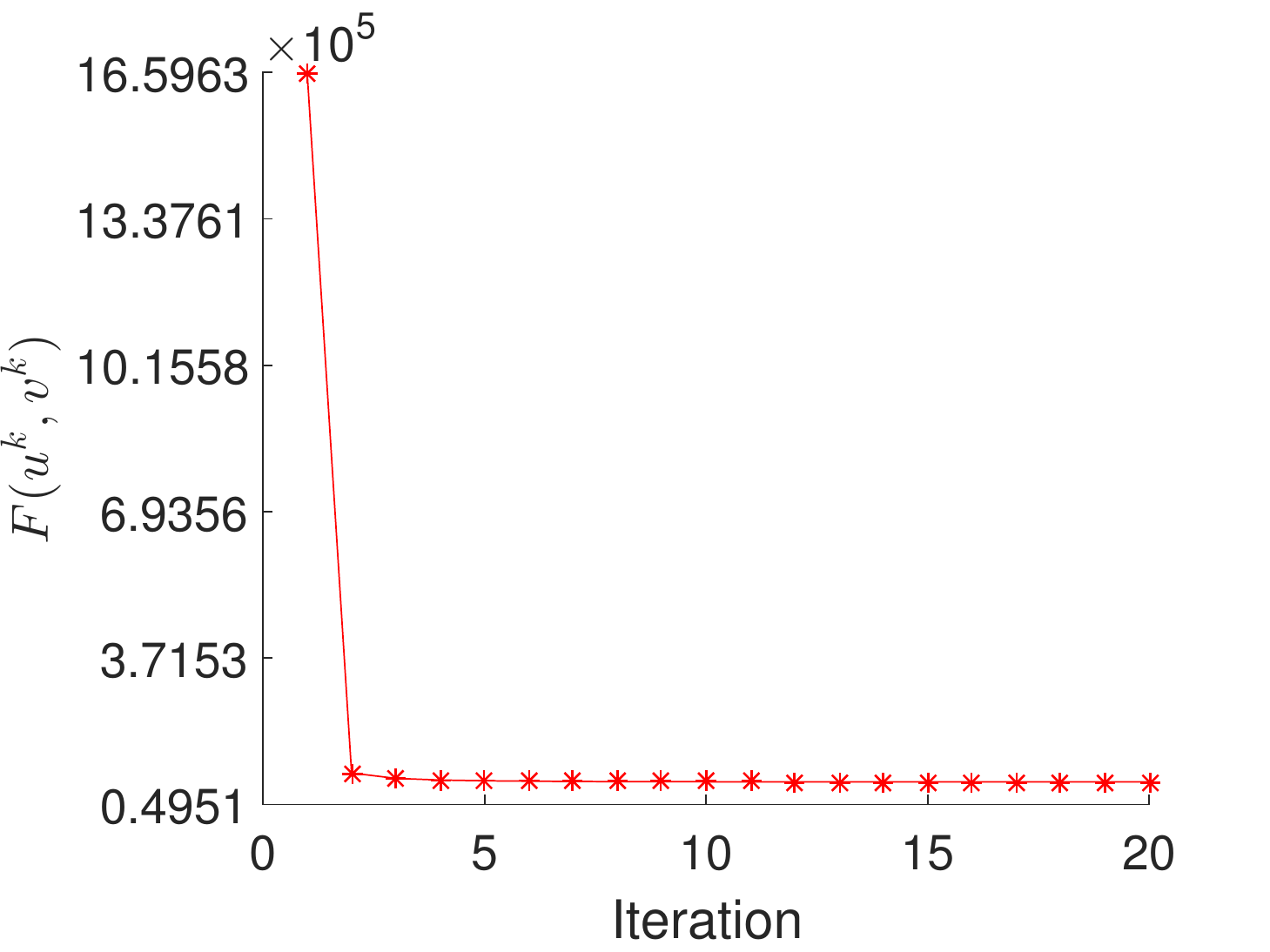} \ &	
			\includegraphics[width=1.3in]{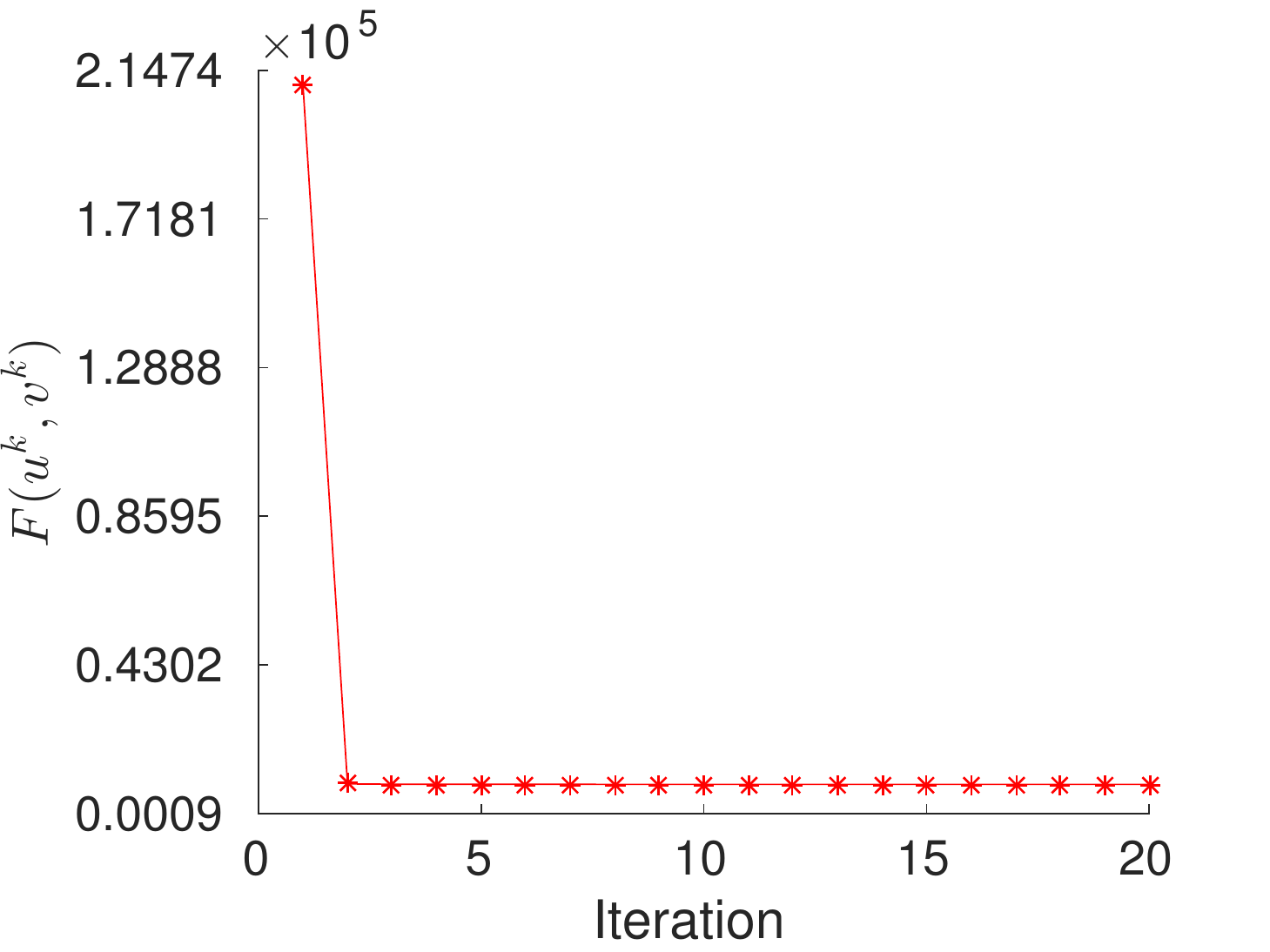} \ &
			\includegraphics[width=1.3in]{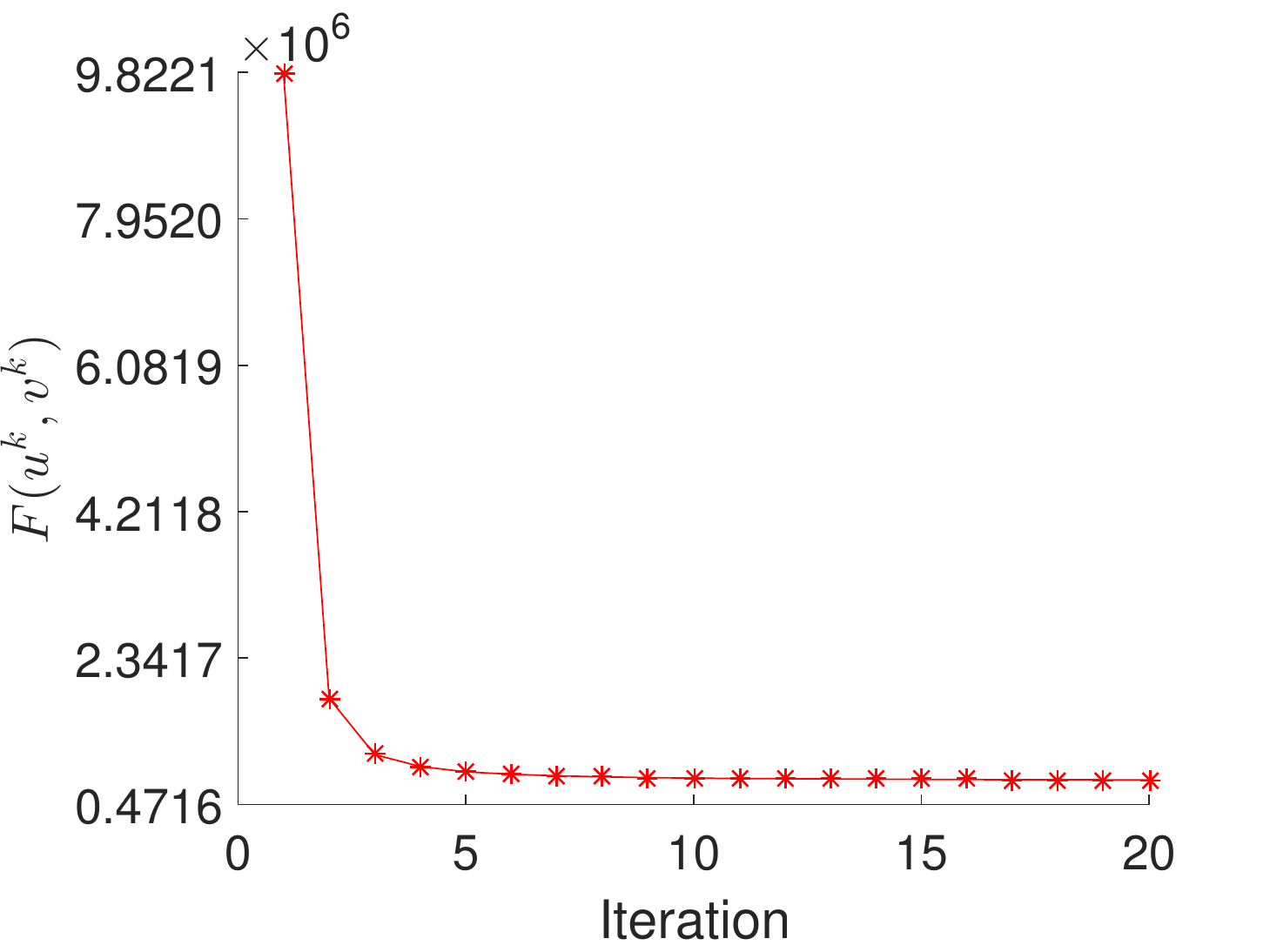} \ &
			\includegraphics[width=1.3in]{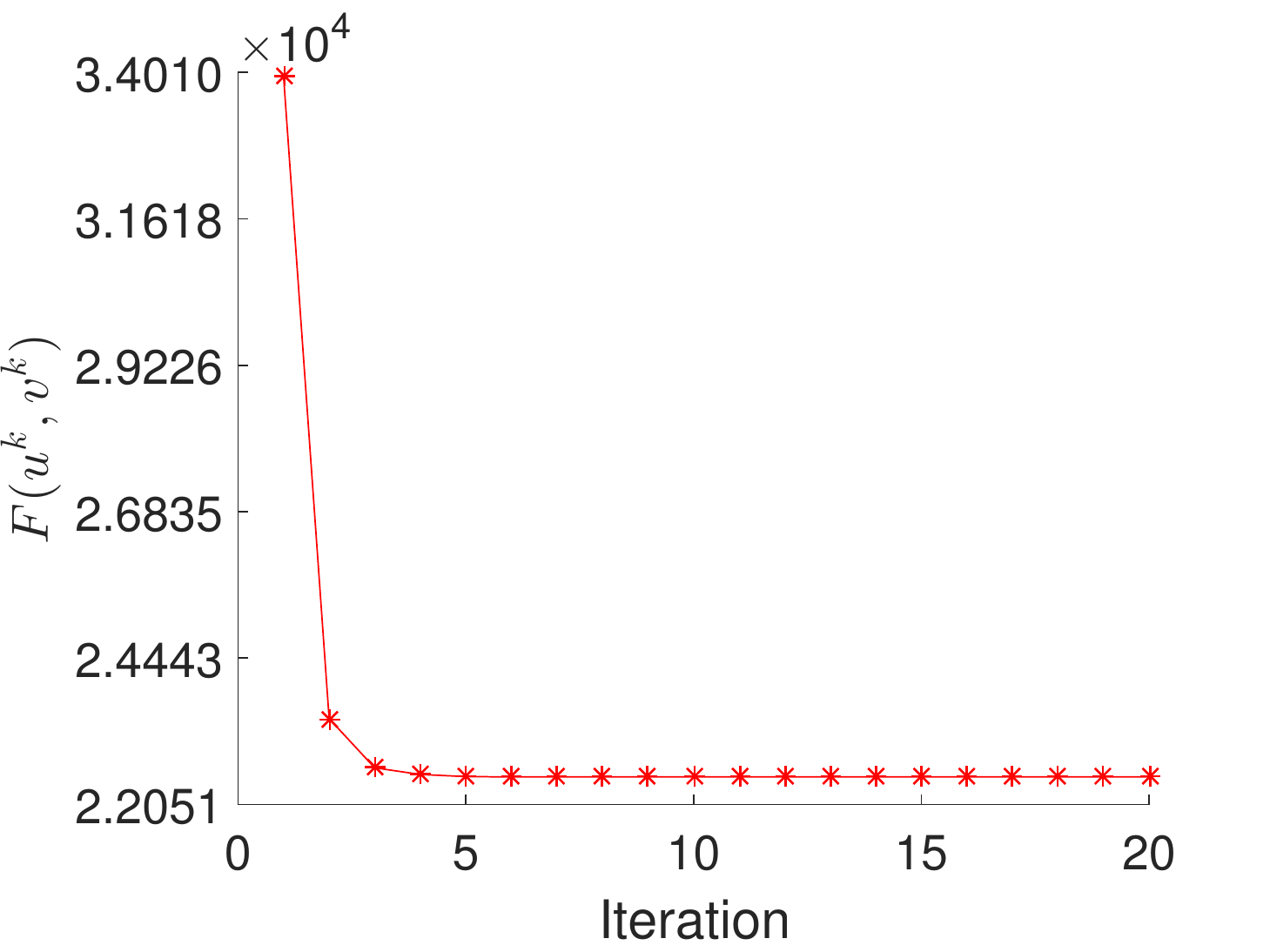} \\
			\includegraphics[width=1.3in]{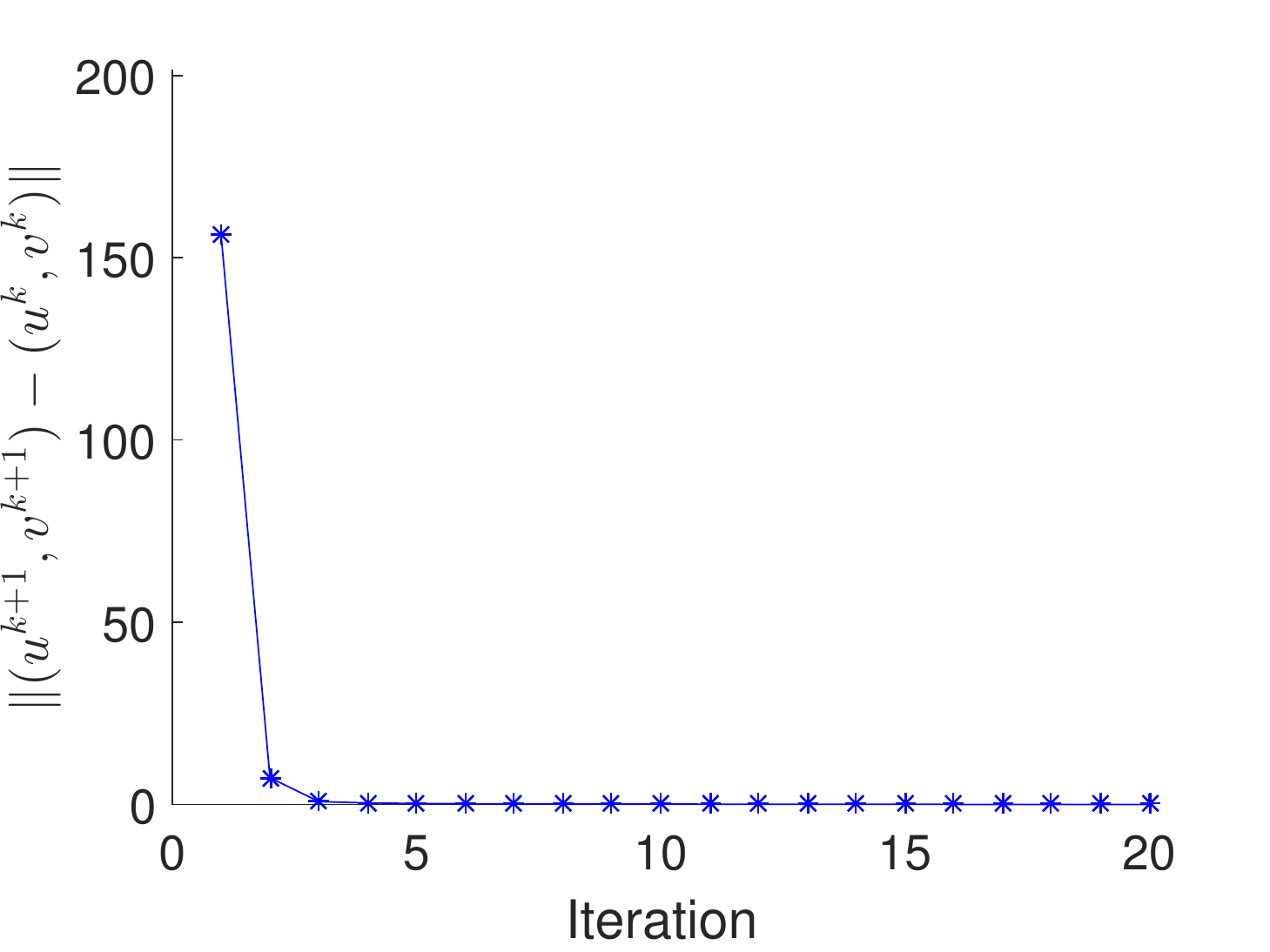} \ &
			\includegraphics[width=1.3in]{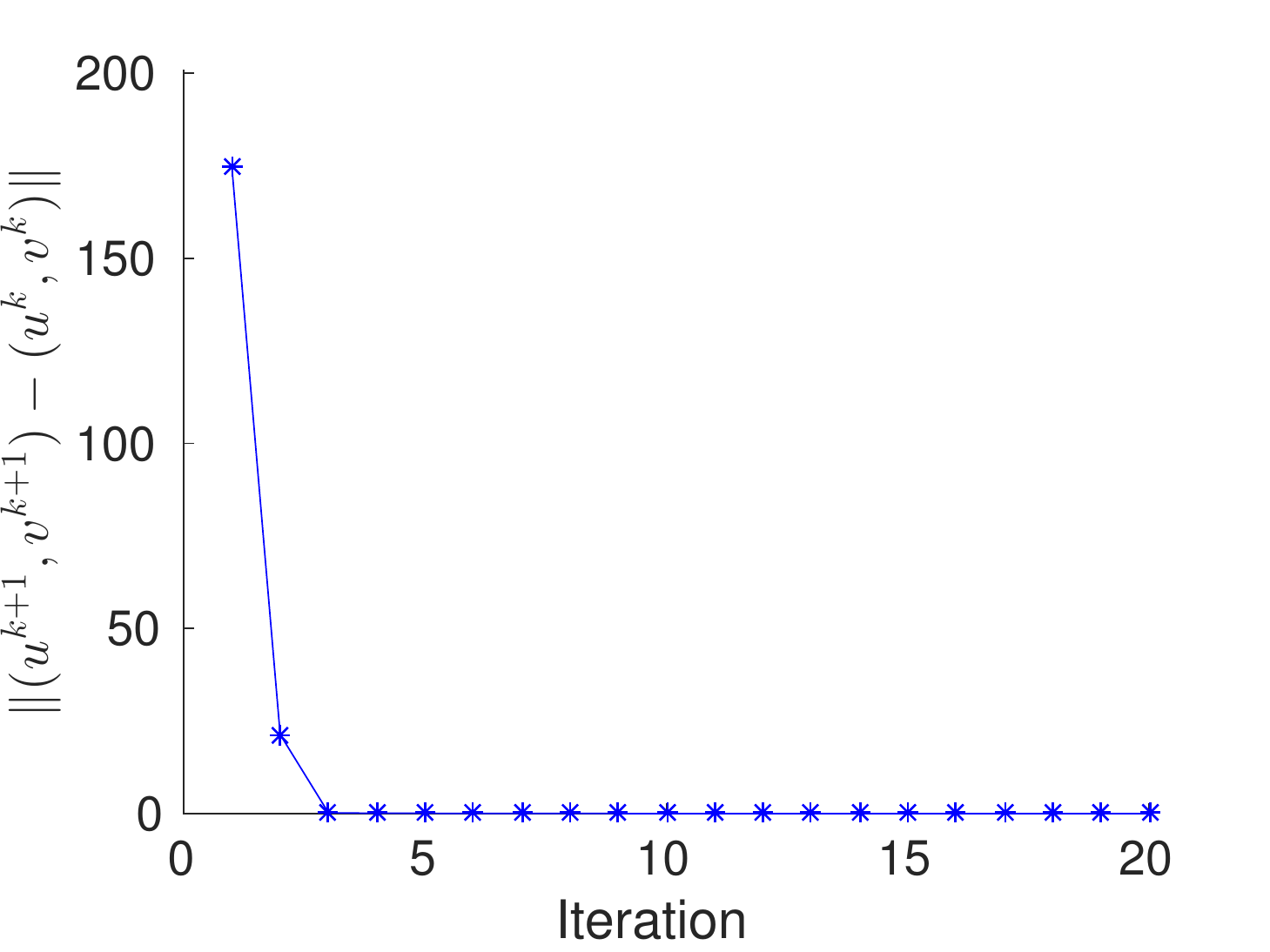} \ &
			\includegraphics[width=1.3in]{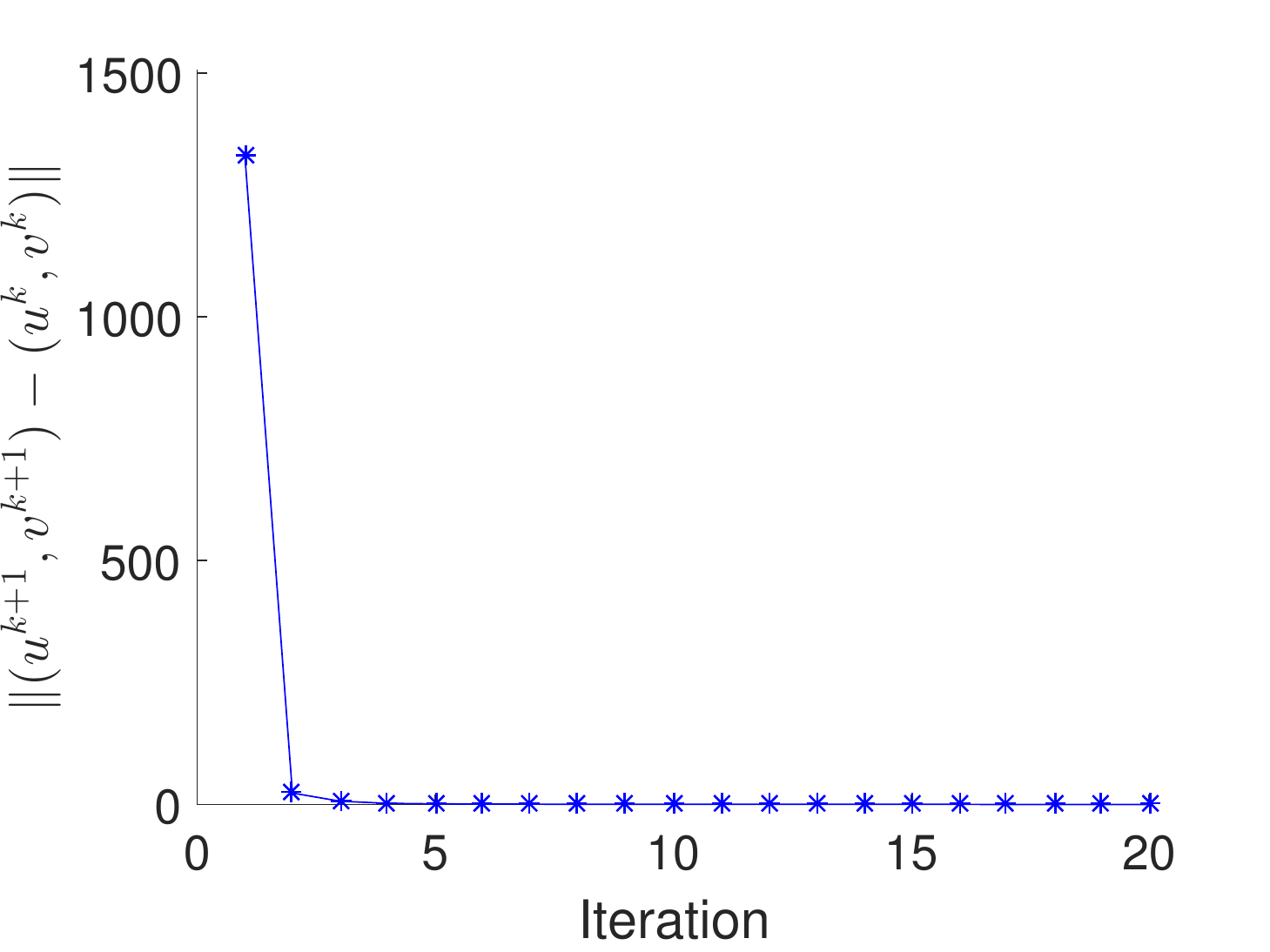} \ &
			\includegraphics[width=1.3in]{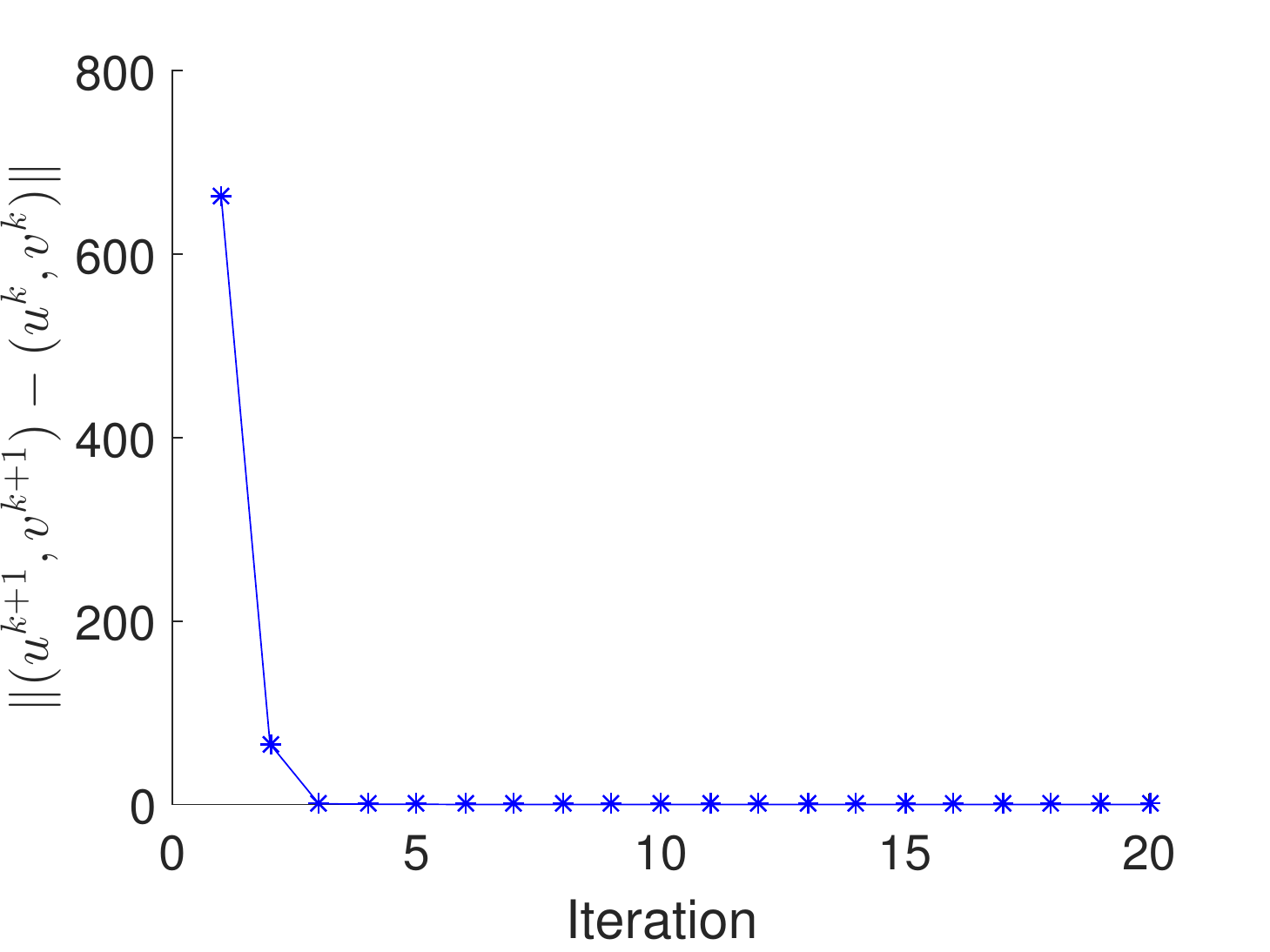} \\
			\includegraphics[width=1.3in]{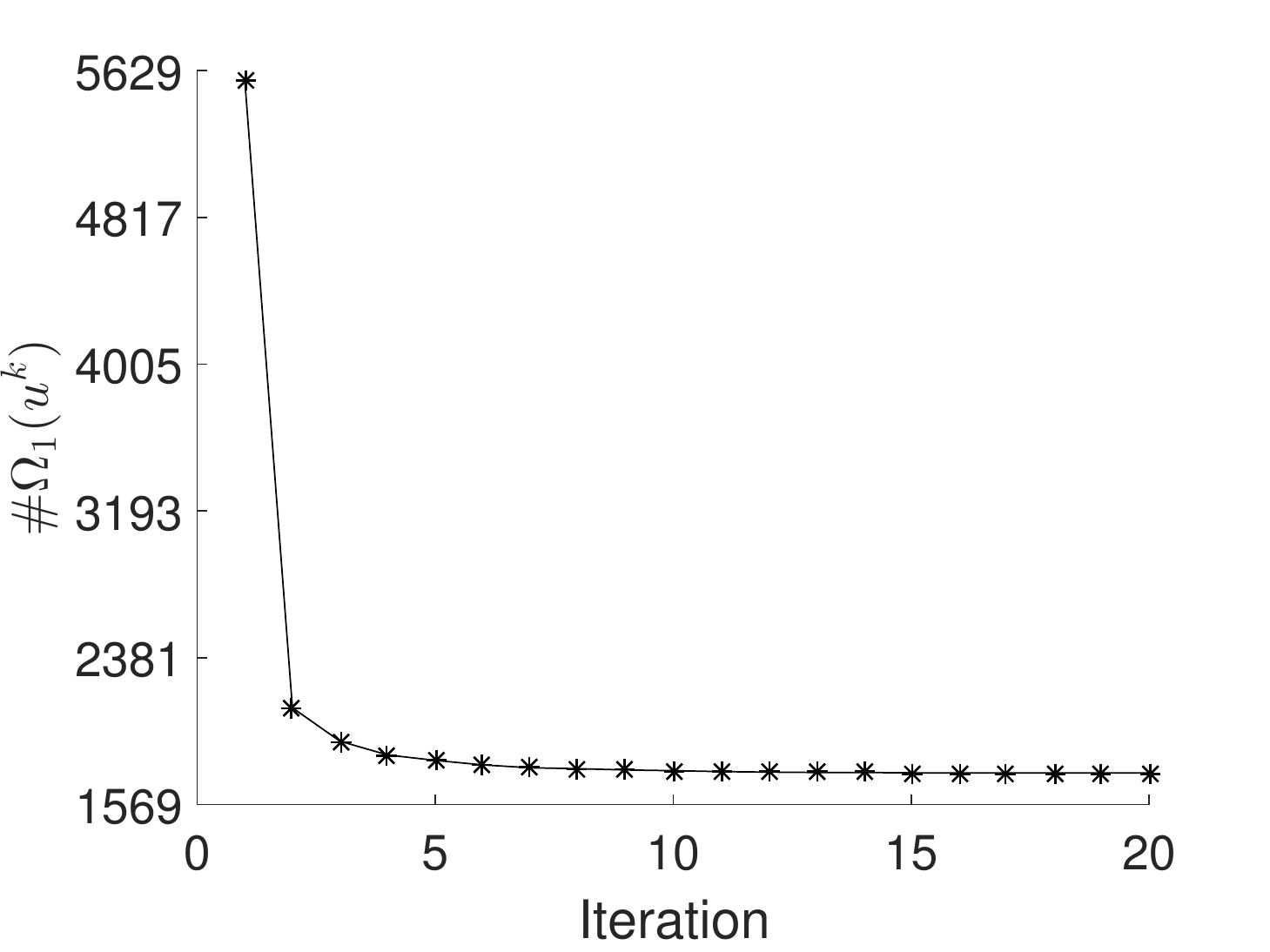} \ &
			\includegraphics[width=1.3in]{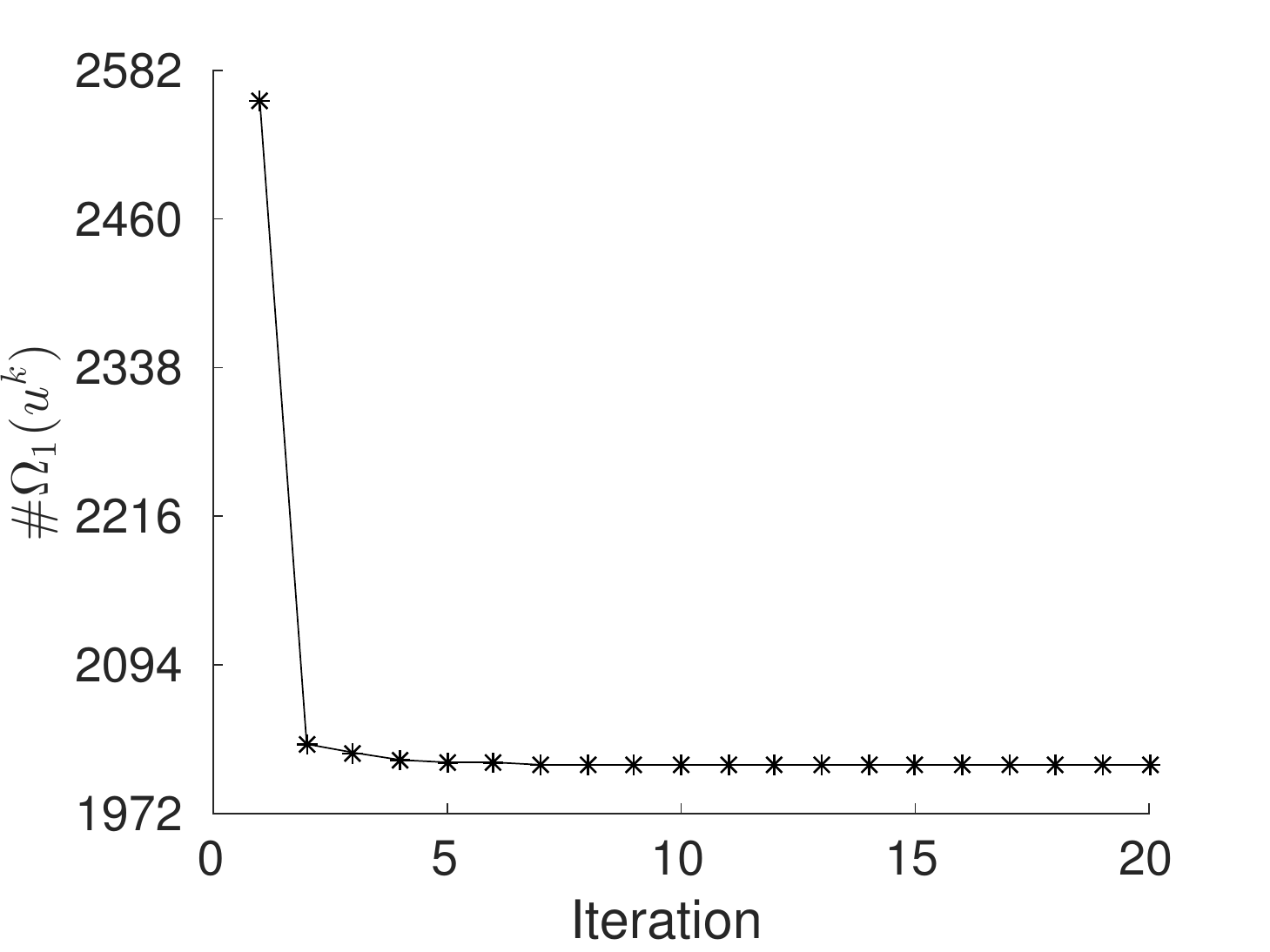} \ &
			\includegraphics[width=1.3in]{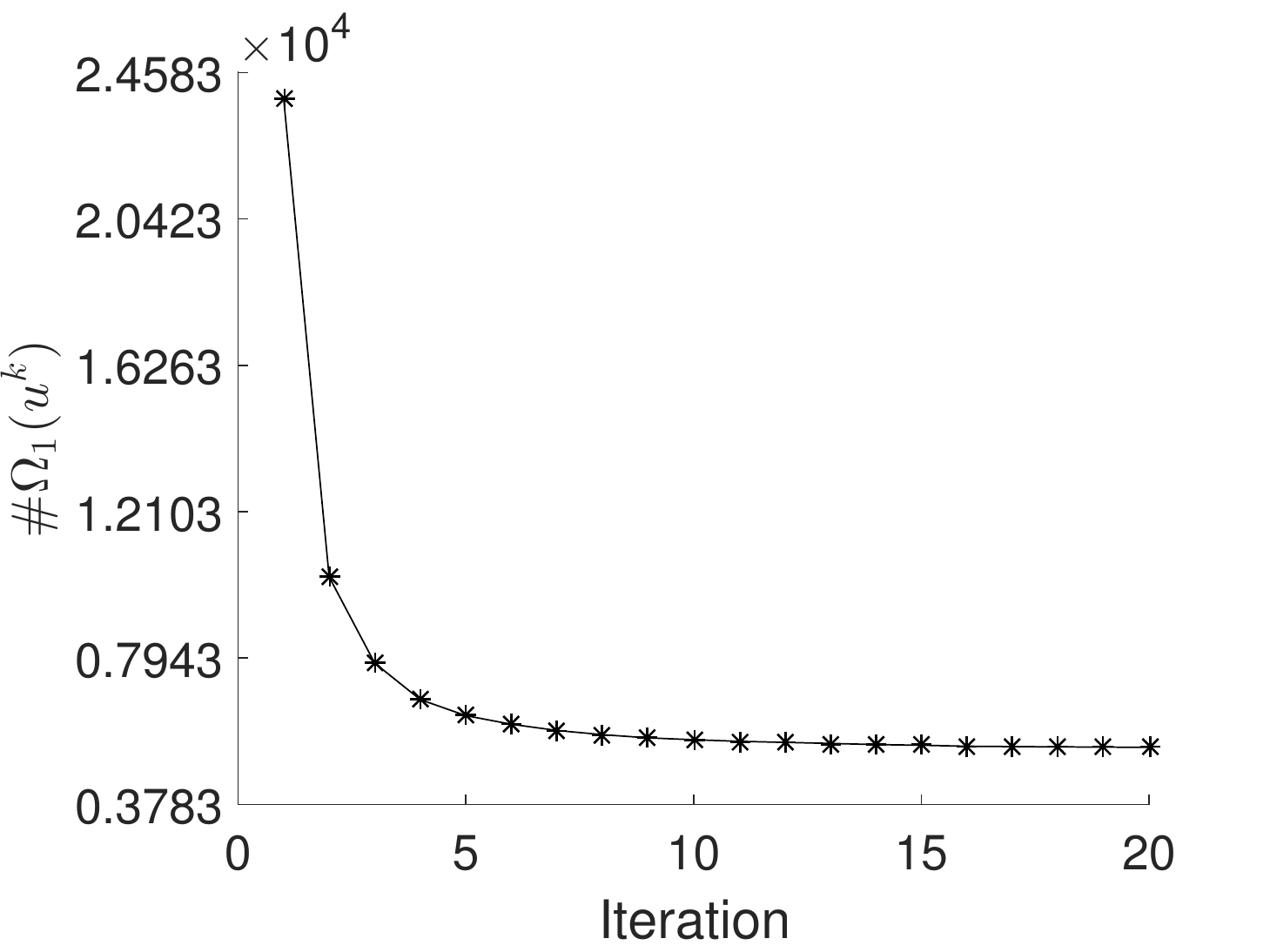} \ &
			\includegraphics[width=1.3in]{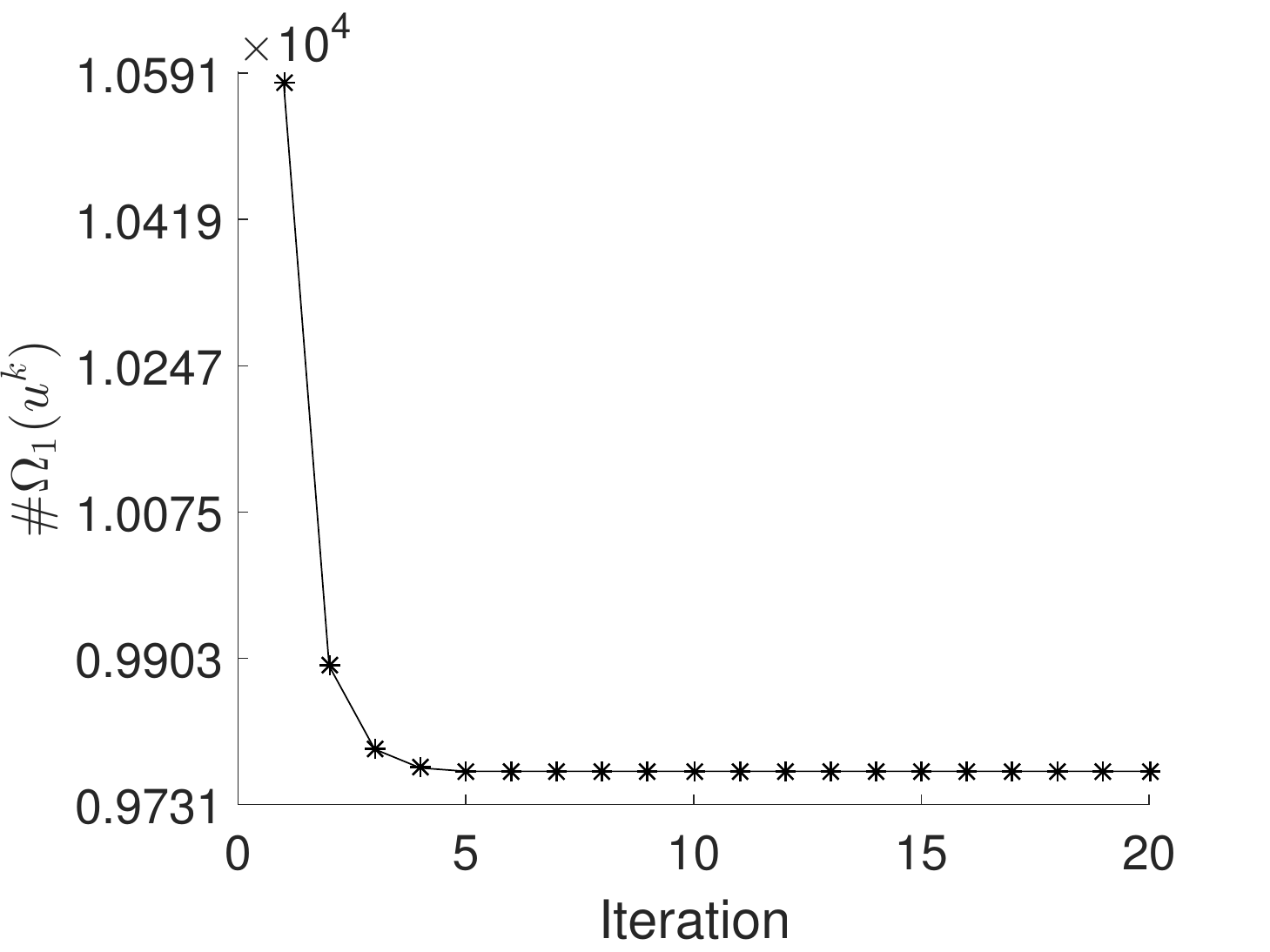} \\
		\end{tabular}
	}
	\caption{\small\sl  Convergence behavior of our algorithm. The first, second and third rows: $F(u^k,v^k)$, $\Vert (u^{k+1},v^{k+1})-(u^{k},v^{k}) \Vert$ and $\#\Omega_1(u^k)$ versus the outer iteration number. The first, second, third and fourth columns: the evolution curves on the second test images in Figure \ref{fig1_geometry1fn}, Figure \ref{fig1_geometry9fn}, Figure \ref{fig2_twophase} and Figure \ref{fig2_mri14}, respectively.}
	\label{fig0_convergence}
\end{figure}

\subsection{Comparisons on synthetic images}\label{experiment1}
In this subsection, we report our experiments on two synthetic images. The test images in Figure \ref{fig1_geometry1fn} and \ref{fig1_geometry9fn} are generated by clean piecewise constant images with additive Gaussian noise and smooth bias field. The results are evaluated visually and quantitatively.

We first look at a two-phase segmentation example on an image corrupted by different levels of inhomogeneity and noise. The test images, the inhomogeneity corrected images and the segmentation results are given in Figure \ref{fig1_geometry1fn}. Therein the first two rows are for a weakly inhomogeneous case, while the other two rows are for a strongly inhomogeneous one. The corresponding CV and JS values quantitatively evaluating the results are given in Table \ref{table_geometry1fn}. We can see that all methods get good segmentation results for the weakly inhomogeneous image. For the strongly inhomogeneous one with more measurement noise, LIC, CCZ and CNCS fail to segment it correctly. The segmentation result of L0MS is correct almost everywhere, but with some isolated speckles, as can be seen by a careful observation. Our method, however, can still segment it very well. These differences can be understood from the corrected images. The corrected image by LIC contains some inhomogeneity, indicating in some sense a possible imperfect segmentation. The corrected images by CCZ and CNCS are also still inhomogeneous, because CCZ and CNCS are restoration models which give a piecewise smooth rather than piecewise constant approximation of the image. The isolated speckles in the segmentation result by L0MS are because of the sparse residual noise in the corresponding corrected image in the third row; see, e.g., the white spots, by zooming in it. This phenomenon can also be observed in the next examples shown in Figure \ref{fig1_geometry9fn} and Figure \ref{fig2_twophase}. The residual speckle noise by L0MS is due to the flatness of the $L_0$ function over $(0,+\infty)$, which tends to generate sparse strong singularities. In contrast, our method gives the best inhomogeneity-corrected image and segmentation result. The quantitative comparisons in Table \ref{table_geometry1fn}, especially the CV values, demonstrate more clearly the overall better performances of our approach. This indeed indicates the robustness of our method to image noise and intensity inhomogeneity.

Now we investigate a five-phase segmentation example on an image with different levels of inhomogeneity and noise. Note that, the code of the LIC method provided by the authors of \cite{li2011level} does not apply to five-phase segmentation, thus we do not include it for the comparison in this example.
In Figure \ref{fig1_geometry9fn}, we give the test images, the inhomogeneity corrected images and the segmentation results. The corresponding CV and JS values are given in Table \ref{table_geometry9fn}. The experimental results and phenomenon are similar to that of the two-phase case in Figure \ref{fig1_geometry1fn} and Table \ref{table_geometry1fn}.

\begin{table}[H]
	\scriptsize
	\renewcommand{\arraystretch}{1.1}
	\caption{\small\sl Quantitative evaluation of the results in Figure \ref{fig1_geometry1fn} in terms of CV and JS values.}
	\centering
	\begin{tabular}{|c|c|c|c|c|c|c|c|c|}
	\hline
	& \multicolumn{4}{|c|}{results of the first test image in Figure \ref{fig1_geometry1fn}} & \multicolumn{4}{|c|}{results of the second test image in Figure \ref{fig1_geometry1fn}} \\
	\hline
	phase & \multicolumn{2}{|c|}{1} & \multicolumn{2}{|c|}{2} & \multicolumn{2}{|c|}{1} & \multicolumn{2}{|c|}{2} \\
	\hline
	& CV & JS & CV & JS & CV & JS & CV & JS \\
	\hline
	LIC & 0.2913 & 1.0000 & 0.0409 & 1.0000 & 0.4091 & 0.8353 & 0.2539 & 0.9122 \\	
	\hline
	CCZ & 0.4745 & 1.0000 & 0.1499 & 1.0000 & 0.4982 & 0.4991 & 0.4061 & 0.6091 \\	
	\hline
	CNCS & 0.3805 & 1.0000 & 0.1420 & 1.0000 & 0.3317 & 0.4880 & 0.3262 & 0.5913 \\	
	\hline
	L0MS & 0.0598 & 1.0000 & 0.0073 & 1.0000 & 0.2250 & 0.9927 & 0.0469 & 0.9967 \\	
	\hline
	Ours & \textbf{0.0185} & 1.0000 & \textbf{0.0011} & 1.0000 & \textbf{0.0293} & \textbf{0.9971} & \textbf{0.0120} & \textbf{0.9987} \\  	
	\hline	
	\end{tabular}
\label{table_geometry1fn}
\end{table}

\begin{table}[H]
	\scriptsize
	\renewcommand{\arraystretch}{1.1}
	\caption{\small\sl Quantitative evaluation of the results in Figure \ref{fig1_geometry9fn} in terms of CV and JS values.}
	\centering
	\begin{tabular}{|c|c|c|c|c|c|c|c|c|c|c|}
		\hline
		\multicolumn{11}{|c|}{results of the first test image in Figure \ref{fig1_geometry9fn}} \\
		\hline
		phase & \multicolumn{2}{|c|}{1} & \multicolumn{2}{|c|}{2} & \multicolumn{2}{|c|}{3} & \multicolumn{2}{|c|}{4} & \multicolumn{2}{|c|}{5} \\
		\hline
		& CV & JS & CV & JS & CV & JS & CV & JS & CV & JS \\
		\hline
		CCZ & 0.5456 & 1.0000 & 0.0241 & 0.9922 & 0.0591 & 0.9959 & 0.0172 & 1.0000 & 0.0179 & 1.0000 \\	
		\hline
		CNCS & 0.4768 & 1.0000 & 0.0234 & 0.9995 & 0.0549 & 0.9997 & 0.0169 & 1.0000 & 0.0162 & 1.0000 \\		
		\hline
        L0MS & 0.0064 & 1.0000 & \textbf{0.0001} & \textbf{1.0000} & 0.0010 & \textbf{1.0000} & 0.0037 & 1.0000 & 0.0006 & 1.0000 \\
		\hline
        Ours & \textbf{0.0012} & 1.0000 & \textbf{0.0001} & \textbf{1.0000} & \textbf{0.0004} & \textbf{1.0000} & \textbf{0.0018} & 1.0000 & \textbf{0.0005} & 1.0000 \\
        \hline
        \hline
        \multicolumn{11}{|c|}{results of the second test image in Figure \ref{fig1_geometry9fn}} \\
        \hline
        phase & \multicolumn{2}{|c|}{1} & \multicolumn{2}{|c|}{2} & \multicolumn{2}{|c|}{3} & \multicolumn{2}{|c|}{4} & \multicolumn{2}{|c|}{5} \\
        \hline
        & CV & JS & CV & JS & CV & JS & CV & JS & CV & JS \\
        \hline
        CCZ & 0.5283 & 0.3630 & 0.0370& 0.2919 & 0.3243 & 0.2423 & 0.0653 &	0.0000 & 0.1008 & 0.4488 \\	
        \hline
        CNCS & 0.4942 & 0.3632 & 0.0402 & 0.2892 & 0.3086 & 0.2349 & 0.0685 & 0.0000 & 0.0982 & 0.4509 \\		
        \hline
        L0MS & 0.0117 & 0.9999 & 0.0052 & 0.9973 & 0.0075 & 0.9985 & 0.0159 & 0.9910 & 0.0078 & 0.9978 \\
        \hline
        Ours & \textbf{0.0007} & \textbf{1.0000} & \textbf{0.0013} & \textbf{0.9997} & \textbf{0.0011} & \textbf{0.9997} & \textbf{0.0025} & \textbf{0.9991} & \textbf{0.0018} & \textbf{0.9998} \\
		\hline	
	\end{tabular}
\label{table_geometry9fn}
\end{table}

\begin{figure}[H]
	\centerline{
		\begin{tabular}{c@{}c@{}c@{}c@{}c@{}c}
			\includegraphics[width=0.8in]{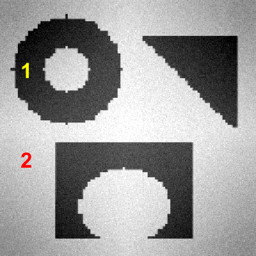} \ &
			\includegraphics[width=0.8in]{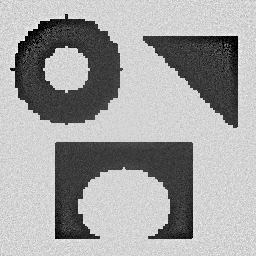} \ &
			\includegraphics[width=0.8in]{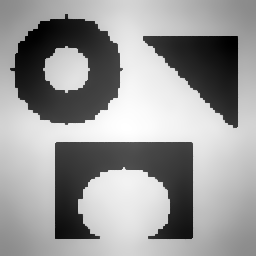} \ &
			\includegraphics[width=0.8in]{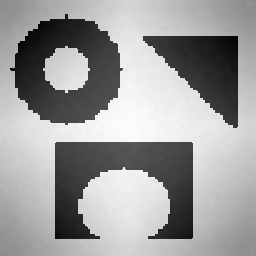} \ &
			\includegraphics[width=0.8in]{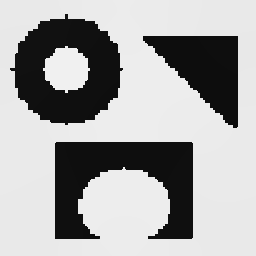} \ &
			\includegraphics[width=0.8in]{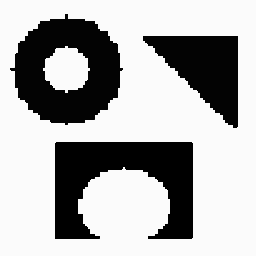}  \\
			&
			\includegraphics[width=0.8in]{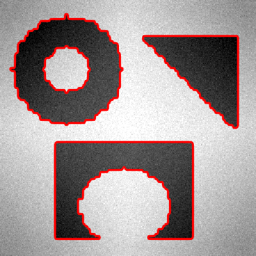} \ &
			\includegraphics[width=0.8in]{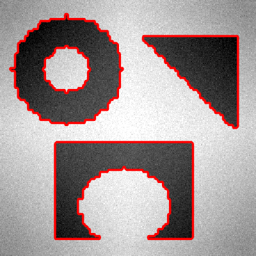} \ &
			\includegraphics[width=0.8in]{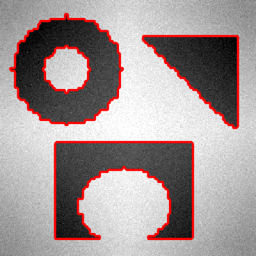} \ &
			\includegraphics[width=0.8in]{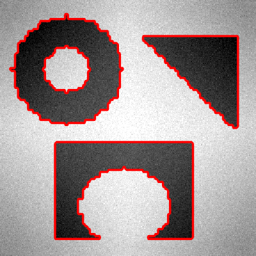} \ &
			\includegraphics[width=0.8in]{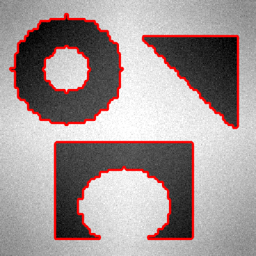}  \\
			\includegraphics[width=0.8in]{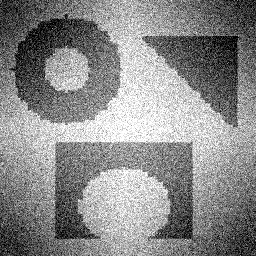} \ &
			\includegraphics[width=0.8in]{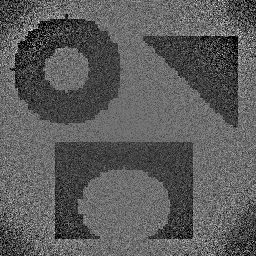} \ &
			\includegraphics[width=0.8in]{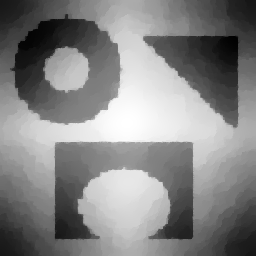} \ &
			\includegraphics[width=0.8in]{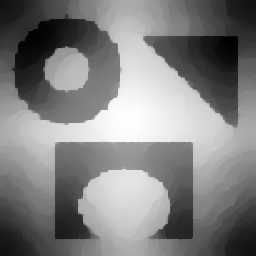} \ &
			\includegraphics[width=0.8in]{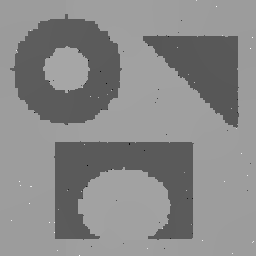} \ &
			\includegraphics[width=0.8in]{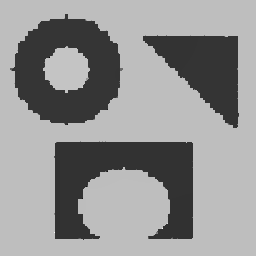}  \\
			&
			\includegraphics[width=0.8in]{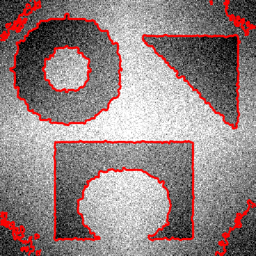} \ &
			\includegraphics[width=0.8in]{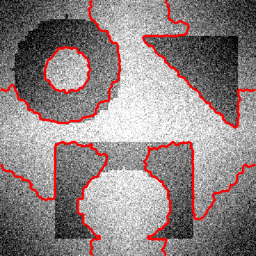} \ &
			\includegraphics[width=0.8in]{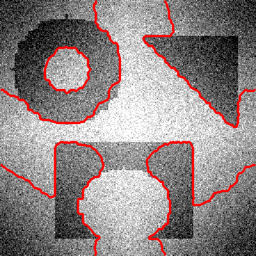} \ &
			\includegraphics[width=0.8in]{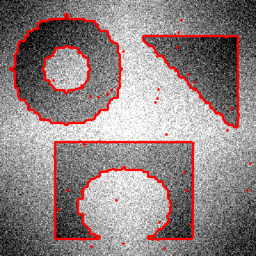} \ &
			\includegraphics[width=0.8in]{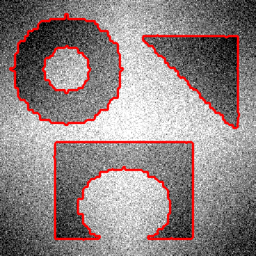}  \\
			(a) Input & (b) LIC & (c) CCZ & (d) CNCS & (e) L0MS & (f) Ours
		\end{tabular}
	}
	\caption{\small\sl Performance comparisons between different methods applied to two-phase segmentation. Row 1: a test noisy image with weak inhomogeneity and its inhomogeneity-corrected versions by different methods; Row 2: segmentation results corresponding to Row 1. Row 3: another test noisy image with strong inhomogeneity and its inhomogeneity-corrected versions; Row 4: segmentation results corresponding to Row 3. The fine-tuned parameters for these two tests are: for LIC, $(\sigma,\Delta t)=(4,0.1),(7,0.2)$; for CCZ, $(\lambda,\mu)=(10,0.1),(6,1)$; for CNCS, $(\lambda,T)=(10,0.01),(8,0.001)$; for L0MS, $(\alpha,k)=(0.02,0.1),(0.02,1)$; and for ours, $(\alpha,\beta)=(0.1,100),(0.1,1000)$. The corresponding CV and JS values are given in Table \ref{table_geometry1fn}.}
	\label{fig1_geometry1fn}
\end{figure}

\begin{figure}[H]
	\centerline{
		\begin{tabular}{c@{}c@{}c@{}c@{}c}
			\includegraphics[width=0.77in]{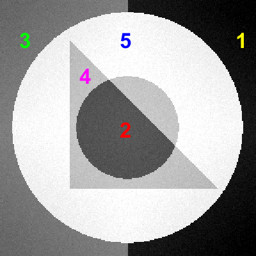} \ &
			\includegraphics[width=0.77in]{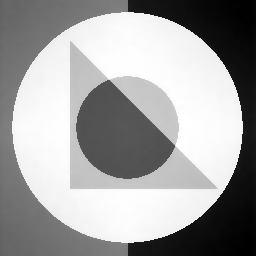} \ &
			\includegraphics[width=0.77in]{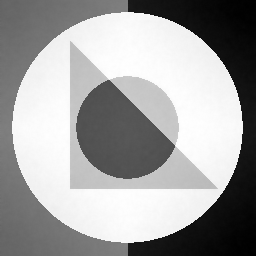} \ &
			\includegraphics[width=0.77in]{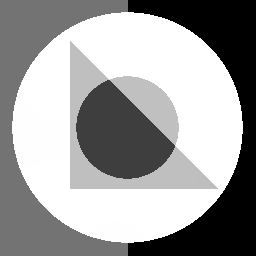} \ &
			\includegraphics[width=0.77in]{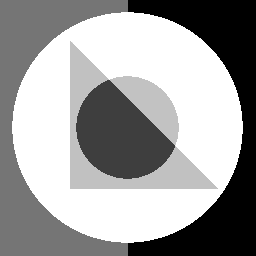}  \\
			&
			\includegraphics[width=0.77in]{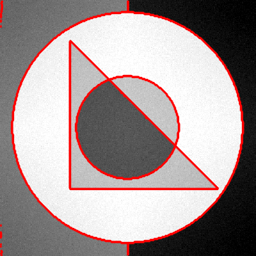} \ &
			\includegraphics[width=0.77in]{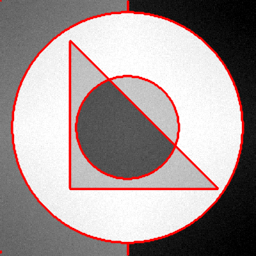} \ &
			\includegraphics[width=0.77in]{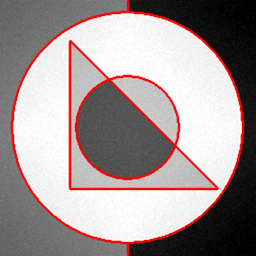} \ &
			\includegraphics[width=0.77in]{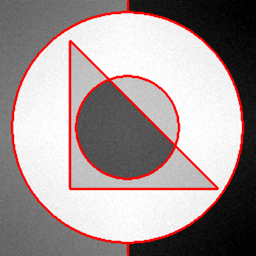}  \\
			\includegraphics[width=0.77in]{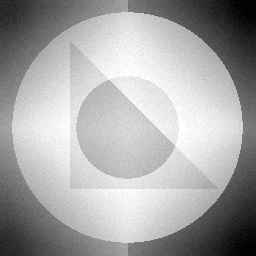} \ &
			\includegraphics[width=0.77in]{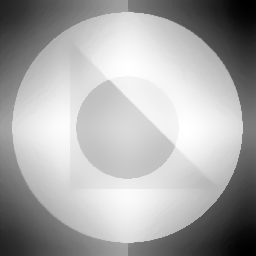} \ &
			\includegraphics[width=0.77in]{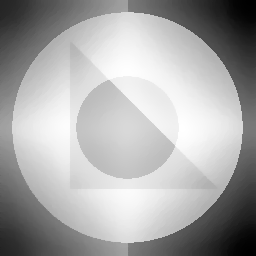} \ &
			\includegraphics[width=0.77in]{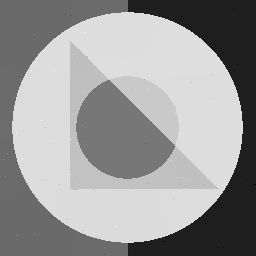} \ &
			\includegraphics[width=0.77in]{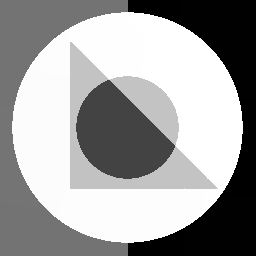}  \\
			&
			\includegraphics[width=0.77in]{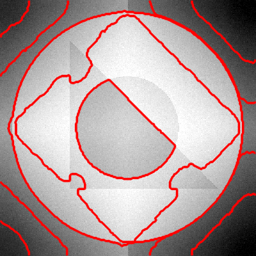} \ &
			\includegraphics[width=0.77in]{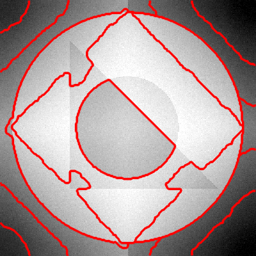} \ &
			\includegraphics[width=0.77in]{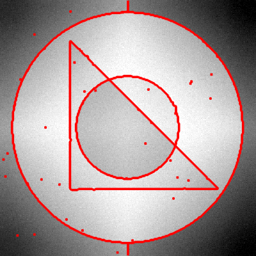} \ &
			\includegraphics[width=0.77in]{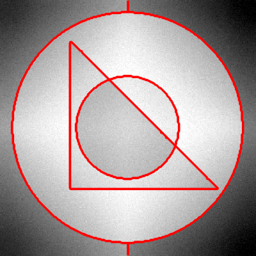}  \\
			(a) Input & (b) CCZ & (c) CNCS & (d) L0MS & (e) Ours
		\end{tabular}
	}
	\caption{\small\sl Performance comparisons between different methods applied to five-phase segmentation. Row 1: a test noisy image with weak inhomogeneity and its inhomogeneity-corrected versions by different methods; Row 2: segmentation results corresponding to Row 1. Row 3: another test noisy image with strong inhomogeneity and its inhomogeneity-corrected versions; Row 4: segmentation results corresponding to Row 3. The fine-tuned parameters for these two tests are: for CCZ, $(\lambda,\mu)=(30,0.01),(10,0.01)$; for CNCS, $(\lambda,T)=(10,0.01),(10,0.001)$; for L0MS, $(\alpha,k)=(0.01,100),(0.007,300)$; and for ours, $(\alpha,\beta)=(0.01,10),(0.01,1000)$. The corresponding CV and JS values are given in Table \ref{table_geometry9fn}.}
	\label{fig1_geometry9fn}
\end{figure}

\subsection{Comparisons on simulated  and real medical images}\label{experiment2}
We give our experiments on four simulated and real medical images in this subsection. The results are compared visually.

We begin with two two-phase segmentation examples on a medical image and a retina vessel image. The test images, the inhomogeneity corrected images and the segmentation results are given in Figure \ref{fig2_twophase}. As shown, the corrected images by CCZ and CNCS are still with inhomogeneity, thus their segmentation results in the second stage are not satisfactory. The segmentation results of LIC, L0MS, and our method are comparable for the first one. A careful observation shows that our segmentation for the second one is a little better, which is with cleaner and smoother boundary curves. Besides, the L0MS segmentation results are with several speckles (see the second and the top boundary of the first corrected image), which is due to the reason stated in the previous subsection. Our method meanwhile provides, although not perfect, but relatively the best inhomogeneity-corrected images, which are almost piecewise constant and with no noise or stair case effect.

Now we present two four-phase segmentation examples on two MRI images, one with noise and inhomogeneity and the other with only inhomogeneity. We mention that, LIC is also tested here, as in \cite{li2011level}, by removing the background whose pixel values are near to zero. The two test images, the inhomogeneity corrected images and the segmentation results are given in Figure \ref{fig2_mri14}. As shown, all methods get relatively satisfactory segmentation results for the first image, which is with little inhomogeneity or noise. For the second one, CCZ and CNCS lose their effectiveness, while LIC, L0MS and our method can segment it well. Little differences can be observed yet; see the blue rectangles in the last row. As for the inhomogeneity correction, both L0MS and our method perform better than others. The first corrected image of LIC contains some weak noise. The corrected images of CCZ and CNCS for the second one are still inhomogeneous.

\begin{figure}[H]
	\centerline{
		\begin{tabular}{c@{}c@{}c@{}c@{}c@{}c}
			\includegraphics[width=1in]{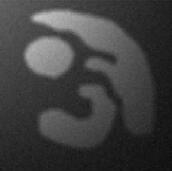} \ &
			\includegraphics[width=1in]{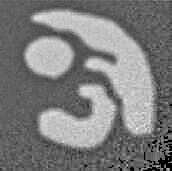} \ &
			\includegraphics[width=1in]{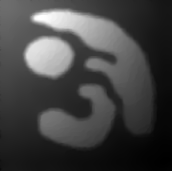} \ &
			\includegraphics[width=1in]{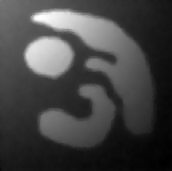} \ &
			\includegraphics[width=1in]{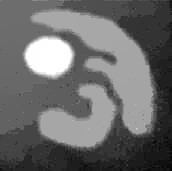} \ &
			\includegraphics[width=1in]{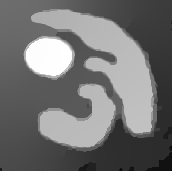}  \\
			&
			\includegraphics[width=1in]{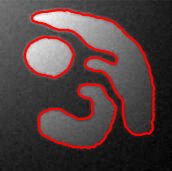} \ &
			\includegraphics[width=1in]{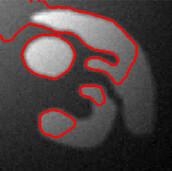} \ &
			\includegraphics[width=1in]{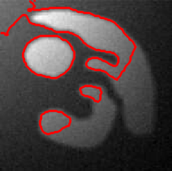} \ &
			\includegraphics[width=1in]{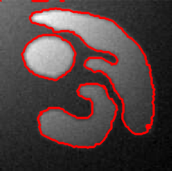} \ &
			\includegraphics[width=1in]{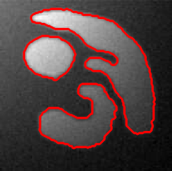}  \\
			\includegraphics[width=1in]{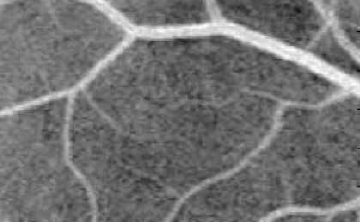} \ &
			\includegraphics[width=1in]{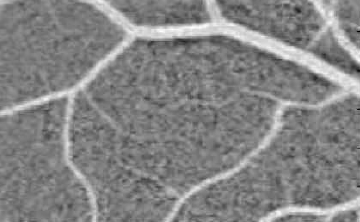} \ &
			\includegraphics[width=1in]{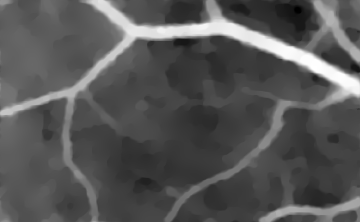} \ &
			\includegraphics[width=1in]{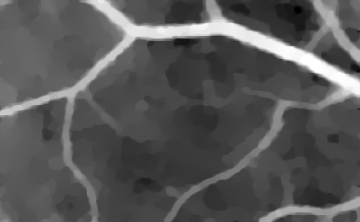} \ &
			\includegraphics[width=1in]{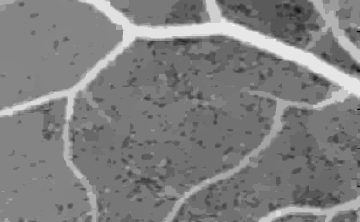} \ &
			\includegraphics[width=1in]{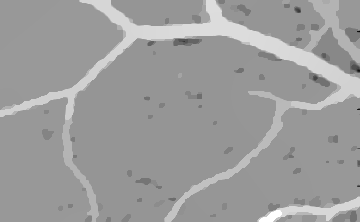}  \\
			&
			\includegraphics[width=1in]{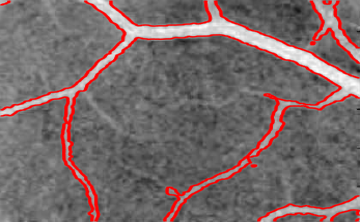} \ &
			\includegraphics[width=1in]{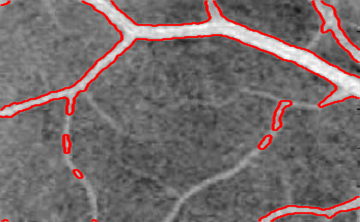} \ &
			\includegraphics[width=1in]{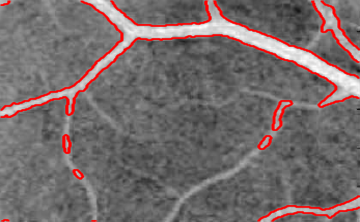} \ &
			\includegraphics[width=1in]{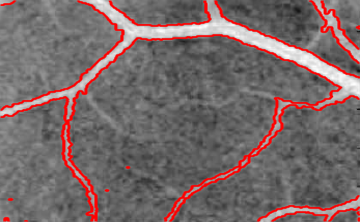} \ &
			\includegraphics[width=1in]{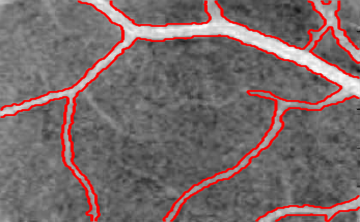}  \\
			(a) Input & (b) LIC & (c) CCZ & (d) CNCS & (e) L0MS & (f) Ours
		\end{tabular}
	}
	\caption{\small\sl Performance comparisons between different methods applied to two-phase segmentation for a medical image and a retina vessel image. Row 1: a medical image with noise and inhomogeneity and its inhomogeneity-corrected versions by different methods; Row 2: segmentation results corresponding to Row 1. Row 3: a retina vessel image with noise and inhomogeneity and its inhomogeneity-corrected versions; Row 4: segmentation results corresponding to Row 3. The fine-tuned parameters for these two tests are: for LIC, $(\sigma,\Delta t)=(6,0.2),(9,0.2)$; for CCZ, $(\lambda,\mu)=(18,1),(8,1)$; for CNCS, $(\lambda,T)=(10,0.006),(9,0.001)$; for L0MS, $(\alpha,k)=(1,310),(0.1,300)$; and for ours, $(\alpha,\beta)=(0.1,1000),(0.1,100)$.}
	\label{fig2_twophase}
\end{figure}

\begin{figure}[H]
	\centerline{
		\begin{tabular}{c@{}c@{}c@{}c@{}c@{}c}
			\includegraphics[width=1in]{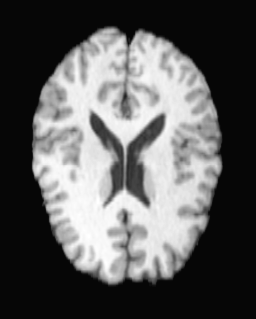} \ &
			\includegraphics[width=1in]{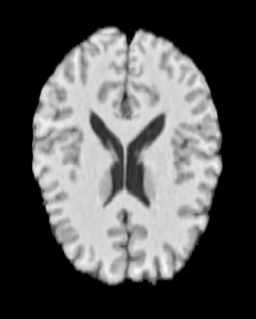} \ &
			\includegraphics[width=1in]{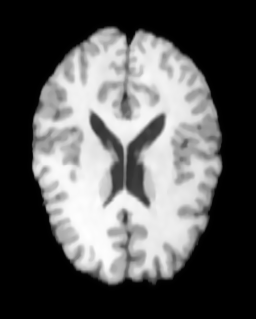} \ &
			\includegraphics[width=1in]{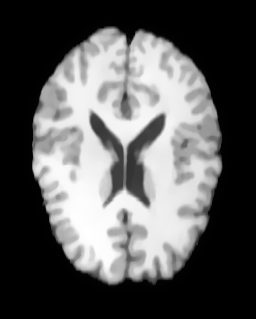} \ &
			\includegraphics[width=1in]{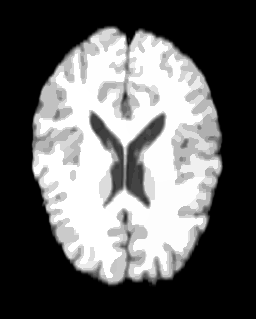} \ &
			\includegraphics[width=1in]{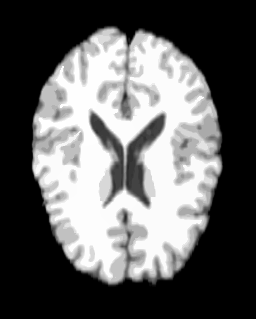} \\
			&
			\includegraphics[width=1in]{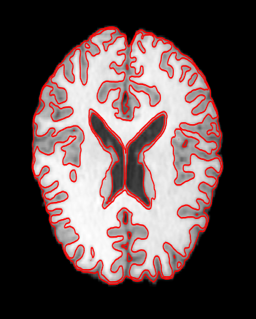} \ &
			\includegraphics[width=1in]{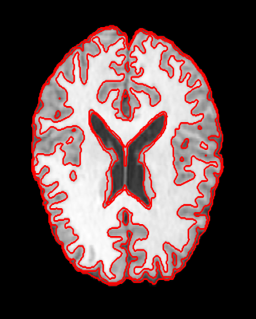} \ &		
			\includegraphics[width=1in]{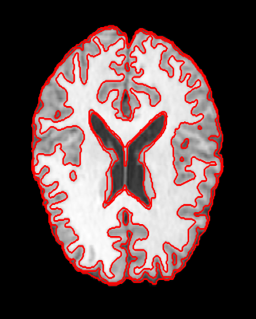} \ &
			\includegraphics[width=1in]{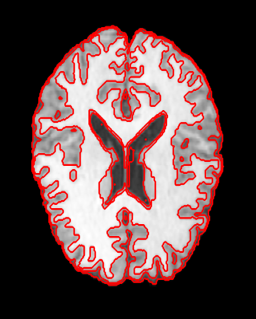} \ &
			\includegraphics[width=1in]{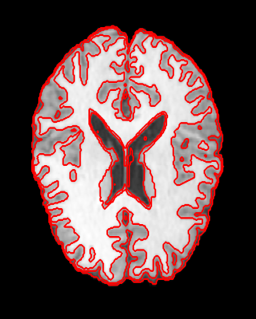} \\
			\includegraphics[width=1in]{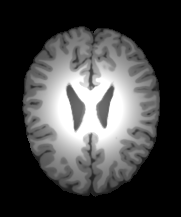} \ &
			\includegraphics[width=1in]{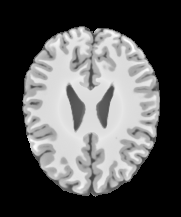} \ &
			\includegraphics[width=1in]{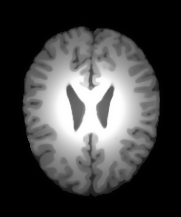} \ &
			\includegraphics[width=1in]{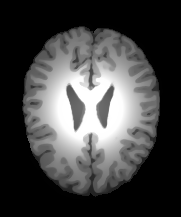} \ &
			\includegraphics[width=1in]{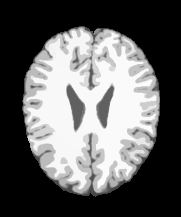} \ &
			\includegraphics[width=1in]{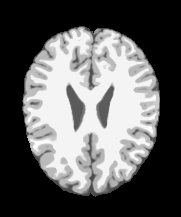} \\
			&
			\includegraphics[width=1in]{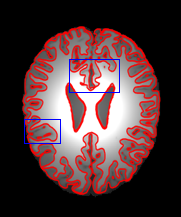} \ &
			\includegraphics[width=1in]{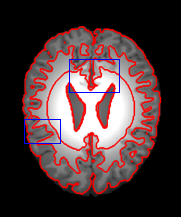} \ &		
			\includegraphics[width=1in]{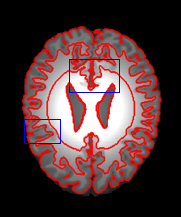} \ &
			\includegraphics[width=1in]{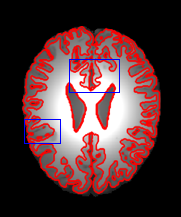} \ &
			\includegraphics[width=1in]{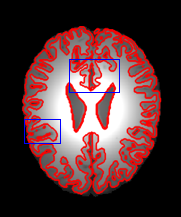} \\
			(a) Input & (b) LIC & (c) CCZ & (d) CNCS & (e) L0MS & (f) Ours
		\end{tabular}
	}	
	\caption{\small\sl Performance comparisons between different methods applied to four-phase segmentation for two MRI images. Row 1: an MRI image with noise and inhomogeneity and its inhomogeneity-corrected versions by different methods; Row 2: segmentation results corresponding to Row 1. Row 3: another MRI image with inhomogeneity and its inhomogeneity-corrected versions; Row 4: segmentation results corresponding to Row 3. The fine-tuned parameters for these two tests are: for LIC, $(\sigma,\Delta t)=(6,0.01),(5,0.01)$; for CCZ, $(\lambda,\mu)=(40,1),(70,1)$; for CNCS, $(\lambda,T)=(10,0.001),(20,0.001)$; for L0MS, $(\alpha,k)=(0.003,100),(2,300)$; and for ours, $(\alpha,\beta)=(0.007,2000),(0.001,100)$.}
	\label{fig2_mri14}
\end{figure}

\subsection{More comparisons on a real brain MRI dataset}\label{experiment3}
In this subsection, we do a brain segmentation test on a real brain MRI dataset shown in Figure \ref{fig3_testimages}. This dataset was also used in \cite{chang2017new} to test their 3D algorithm. It is a 3D dataset containing 12 slices of $256\times256$ images where significant intensity inhomegeneities can be observed. There is public available ground truth for the location of the brain region; see Figure \ref{fig3_groundtruth} for examples of the ground truth location in slice1, slice6, slice12, respectively. The results in this subsection are assessed visually and quantitatively.

\begin{figure}[H]
	\centerline{
		\begin{tabular}{c@{}c@{}c@{}c@{}c@{}c}
			\includegraphics[width=0.8in]{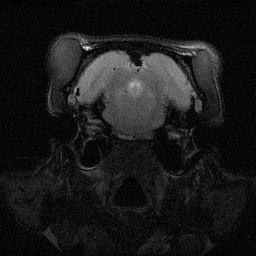} \ &
			\includegraphics[width=0.8in]{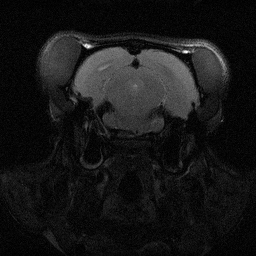}\ &	
			\includegraphics[width=0.8in]{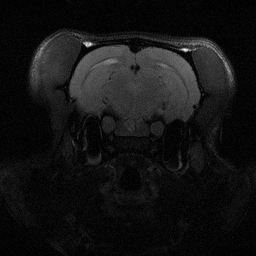} \ &
			\includegraphics[width=0.8in]{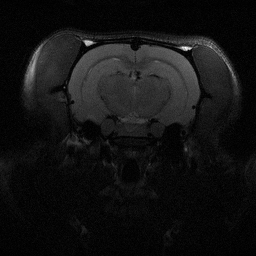} \ &
			\includegraphics[width=0.8in]{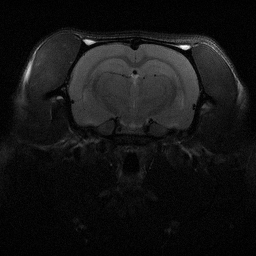} \ &
			\includegraphics[width=0.8in]{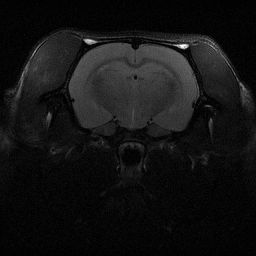} \\
			slice1 & slice2 & slice3 & slice4 & slice5 & slice6 \\
			\includegraphics[width=0.8in]{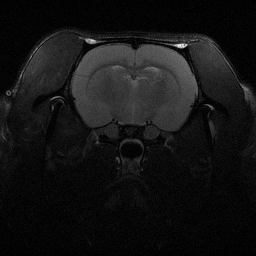} \ &
			\includegraphics[width=0.8in]{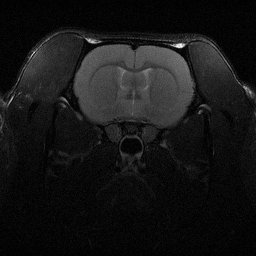} \ &		
			\includegraphics[width=0.8in]{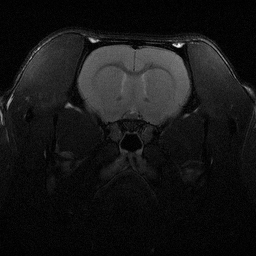} \ &
			\includegraphics[width=0.8in]{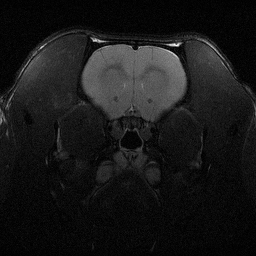} \ &
			\includegraphics[width=0.8in]{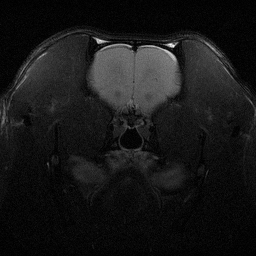} \ &
			\includegraphics[width=0.8in]{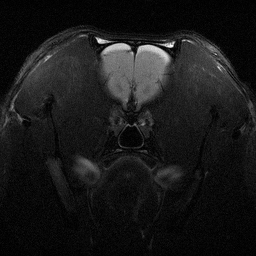} \\
			slice7 & slice8 & slice9 & slice10 & slice11 & slice12 \\
		\end{tabular}
	}
	\caption{\small\sl  A real brain MRI dataset used for a brain segmentation test: 12 slices of $256\times256$ images.}
	\label{fig3_testimages}
\end{figure}

\begin{figure}[H]
	\centerline{
		\begin{tabular}{c@{}c@{}c}		
			\includegraphics[width=0.8in]{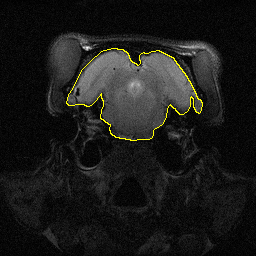} \ &
			\includegraphics[width=0.8in]{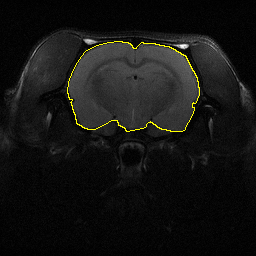} \ &
			\includegraphics[width=0.8in]{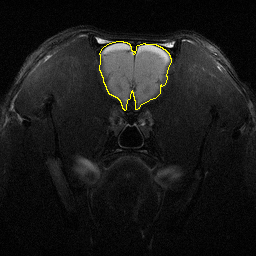} \\
			slice1 & slice6 & slice12 \\
		\end{tabular}
	}
	\caption{\small\sl Ground truth location of the brain region in slice1, slice6 and slice12 of the dataset in Figure \ref{fig3_testimages}.}
	\label{fig3_groundtruth}
\end{figure}

Here we elaborate on how this two-stage segmentation experiment was performed. In the first stage for bias correction, the 2D methods CCZ\cite{cai2013two}, CNCS\cite{Chan2018Convex} and ours process this dataset slice by slice, respectively; while the truly 3D method HoL0MS\cite{chang2017new} acts one-time on the whole dataset. In the second stage, all methods adopt the package in \cite{chang2017new}, i.e., a 3-phase clustering procedure classifying the image domain into the brain region, the surrounding non-brain region, and the background outside the body, followed by certain morphological operations, to get the final brain segmentation results.

Figure \ref{fig3_result1} gives the inhomogeneity corrected images in the first stage, and the CV values of the brain region in these corrected images are given in Table \ref{table_3D}. As shown in Figure \ref{fig3_result1}, both our method and HoL0MS can provide piecewise constant corrected images suitable for clustering in the next step, while those by CCZ and CNCS are piecewise smooth. A careful observation indicates that our corrected images recover a lot of information which are not visible in the original images; see, e.g., the corrected images of slice 1 and slice 12. From Table \ref{table_3D}, the CV values of results by our method are the lowest, quantitatively demonstrating that our method provides the best corrected images.

Figure \ref{fig3_result2} presents the brain segmentation results in the second stage, and the JS values for these brain segmentation results are given in Table \ref{table_3D}. Visually we can see that the results of HoL0MS and ours are better than those of CCZ and CNCS. As shown, the segmentation results of CCZ and CNCS are not satisfactory in most cases, which are due to their piecewise smooth corrected images in the first stage. A careful observation shows that over all our method locates the brain boundaries more accurately than HoL0MS (see, e.g., the results for slice 12). Moreover, the JS values in Table \ref{table_3D}  quantitatively demonstrate the performance advantage of our approach.

\begin{table}[H]
	\scriptsize
	\renewcommand{\arraystretch}{1.1}
	\caption{\small\sl CV values of the brain region in the bias corrected images in Figure \ref{fig3_result1} and JS values of the brain segmentation results in Figure \ref{fig3_result2}.}
	\centering
	\begin{tabular}{|c|c|c|c|c|c|c|c|c|c|c|c|c|}
		\hline
		slice & 1 & 2 & 3 & 4 & 5 & 6 & 7 & 8 & 9 & 10 & 11 & 12 \\
		\hline
		& CV & CV & CV & CV & CV & CV & CV & CV & CV & CV & CV & CV \\
		\hline
		CCZ & 0.3127 & 0.2958 & 0.3115 & 0.2839 & 0.2642 & 0.2268 & 0.2124 & 0.2069 & 0.1995 & 0.2025 & 0.2186 & 0.2958 \\
		\hline
		CNCS & 0.2923 & 0.2780 & 0.2942 & 0.2696	& 0.2499 & 0.2166 & 0.2021 & 0.1967	& 0.1918 & 0.1977 & 0.2166 & 0.2958 \\
		\hline
		HoL0MS & 0.2018	& 0.1839 & 0.1688 & 0.1530 & 0.1498 & 0.1302 & 0.1237 & 0.1317 & 0.1315 & 0.1337 & 0.1492 & 0.2129 \\
		\hline
		Ours & \textbf{0.1773} & \textbf{0.1733} & \textbf{0.1585} & \textbf{0.1364} & \textbf{0.1331} & \textbf{0.1183} & \textbf{0.1018} & \textbf{0.1125} & \textbf{0.1197} & \textbf{0.1249} & \textbf{0.1466} & \textbf{0.2040} \\ 		
		\hline
		& JS & JS & JS & JS & JS & JS & JS & JS & JS & JS & JS & JS \\
		\hline
		CCZ & 0.6385 & 0.6787 & 0.4925 & 0.3108 & 0.2430 & 0.2041 & 0.3029 & 0.6410 & 0.7591 & 0.8353 & 0.8885 & 0.8998 \\
		\hline
		CNCS & 0.6617 & 0.7033 & 0.5056 & 0.3469 & 0.3467 & 0.3431 & 0.5013 & 0.7567 & 0.8164 & 0.8761 & 0.9019 & 0.9207 \\
		\hline
		HoL0MS & 0.7919 & 0.8927 & 0.8855 & 0.8921 & 0.8946 & 0.9278 & 0.9439 & 0.9447 & 0.9444 & 0.9412 & 0.9287 & 0.7967 \\
		\hline
		Ours & \textbf{0.8058} & \textbf{0.9156} & \textbf{0.9149} & \textbf{0.9232} & \textbf{0.9191} & \textbf{0.9455} & \textbf{0.9589} & \textbf{0.9558} & \textbf{0.9527} & \textbf{0.9490} & \textbf{0.9305} & \textbf{0.9290} \\ 		
		\hline
	\end{tabular}
	\label{table_3D}
\end{table}

\begin{figure}[H]
	\centerline{
		\begin{tabular}{c@{}c@{}c@{}c@{}c@{}c@{}c}
			CCZ \ &
			\includegraphics[width=0.8in]{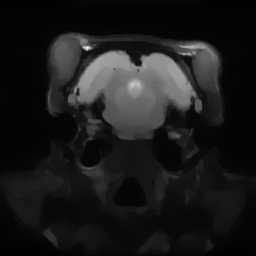} \ &
			\includegraphics[width=0.8in]{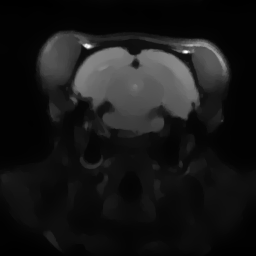} \ &		
			\includegraphics[width=0.8in]{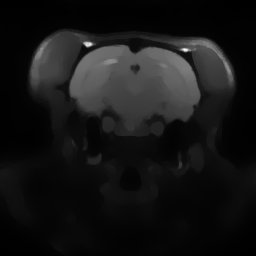} \ &
			\includegraphics[width=0.8in]{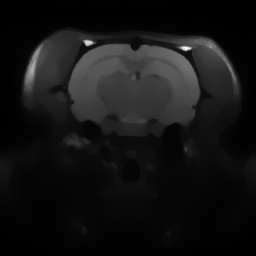} \ &
			\includegraphics[width=0.8in]{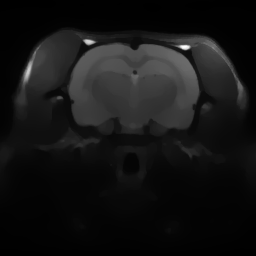} \ &
			\includegraphics[width=0.8in]{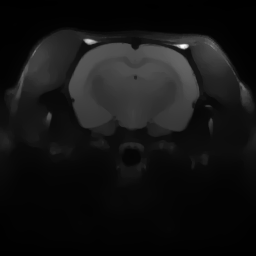} \\
			CNCS \ &
			\includegraphics[width=0.8in]{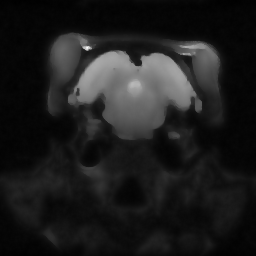} \ &
			\includegraphics[width=0.8in]{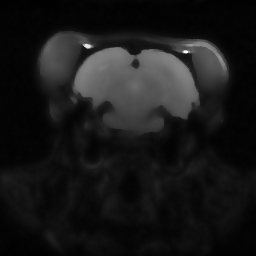} \ &		
			\includegraphics[width=0.8in]{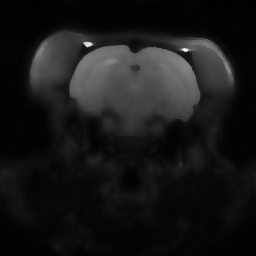} \ &
			\includegraphics[width=0.8in]{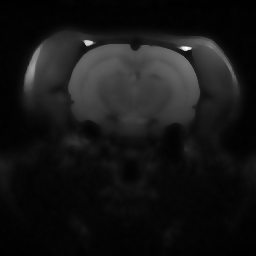} \ &
			\includegraphics[width=0.8in]{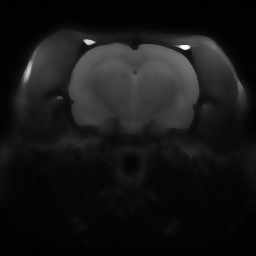} \ &
			\includegraphics[width=0.8in]{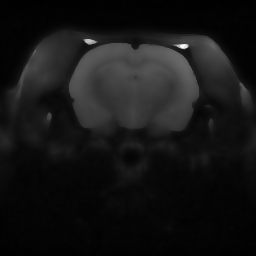} \\
			HoL0MS \ &
			\includegraphics[width=0.8in]{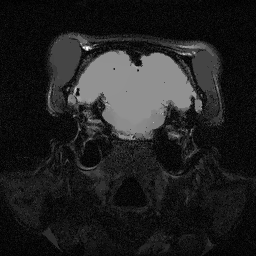} \ &
			\includegraphics[width=0.8in]{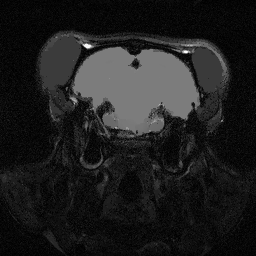} \ &		
			\includegraphics[width=0.8in]{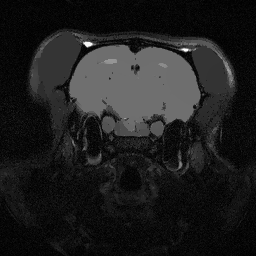} \ &
			\includegraphics[width=0.8in]{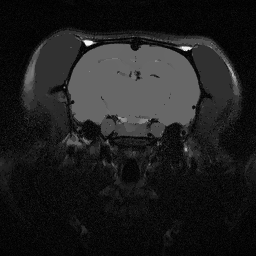} \ &
			\includegraphics[width=0.8in]{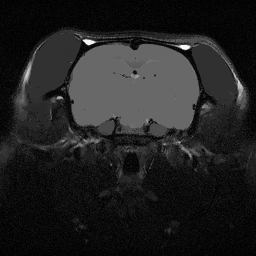} \ &
			\includegraphics[width=0.8in]{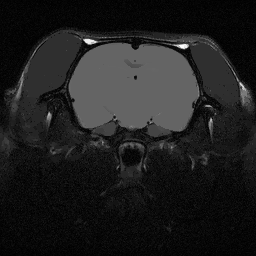} \\
			Ours \ &
			\includegraphics[width=0.8in]{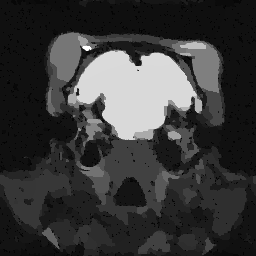} \ &
			\includegraphics[width=0.8in]{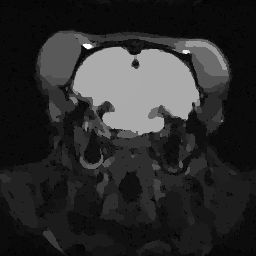} \ &		
			\includegraphics[width=0.8in]{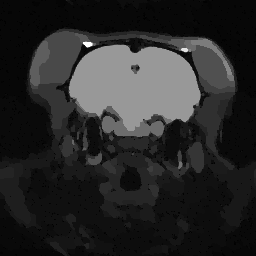} \ &
			\includegraphics[width=0.8in]{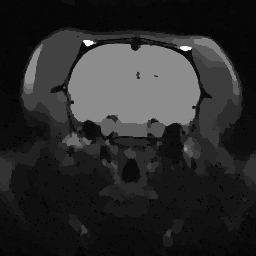} \ &
			\includegraphics[width=0.8in]{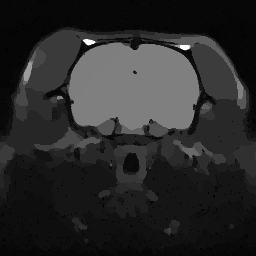} \ &
			\includegraphics[width=0.8in]{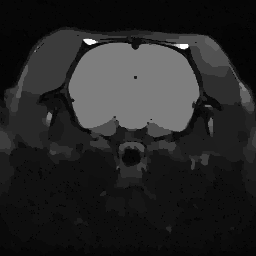} \\
			& slice1 & slice2 & slice3 & slice4 & slice5 & slice6 \\
			CCZ \ &
			\includegraphics[width=0.8in]{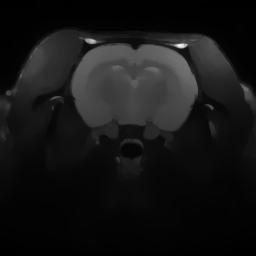} \ &
			\includegraphics[width=0.8in]{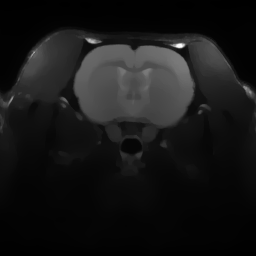} \ &		
			\includegraphics[width=0.8in]{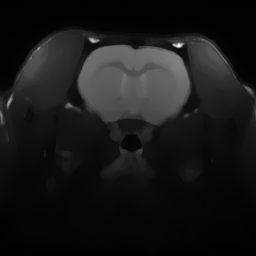} \ &
			\includegraphics[width=0.8in]{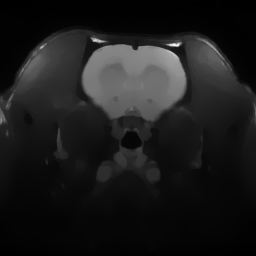} \ &
			\includegraphics[width=0.8in]{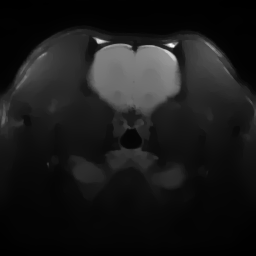} \ &
			\includegraphics[width=0.8in]{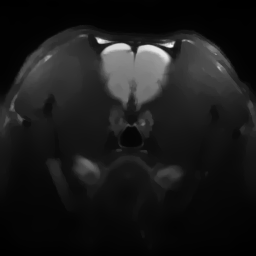} \\
			CNCS \ &
			\includegraphics[width=0.8in]{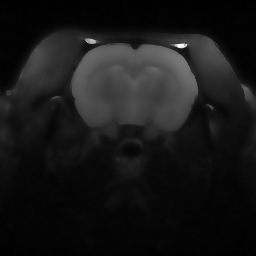} \ &
			\includegraphics[width=0.8in]{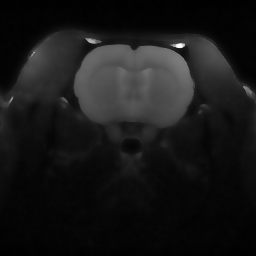} \ &		
			\includegraphics[width=0.8in]{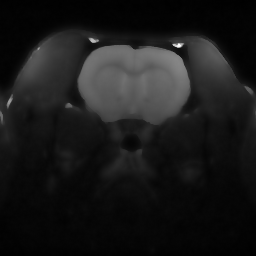} \ &
			\includegraphics[width=0.8in]{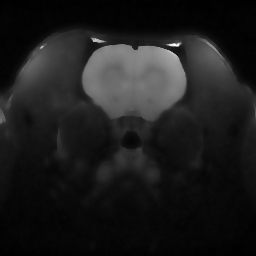} \ &
			\includegraphics[width=0.8in]{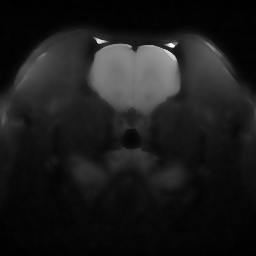} \ &
			\includegraphics[width=0.8in]{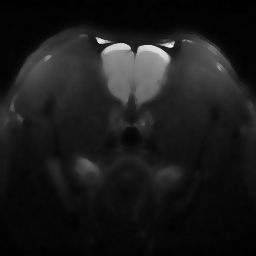} \\
			HoL0MS \ &
			\includegraphics[width=0.8in]{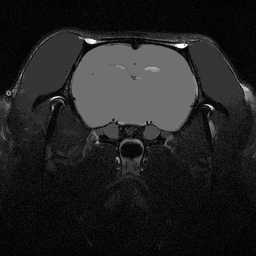} \ &
			\includegraphics[width=0.8in]{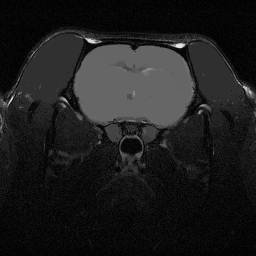} \ &		
			\includegraphics[width=0.8in]{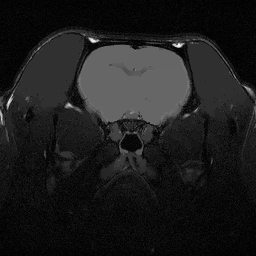} \ &
			\includegraphics[width=0.8in]{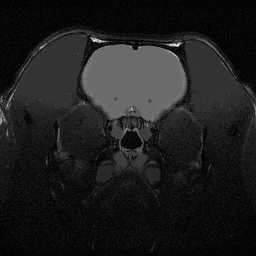} \ &
			\includegraphics[width=0.8in]{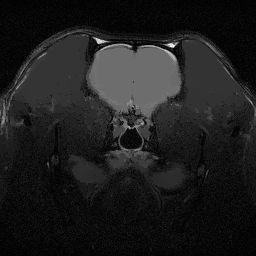} \ &
			\includegraphics[width=0.8in]{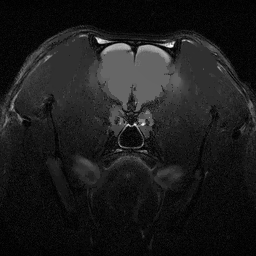} \\
			Ours \ &
			\includegraphics[width=0.8in]{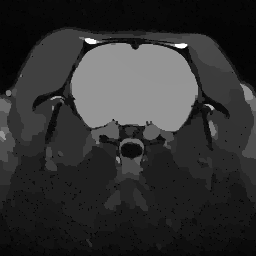} \ &
			\includegraphics[width=0.8in]{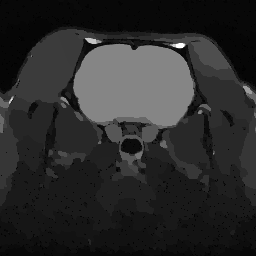} \ &		
			\includegraphics[width=0.8in]{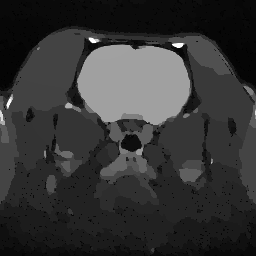} \ &
			\includegraphics[width=0.8in]{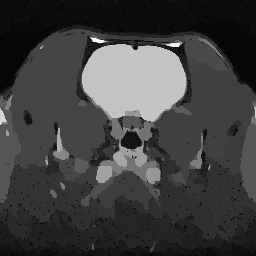} \ &
			\includegraphics[width=0.8in]{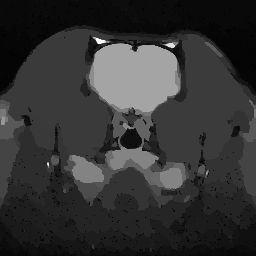} \ &
			\includegraphics[width=0.8in]{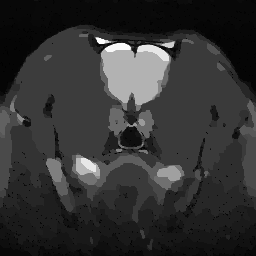} \\
			& slice7 & slice8 & slice9 & slice10 & slice11 & slice12 \\
		\end{tabular}
	}
	\caption{\small\sl Performance comparisons between different methods applied to a brain segmentation test for the brain MRI dataset in Figure \ref{fig3_testimages}: the inhomogeneity corrected images obtained in the first stage. Row 1-4: the inhomogeneity-corrected versions of slice 1-6 by CCZ, CNCS, HoL0MS and our method, respectively; Row 5-8: the inhomogeneity-corrected versions of slice 7-12 by CCZ, CNCS, HoL0MS and our method, respectively. The fine-tuned parameters for this dataset are: for CCZ, $(\lambda,\mu)=(20,5)$; for CNCS, $(\lambda,T)=(9,0.01)$; for HoL0MS, $(\alpha,\mu)=(0.01,0.01)$; and for ours, $(\alpha,\beta)=(0.4,8000)$. The CV values of the brain region in these corrected images are given in Table \ref{table_3D}.}
	\label{fig3_result1}
\end{figure}

\begin{figure}[H]
	\centerline{
		\begin{tabular}{c@{}c@{}c@{}c@{}c@{}c@{}c}
			CCZ \ &
			\includegraphics[width=0.8in]{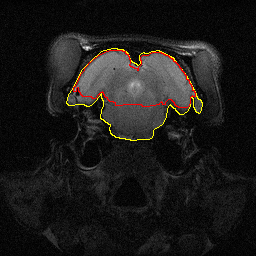} \ &
			\includegraphics[width=0.8in]{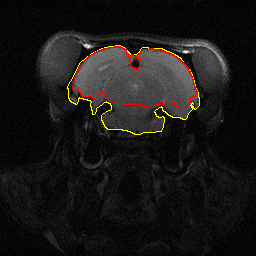} \ &		
			\includegraphics[width=0.8in]{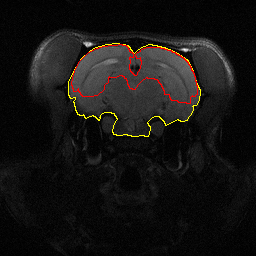} \ &
			\includegraphics[width=0.8in]{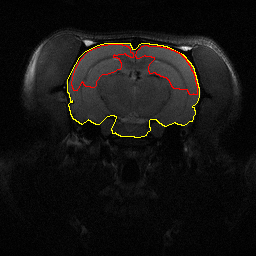} \ &
			\includegraphics[width=0.8in]{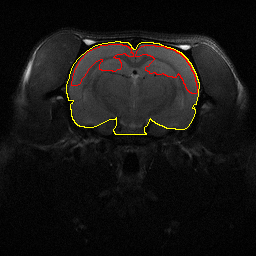} \ &
			\includegraphics[width=0.8in]{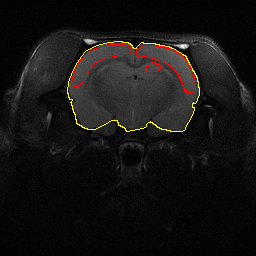} \\
			CNCS \ &
			\includegraphics[width=0.8in]{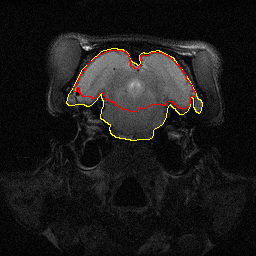} \ &
			\includegraphics[width=0.8in]{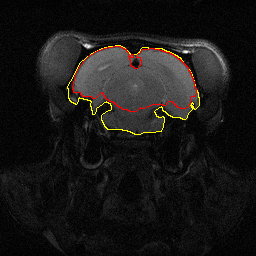} \ &		
			\includegraphics[width=0.8in]{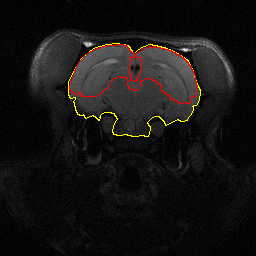} \ &
			\includegraphics[width=0.8in]{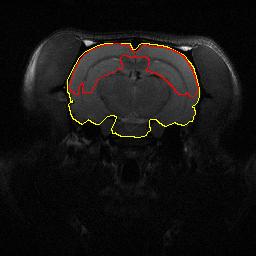} \ &
			\includegraphics[width=0.8in]{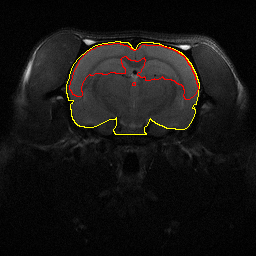} \ &
			\includegraphics[width=0.8in]{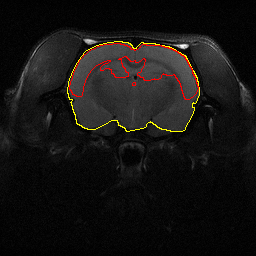} \\
			HoL0MS \ &
			\includegraphics[width=0.8in]{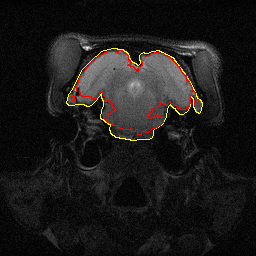} \ &
			\includegraphics[width=0.8in]{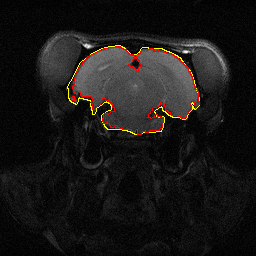} \ &		
			\includegraphics[width=0.8in]{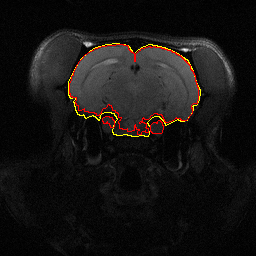} \ &
			\includegraphics[width=0.8in]{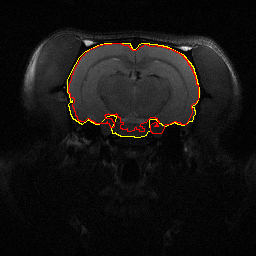} \ &
			\includegraphics[width=0.8in]{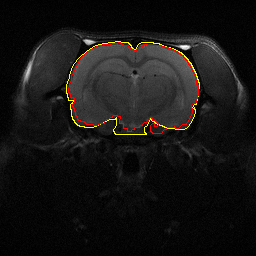} \ &
			\includegraphics[width=0.8in]{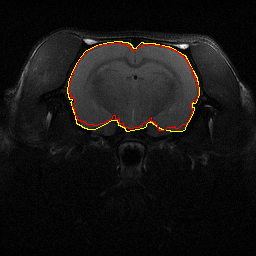} \\
			Ours \ &
			\includegraphics[width=0.8in]{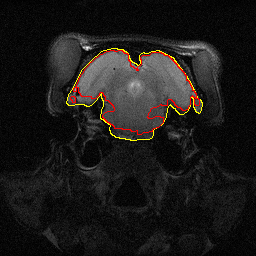} \ &
			\includegraphics[width=0.8in]{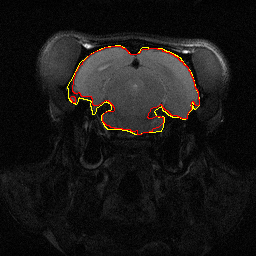} \ &		
			\includegraphics[width=0.8in]{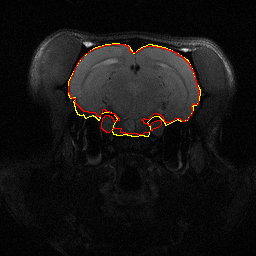} \ &
			\includegraphics[width=0.8in]{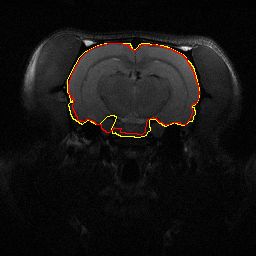} \ &
			\includegraphics[width=0.8in]{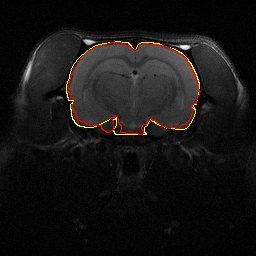} \ &
			\includegraphics[width=0.8in]{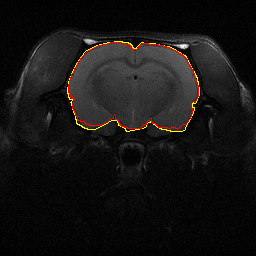} \\
			& slice1 & slice2 & slice3 & slice4 & slice5 & slice6 \\
			CCZ \ &
			\includegraphics[width=0.8in]{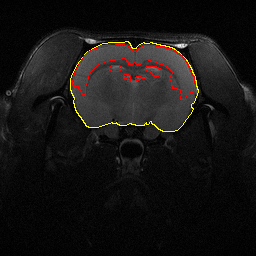} \ &
			\includegraphics[width=0.8in]{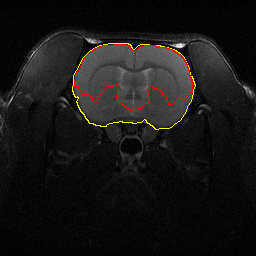} \ &		
			\includegraphics[width=0.8in]{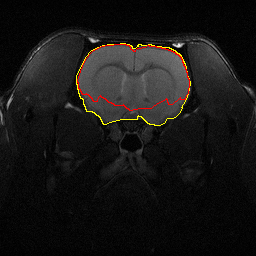} \ &
			\includegraphics[width=0.8in]{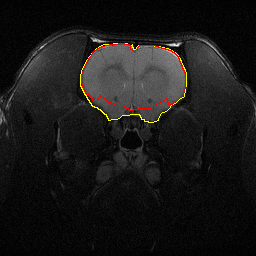} \ &
			\includegraphics[width=0.8in]{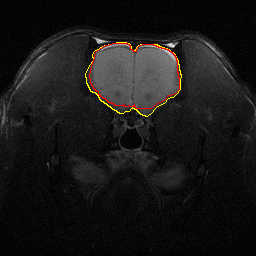} \ &
			\includegraphics[width=0.8in]{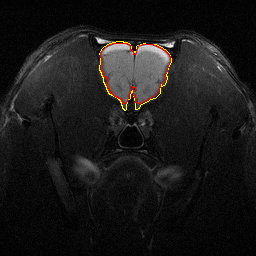} \\
			CNCS \ &
			\includegraphics[width=0.8in]{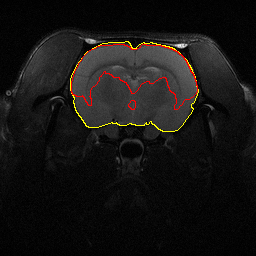} \ &
			\includegraphics[width=0.8in]{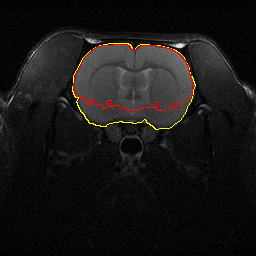} \ &		
			\includegraphics[width=0.8in]{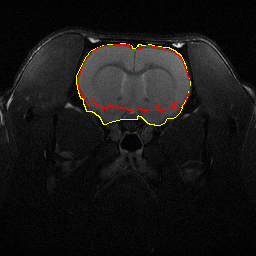} \ &
			\includegraphics[width=0.8in]{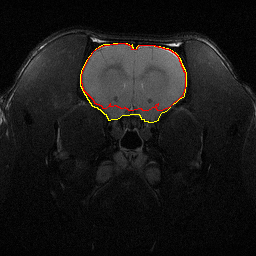} \ &
			\includegraphics[width=0.8in]{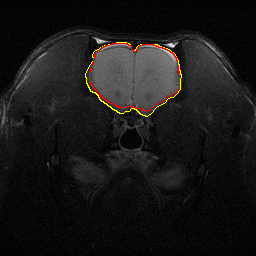} \ &
			\includegraphics[width=0.8in]{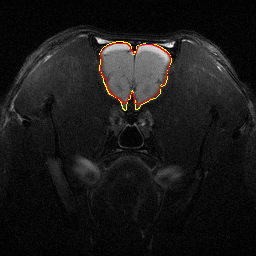} \\
			HoL0MS \ &
			\includegraphics[width=0.8in]{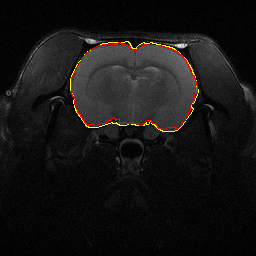} \ &
			\includegraphics[width=0.8in]{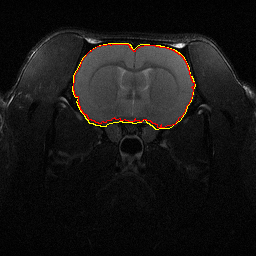} \ &		
			\includegraphics[width=0.8in]{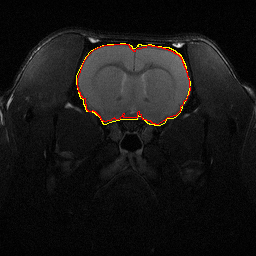} \ &
			\includegraphics[width=0.8in]{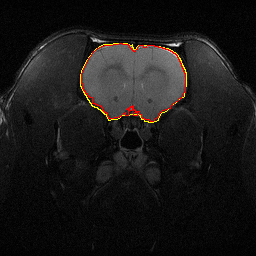} \ &
			\includegraphics[width=0.8in]{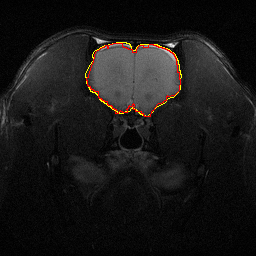} \ &
			\includegraphics[width=0.8in]{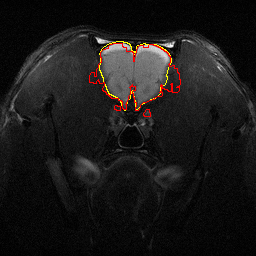} \\
			Ours \ &
			\includegraphics[width=0.8in]{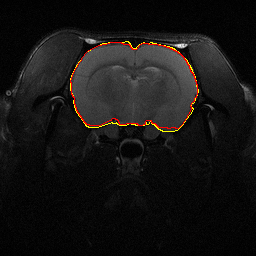} \ &
			\includegraphics[width=0.8in]{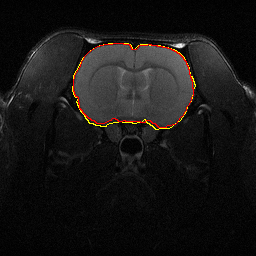} \ &		
			\includegraphics[width=0.8in]{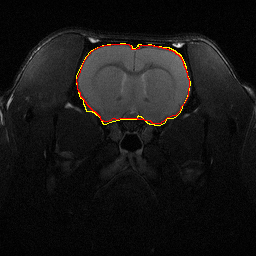} \ &
			\includegraphics[width=0.8in]{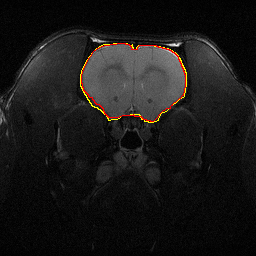} \ &
			\includegraphics[width=0.8in]{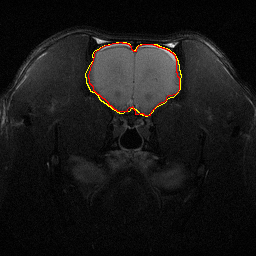} \ &
			\includegraphics[width=0.8in]{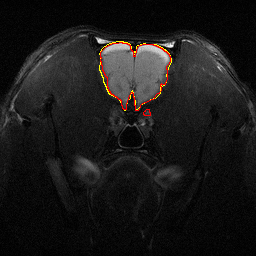} \\
			& slice7 & slice8 & slice9 & slice10 & slice11 & slice12 \\
		\end{tabular}
	}
	\caption{\small\sl Performance comparisons between different methods applied to a brain segmentation test for the brain MRI dataset in Figure \ref{fig3_testimages}: the brain segmentation results in the second stage. Row 1-4: the segmentation results of slice 1-6 by CCZ, CNCS, HoL0MS and our method, respectively; Row 5-8: the segmentation results of slice 7-12 by CCZ, CNCS, HoL0MS and our method, respectively. The yellow curves denote the ground truth location of the brain region, and the red curves denote the brain segmentation results by the compared methods. The corresponding JS values are given in Table \ref{table_3D}.}
	\label{fig3_result2}
\end{figure}

\subsection{Experimental summary}\label{experiment summary}
Let us summarize our experimental observations briefly. For images with weak intensity inhomogeneity, all the compared methods work quite well. For those with strong inhomogeneity, performance differences appear. Both CCZ and CNCS do not model the inhomogeneity explicitly, and thus naturally work poorly for the strongly inhomogeneous images. LIC is hard to be applicable to more than three phases segmentation in general images. L0MS can deal with strong inhomogeneity well, but isolated residual noise exists sometimes in its inhomogeneity-corrected images, which will influence its segmentation in the second stage. HoL0MS is a 3D high order variant of L0MS and is applicable to direct 3D segmentation. Both visual and quantitative comparisons on various datasets (with various levels of inhomogeneities, possible existence of noise, and multiphase segmentation tasks) show that our method always performs quite well. In most cases it gives better both inhomogeneity corrections and segmentations. Besides, our approach has proved convergence guarantee.

\section{Conclusion} \label{sec_conclusions}
We presented a new two-stage image segmentation method, where the key is to compute a piecewise constant approximate image suitable for the following thresholding operation. This is done by using a continuous but non-Lipschitz decomposition model. Motivated by a non-expansive property of the gradient support set for the non-Lipschitz term, we naturally extended previous iterative support shrinking algorithm to solve our decomposition model, with an ADMM inner solver. A lower bound theory for the iteration sequence has been given, showing that our algorithm can generate good approximate image components suitable for thresholding. The iterative sequence was also shown globally convergent to a stationary point of the original objective function in the decomposition model. Our method not only works very well for homogeneous images, but also can deal with multiphase segmentation for images with intensity inhomogeneity and noise. Its effectiveness and good convergence properties have been demonstrated by a series of numerical experiments, as well as visual and quantitative comparisons.

\section{Appendix}\label{Append_KL property}
\begin{definition}(Subdifferentials \cite{rockafellar2009variational})\label{Append_subdifferential}
	Let $\sigma : \mathbb{R}^d \rightarrow (-\infty,+\infty]$ be a proper and lower semicontinuous function. The domain of
	$\sigma$ is defined as $\dom \sigma=\{u \in \mathbb{R}^d: \sigma(u)<+\infty\}$. For a point $u \in \dom \sigma$,
	\begin{enumerate}
		\item the regular subdifferential of $\sigma$ at $u$ is defined as
		$$
		\widehat{\partial}\sigma(u) =\left\{w\in \mathbb{R}^d: \lim_{v \neq u}\inf_{v \rightarrow u}\frac{\sigma(v)-\sigma(u)-\langle w,v-u\rangle}{\|v-u\|}\geq 0 \right\};
		$$
		\item  the subdifferential of $\sigma$ at $u$ is defined as
		$$
		\partial \sigma(u)=\{w\in \mathbb{R}^d: \exists u^k\rightarrow u,\sigma(u^k)\rightarrow \sigma(u) \mbox{ and }
		w^k \in \widehat{\partial}\sigma(u^k)\rightarrow w \mbox{ as } k\rightarrow \infty \}.
		$$
	\end{enumerate}
\end{definition}

\begin{remark}\label{remark_2}
	From Definition \ref{Append_subdifferential}, it is clear that, if $\sigma$ is differentiable at $u$,
		then $\widehat{\partial}\sigma(u) = \partial \sigma(u) = \{\nabla \sigma(u)\}$. We also call a point $u \in \mathbb{R}^d$ a critical point, if $0 \in \partial \sigma(u)$.
\end{remark}

\begin{definition}(Kurdyka-{\L}ojasiewicz (KL) property \cite{attouch2010proximal})
	\begin{enumerate}
		\item The function $\sigma: \mathbb{R}^d \rightarrow (-\infty,+\infty]$ is said to have the Kurdyka-{\L}ojasiewicz property at $\overline{u} \in \dom \partial \sigma:=
		\{u \in \mathbb{R}^d: \partial \sigma(u) \neq \emptyset\}$ if there exist $\eta \in (0,+\infty]$, a neighborhood $U$ of $\overline{u}$, and a continuous concave
		function $\psi :[0,\eta) \rightarrow (0,+\infty]$ such that
		\begin{enumerate}[(i)]
			\item  $\psi(0)=0$;
			\item  $\psi$ is continuously differentiable on $(0,\eta)$;
			\item  for all $s \in (0,\eta)$, $\psi'(s)>0$;
			\item  for all $u \in U \cap \{v \in \mathbb{R}^d: \sigma(\overline{u}) < \sigma(v) < \sigma(\overline{u})+ \eta\}$,
			the Kurdyka-{\L}ojasiewicz (KL) inequality holds:
			$$
			\psi'(\sigma(u)-\sigma(\overline{u})) \dis (0,\partial \sigma(u)) \geq 1,
			$$
			where $\dis (0, \partial \sigma(u)):= \inf\{\|v\|: v \in \partial \sigma(u)\}$.
		\end{enumerate}
	\end{enumerate}
\end{definition}

A function $\sigma$ is called a KL function, if $\sigma$ satisfies the KL property at each point of $\dom \partial \sigma$. A rich class of KL functions of great interests are in a so-called o-minimal structure defined in \cite{van1996geometric}. The following definition is from \cite[Definition 4.1]{attouch2010proximal}.

\begin{definition}(o-minimal structure on $\mathbb{R}$)
	Let $\mathscr{O} = \{\mathscr{O}_n\}_{n \in \mathbb{N}}$ such that each $\mathscr{O}_n$ is a collection of subsets of $\mathbb{R}^n$. The family $\mathscr{O}$ is an
	o-minimal structure on $\mathbb{R}$, if it satisfies the following axioms:
	\begin{enumerate}[(i)]
		\item Each $\mathscr{O}_n$ is a boolean algebra. Namely $\emptyset \in \mathscr{O}_n$ and for each $A,B$ in $\mathscr{O}_n$, $A \cup B$, $A \cap B$, and
		$\mathbb{R}^n \setminus A$ belong to $\mathscr{O}_n$.
		\item For all $A$ in $\mathscr{O}_n$, $A \times \mathbb{R}$ and $\mathbb{R} \times A$ belong to $\mathscr{O}_{n+1}$.
		\item For all $A$ in $\mathscr{O}_{n+1}$, $\Pi (A):= \{(x_1,\ldots,x_n)\in \mathbb{R}^n: (x_1,\ldots,x_n,x_{n+1}) \in A\}$ belongs to $\mathscr{O}_{n}$.
		\item For all $i \neq j$ in $\{1,2,\ldots,n\}$, $\{(x_1,\ldots,x_n) \in \mathbb{R}^n: x_i = x_j\}$ belongs to $\mathscr{O}_n$.
		\item The set $\{(x_1,x_2) \in \mathbb{R}^2: x_1<x_2\}$ belongs to $\mathscr{O}_2$.
		\item The elements of $\mathscr{O}_1$ are exactly finite unions of intervals.
	\end{enumerate}
\end{definition}

Let $\mathscr{O}$ be an o-minimal structure on $\mathbb{R}$. We call a set $A \subseteq \mathbb{R}^n$ definable on $\mathscr{O}$ if $A \in \mathscr{O}_n$,
and a map $f: \mathbb{R}^n \rightarrow \mathbb{R}^m$ definable on $\mathscr{O}$ if its graph $\{(x,y) \in \mathbb{R}^n \times \mathbb{R}^m: y \in f(x)\}$ is definable on $\mathscr{O}$. A definable function is a special definable map. Some elementary properties of definable functions \cite{attouch2010proximal}\cite{Chao2018An} are as follows.
\begin{enumerate}[(i)]
	\item compositions of definable functions are definable;
	\item finite sums of definable functions are definable;
	\item indicator functions of definable sets are definable.	
\end{enumerate}

We have a very useful class of o-minimal structure, i.e., the log-exp structure \cite[Example 2.5]{van1996geometric}. By this, the following functions are all definable:
\begin{enumerate}[(1)]
	\item semi-algebraic functions \cite[Definition 5]{bolte2014proximal}, such as real polynomial functions, and $f: \mathbb{R} \rightarrow \mathbb{R}$ defined by
	$x \mapsto |x|$.
	\item $x^r : \mathbb{R} \rightarrow \mathbb{R}$ defined by
	$$
	a \mapsto \begin{cases}
	a^r, & a >0 \\
	0, & a \leq 0,
	\end{cases}
	$$
	where $r \in \mathbb{R}$.
\end{enumerate}

We know that any proper lower semicontinuous function definable on an o-minimal structure is a KL function; see \cite{bolte2007clarke} and \cite[Theorem 14]{attouch2010proximal}. For $F(u,v)$ in this paper, $\|f-u-v\|^2$, $\|D_i u\|$, $\|Hv\|^2$ and $\|v\|^2$ are all semi-algebraic functions. In addition, from examples (1)(2) and the elementary properties (\romannumeral1)(\romannumeral2) of definable functions, we know that $F(u,v)$ is definable. Thus $F(u,v)$ is a KL function.

\section*{Acknowledgments}
We greatly appreciate the authors of \cite{cai2013two}, \cite{duan2015l_}, \cite{chang2017new}, \cite{Chan2018Convex} and \cite{li2011level} for sharing their source codes. We are also very grateful to the anonymous reviewers for their valuable comments and suggestions. This work is supported in part by the Key Laboratory for Medical Data Analysis and Statistical Research of Tianjin~(C.~Wu,Y.~Xue), NSFTJ-17JCYBJC15800~(Y.~Xue), NSFC~11871035~(C.~Wu), NSFC~11531013~(C.~Wu) and Recruitment Program of Global Young Experts~(C.~Wu).

\bibliographystyle{unsrt}

\end{document}